\documentclass[reqno]{amsart}
\usepackage{amsmath, amsthm, amssymb, amstext}

\usepackage[left=3cm,right=3cm,top=3cm,bottom=3cm]{geometry}
\usepackage{hyperref,xcolor}
\hypersetup{pdfborder={0 0 0},colorlinks}
\usepackage{enumitem}
\allowdisplaybreaks
\usepackage{esint}

\newtheorem{theorem}{Theorem}
\newtheorem{remark}[theorem]{Remark}
\newtheorem{lemma}[theorem]{Lemma}
\newtheorem{proposition}[theorem]{Proposition}

\newtheorem{definition}[theorem]{Definition}

\newcommand{\Lp}[1]{L^{#1}(\Omega)}
\newcommand{\Wpzero}[1]{W^{1,#1}_0(\Omega)}

\newcommand{\N}{\mathbb{N}}
\newcommand{\R}{\mathbb{R}}
\renewcommand{\l}{\left}
\renewcommand{\r}{\right}

\def\abs#1{\left|{#1}\right|}

\numberwithin{theorem}{section}
\numberwithin{equation}{section}

\title[Logarithmic double phase problems with generalized critical growth]{Logarithmic double phase problems with generalized critical growth}

\author[R. Arora]{Rakesh Arora}
\address[R. Arora]{ Department of Mathematical Sciences, Indian Institute of Technology Varanasi (IIT-BHU), Uttar Pradesh 221005, India}
\email{rakesh.mat@iitbhu.ac.in, arora.npde@gmail.com}

\author[\'{A}. Crespo-Blanco]{\'{A}ngel Crespo-Blanco}
\address[\'{A}. Crespo-Blanco]{Technische Universit\"{a}t Berlin, Institut f\"{u}r Mathematik, Stra\ss e des 17.\,Juni 136, 10623 Berlin, Germany}
\email{crespo@math.tu-berlin.de}

\author[P. Winkert]{Patrick Winkert}
\address[P. Winkert]{Technische Universit\"{a}t Berlin, Institut f\"{u}r Mathematik, Stra\ss e des 17.\,Juni 136, 10623 Berlin, Germany}
\email{winkert@math.tu-berlin.de}

\subjclass{35A01, 35J20, 35J25, 35J62, 35J92, 35Q74}
\keywords{Concentration compactness principle, critical growth, genus theory, logarithmic double phase operator, logarithmic Musielak-Orlicz spaces, Sobolev embeddings}

\begin{document}

\begin{abstract}
	In this paper we study logarithmic double phase problems with variable exponents involving nonlinearities that have generalized critical growth. We first prove new continuous and compact embedding results in order to guarantee the well-definedness by studying the Sobolev conjugate function of our generalized $N$-function. In the second part we prove the concentration compactness principle for Musielak-Orlicz Sobolev spaces having logarithmic double phase modular function structure. Based on this we are going to show multiplicity results for the problem under consideration for superlinear and sublinear growth, respectively.
\end{abstract}

\maketitle

\section{Introduction}

Recently, the authors \cite{Arora-Crespo-Blanco-Winkert-2023} introduced a new class of logarithmic double phase operators given by
\begin{align}\label{log-double-phase-operator}
	\operatorname{div}\left( \abs{ \nabla u }^{p(x) - 2} \nabla u + \mu(x) \left[ \log ( e + \abs{\nabla u} ) + \frac{ \abs{\nabla u} }{q(x) (e + \abs{\nabla u})} \right] \abs{ \nabla u }^{q(x) - 2} \nabla u  \right),
\end{align}
whose energy functional is given by
\begin{align}\label{log-functional}
	u \mapsto \int_\Omega \left(  \frac{ \abs{\nabla u}^{p(x)} }{p(x)} + \mu(x) \frac{ \abs{\nabla u}^{q(x)} }{q(x)} \log (e + \abs{\nabla u}) \right)  \mathrm{d}x,
\end{align}
where $\Omega \subseteq \R^N$, $N \geq 2$, is a bounded domain with Lipschitz boundary $\partial \Omega$, $e$ stands for Euler's number, $p,q \in C(\overline{\Omega})$ with $1<p(x) \leq q(x)$ for all $x \in \overline{\Omega}$ and $\mu \in \Lp{1}$. The related Musielak-Orlicz Sobolev space is generated by the function
\begin{align*}
	\mathcal{H}_{\log}(x,t)=t^{p(x)} + \mu(x) t^{q(x)} \log (e + t)\quad\text{for all }(x,t)\in \overline{\Omega} \times [0,\infty).
\end{align*}
In \cite{Arora-Crespo-Blanco-Winkert-2023} it is shown that the spaces $W^{1, \mathcal{H}_{\log}}(\Omega)$ and $W^{1, \mathcal{H}_{\log}}_0(\Omega)$ are separable, reflexive Banach spaces and the operator \eqref{log-double-phase-operator} turns out to be bounded, continuous, strictly monotone and of type \textnormal{(S$_+$)}. Note that similar functionals as in \eqref{log-functional} for constant exponents have been studied by Baroni--Colombo--Mingione \cite{Baroni-Colombo-Mingione-2016} in order to prove the local H\"older continuity of the gradient of local minimizers of the functional
\begin{align} \label{log-functional-Mingione1}
	w \mapsto \int_\Omega \left[  \abs{D w}^p  + a(x) \abs{D w}^p \log (e + \abs{D w}) \right]\,\mathrm{d} x
\end{align}
whenever $1 < p < \infty$ and $0 \leq a(\cdot) \in C^{0,\alpha} (\overline{\Omega})$. We point out that the functionals in \eqref{log-functional} and \eqref{log-functional-Mingione1} coincide in case  $p=q$ are constant. In this direction we also mention the paper by De Filippis--Mingione \cite{DeFilippis-Mingione-2023} who proved the local H\"{o}lder continuity of the gradients of local minimizers of the functional
\begin{align}\label{log-functional-Mingione2}
	w \mapsto \int_\Omega \big[|D w|\log(1+|D w|)+a(x)|D w|^q\big]\,\mathrm{d} x
\end{align}
provided $0 \leq a(\cdot)\in C^{0,\alpha}(\overline{\Omega})$ and $1<q<1+\frac{\alpha}{n}$. Note that  \eqref{log-functional-Mingione2} originates from functionals with nearly linear growth given by
\begin{align}\label{log-functional-Mingione3}
	w \mapsto \int_\Omega |D w|\log(1+|D w|)\,\mathrm{d} x,
\end{align}
see the works by Fuchs--Mingione \cite{Fuchs-Mingione-2000} and Marcellini--Papi \cite{Marcellini-Papi-2006}. Functionals as given in \eqref{log-functional-Mingione3} appear, for example, in the theory of plasticity with logarithmic hardening, see, Seregin--Frehse \cite{Seregin-Frehse-1999} and Fuchs--Seregin \cite{Fuchs-Seregin-2000}.

In this paper we are going to deepen the study of the operator \eqref{log-double-phase-operator} and consider generalized $N$-functions given by
\begin{align}\label{N-function}
	\mathcal{S}(x,t)=a(x) t^{p(x)}+b(x) t^{q(x)}\log^{s(x)}(1 + t) \quad \text{for } (x, t) \in \Omega \times (0, \infty),
\end{align}
where $p, q \in C(\overline{\Omega})$ with $1 < p(x), q(x)< N$ for a.a.\,$x \in \overline{\Omega}$, $s \in L^\infty(\Omega)$ such that $q(x)+ s(x) \geq r> 1$ and $0 \leq a, b\in L^1(\Omega)$ such that $a(x) + b(x) \geq d >0$ for all $x \in \overline{\Omega}$, see \eqref{main:assump} for the precise assumptions. First we are interested in continuous and compact embeddings from $W^{1,\mathcal{S}}(\Omega)$ into suitable Musielak-Orlicz Lebesgue spaces. In particular, we prove that $W^{1,\mathcal{S}}(\Omega)$ is continuously embedded into $L^{\mathcal{S}^{\ast}}(\Omega)$, where $\mathcal{S}^\ast\colon \overline{\Omega} \times [0, \infty)\to  [0, \infty)$ is defined as
\begin{align*}
	\mathcal{S}^\ast(x,t):= \left( \left(a(x)\right)^\frac{1}{p(x)} t\right)^{p^\ast(x)}  + \left(\left(b(x) \log^{s(x)}(1+t)\right)^\frac{1}{q(x)} t  \right)^{q^\ast(x)},
\end{align*}
see Proposition \ref{imp:embedding}. Such type of embedding is sharp in the sense that, for each fixed $x$, it coincides with the sharp Sobolev conjugate in classical Orlicz spaces. Here we use ideas from the papers by Arora--Crespo-Blanco--Winkert \cite{Arora-Crespo-Blanco-Winkert-2023}, Cianchi--Diening \cite{Cianchi-Diening-2024}, Colasuonno--Squassina \cite{Colasuonno-Squassina-2016}, Crespo-Blanco--Gasi\'{n}ski--Harjulehto--Winkert, \cite{Crespo-Blanco-Gasinski-Harjulehto-Winkert-2022}, Fan \cite{Fan-2012}, and Ho--Winkert \cite{Ho-Winkert-2023}. Furthermore, we point out that the new logarithmic double phase operator generated by $\mathrm{S}$ given in \eqref{N-function} is defined by
\begin{equation}\label{main:operator}
	\begin{aligned}
		&\operatorname{div}\bigg (a(x) |\nabla u|^{p(x)-2} \nabla u\\
		&\qquad\qquad+ b(x) |\nabla u|^{q(x)-2} \log^{s(x)-1 }(1+ |\nabla u|) \left(\log(1+ |\nabla u|) + \frac{s(x)}{q(x)} \frac{|\nabla u|}{1+ |\nabla u|} \right) \nabla u\bigg).
	\end{aligned}
\end{equation}
The study of problems involving the new logarithmic double phase operator \eqref{main:operator} (see problem \eqref{main:prob} below) has two different features in contrast to the known works. The first different feature is the degeneracy or singularity of the operator at $\nabla u=0$ which is of polynomial type perturbed with a logarithmic term, which again has degeneracy or singularity at $\nabla u=0$ depending upon the exponent $s(\cdot)$. The second different feature is the exponent $s(\cdot)$ itself on the logarithmic term which may change sign depending upon the space variable $x \in \Omega$. This sign-changing nature of the exponent $s(\cdot)$ creates several challenges and technical difficulties in our analysis. In particular, it leads to losing the convexity of the corresponding generalized $N$-function \eqref{N-function}, which is essential in setting up the corresponding Banach spaces. Another difficulty arises in finding the suitable structure of the generalized (critical) $N$-function for continuous and compact embeddings which takes the form of non-trivial estimates. To handle these issues, we introduce a balance condition between the sign-changing exponent $s(\cdot)$ and $q(\cdot)$, which together preserves the convexity of the generalized $N$-function \eqref{N-function} and several other properties of the new logarithmic double phase operator \eqref{main:operator}. This makes the study of the new logarithmic double phase operator \eqref{main:operator} more challenging and interesting.
In general, there are only few works for the double phase operator with a logarithmic perturbation, see the papers by the authors \cite{Arora-Crespo-Blanco-Winkert-2023}, Lu--Vetro--Zeng \cite{Lu-Vetro-Zeng-2024}, Vetro--Winkert \cite{Vetro-Winkert-2024} and Vetro--Zeng \cite{Vetro-Zeng-2024}. In particular, in \cite{Lu-Vetro-Zeng-2024} the authors introduced the operator
\begin{align}\label{div}
	u\mapsto\Delta_{\mathcal {H}_{L}} u= \operatorname{div} \left(\frac{\mathcal {H}'_{L}(x,|\nabla u|)}{|\nabla u|}\nabla u \right), \quad u\in W^{1,\mathcal {H}_{L}}(\Omega)
\end{align}
where $\mathcal {H}_{L} \colon \Omega \times \times [0, \infty)\to  [0, \infty)$ is given by
\begin{align*}
	\mathcal {H}_{L}(x,t) =[t^{p(x)}+\mu(x)t^{q(x)}]\log(e+\alpha t)
\end{align*}
with $\alpha \geq 0$. Note that \eqref{div} is a different operator than the one in this paper and also the one in \eqref{log-double-phase-operator}. Moreover, the work in \cite{Lu-Vetro-Zeng-2024} can be seen as the extension of Vetro--Zeng \cite{Vetro-Zeng-2024} from the constant exponent case to the variable one while the paper by Vetro--Winkert \cite{Vetro-Winkert-2024} uses the same operator as in \eqref{log-double-phase-operator}.

In the second part of this paper we prove a concentration compactness principle of Lions type for the Musielak-Orlicz Sobolev space $W_0^{1, \mathcal{S}}(\Omega)$ having logarithmic double phase modular function structure. Such result is of independent interest and can be also used for other types of logarithmic double phase problems. Our result extends the works by Fern\'{a}ndez Bonder--Silva \cite{Fernandez-Bonder-Silva-2010, Fernandez-Bonder-Silva-2024} for variable exponent spaces and Orlicz spaces, respectively and Ha--Ho \cite{Ha-Ho-2024} (see also \cite{Ha-Ho-2025} by the same authors for the entire space) for Musielak-Orlicz Sobolev spaces having double phase modular function structure without logarithmic perturbation who proved a concentration compactness principle by extending the classical work of Lions \cite{Lions-1985}. Other versions of the concentration compactness principle have been proved, among others, by Fern\'{a}ndez Bonder--Saintier--Silva \cite{Fernandez-Bonder-Saintier-Silva-2018}, Fu--Zhang \cite{Fu-Zhang-2010}, Ho--Kim \cite{Ho-Kim-2021} and Palatucci--Pisante \cite{Palatucci-Pisante-2014}. In contrast to the cases just mentioned, for the logarithmic double phase operator, the same ideas cannot be extended easily due to the non-uniform limit (with respect to the space variable) of the Matuszewska index of the generalized $N$-function $S$ defined in \eqref{N-function} (compared to Fern\'{a}ndez Bonder--Silva \cite{Fernandez-Bonder-Silva-2024}) and the lack of appropriate scaling and homogeneity because of the logarithmic perturbation (compared to Ho--Kim \cite{Ho-Kim-2021}). In order to handle this issue, we introduce suitable sub-homogeneous functions perturbed with logarithmic growth, which plays a crucial role in our analysis. This approach can also be extended to prove the concentration compactness principle for a larger class of Musielak-Orlicz Sobolev spaces. In addition, we also prove a Br\'{e}zis-Lieb lemma and a reverse H\"older type inequality.

In the last part of the paper, based on the continuous embedding  $W^{1,\mathcal{S}}(\Omega) \hookrightarrow L^{\mathcal{S}^{\ast}}(\Omega)$ and the concentration compactness principle both developed before, we are going to study quasilinear elliptic equations driven by the logarithmic double phase operator and with right-hand sides having the new critical growth perturbed with superlinear and sublinear growth nonlinearities
\begin{equation}\label{main:prob}
	\begin{aligned}
		- \operatorname{div}\left(\mathbb{M}(x,\nabla u)\right) & = \Lambda \mathbb{M}^\ast(x,u) + \lambda \mathbb{M}_\star(x,u)\quad &  & \text{in } \Omega,\\
		u& = 0 &  & \text{on } \partial\Omega,
	\end{aligned}
\end{equation}
where $\Omega \subset \mathbb{R}^N$, $N\geq 2$, is a bounded domain with Lipschitz boundary $\partial\Omega$, $\lambda, \Lambda >0$ are parameters to be specified and $\mathbb{M}\colon \Omega \times \mathbb{R}^N \to \mathbb{R}$ is given by
\begin{align*}
	\mathbb{M}(x, \xi) = \partial_{\xi} (\mathcal{M}(x,\xi)), \quad \mathcal{M}(x, \xi):= \left( \frac{a(x)}{p(x)} |\xi|^{p(x)}  + \frac{b(x)}{q(x)} |\xi|^{q(x)} \log^{s(x)}(1+|\xi|)\right)
\end{align*}
while $\mathbb{M}^\ast, \mathbb{M}_\ast$ are Carath\'eodory functions on $\Omega \times \mathbb{R}$ defined by
\begin{align*}
	\mathbb{M}^\ast(x,t)= \partial_{t} (\mathcal{M}^\ast(x,t)), \quad \mathbb{M}_\star(x,t) = \partial_{t} (\mathcal{M}_\star(x,t)),
\end{align*}
where
\begin{align*}
	\mathcal{M}^\ast(x,t) &= \frac{1}{p^\ast(x)} \left(\left(a(x)\right)^\frac{1}{p(x)} |t|\right)^{p^\ast(x)} + \frac{1}{q^\ast(x)} \left(\left(b(x) \log^{s(x)}(1+|t|)\right)^\frac{1}{q(x)} |t|\right)^{q^\ast(x)},\\
	\mathcal{M}_\star(x, t)&= \frac{1}{p_\star(x)} \left(\left(a(x)\right)^\frac{1}{p(x)} |t|\right)^{p_\star(x)} + \frac{1}{q_\star(x)} \left(\left(b(x) \log^{s_\star(x)}(1+|t|)\right)^\frac{1}{q(x)} |t|\right)^{q_\star(x)},
\end{align*}
with $p_\star(\cdot)$, $q_\star(\cdot)$ and $p^\ast(\cdot)$, $q^\ast(\cdot)$ being the subcritical and critical Sobolev variable exponents of $p(\cdot)$ and $q(\cdot)$, respectively, that is, $p_\star (\cdot) < p^\ast(\cdot)$ and $q_\star (\cdot) < q^\ast(\cdot)$, see \eqref{main:assump-1-2} below. We consider the cases of superlinear and sublinear growth separately and get multiplicity results in the Theorems \ref{theo-1} and \ref{theo-2}. In the case of superlinear growth, by extending the ideas of Komiya--Kajikiya \cite[Theorem 2.2]{Komiya-Kajikiya-2016} via genus theory and the deformation lemma, for each $n \in \mathbb{N}$, we show the existence of at least $n$-pairs of solutions for $\Lambda \in (0, \Lambda_n)$. In case of the sublinear growth, using Krasnosel'skii genus theory (see Krasnosel'skii \cite{Krasnoselskii-1964})  along with appropriate truncation technique following ideas by Garc\'{\i}a Azorero--Peral Alonso \cite{Garcia-Azorero-Peral-Alonso-1991} and Farkas--Fiscella--Winkert \cite{Farkas-Fiscella-Winkert-2022}, the existence of infinitely many weak solutions with negative energy sign has been shown. Note that the appearance of the logarithmic term in our operator makes the treatment much more complicated than in the known works.

In general, there are only few existence results for critical double phase problems without logarithmic perturbation in case of constant or variable exponents, see the papers by Arora--Fiscella--Mukherjee--Winkert \cite{Arora-Fiscella-Mukherjee-Winkert-2022}, Farkas--Fiscella--Winkert \cite{Farkas-Fiscella-Winkert-2022}, Ho--Kim--Zhang \cite{Ho-Kim-Zhang-2024}, Liu--Papageorgiou \cite{Liu-Papageorgiou-2021} and Papageorgiou--Vetro--Winkert
\cite{Papageorgiou-Vetro-Winkert-2024,Papageorgiou-Vetro-Winkert-2023}. All these works have just terms like $|u|^{p^*-2}u$ or $|u|^{p^*(x)-2}u$ on the right-hand side of their problems without any term likes $\mu(x)|u|^{q^*-2}u$. In the recent paper by Colasuonno--Perera \cite{Colasuonno-Perera-2025} the authors study double phase problems in the local and in the nonlocal Kirchhoff case by allowing a growth on the right-hand side given by
\begin{align}\label{general-growth-1}
	|u|^{p^*-2}u+b(x)|u|^{q^*-2}
\end{align}
with a suitable weight function $b(\cdot)$. Moreover, Ha--Ho \cite{Ha-Ho-2024} studied Kirchhoff double phase problems with variable exponents via a new concentration compactness principle (we already mentioned this fact above) and with a growth like
\begin{align}\label{general-growth-2}
	|u|^{p^*(x)-2}u+a(x)^{\frac{q^*(x)}{q(x)}}|u|^{q^*(x)-2}u,
\end{align}
see also a paper by the same authors \cite{Ha-Ho-2025} for the case in $\mathbb{R}^N$. For the use of a growth as in \eqref{general-growth-2}, the authors in \cite{Ha-Ho-2024} have used general embedding results for variable exponent double phase problem proved by Ho--Winkert \cite{Ho-Winkert-2023}, see also the work by Cianchi--Diening \cite{Cianchi-Diening-2024} for a general Sobolev embedding theorem in Musielak-Orlicz Sobolev spaces. In contrast to \eqref{general-growth-1} and \eqref{general-growth-2} we allow a growth of the form
\begin{align*}
	\frac{1}{p^\ast(x)} \left(\left(a(x)\right)^\frac{1}{p(x)} |t|\right)^{p^\ast(x)} + \frac{1}{q^\ast(x)} \left(\left(b(x) \log^{s(x)}(1+|t|)\right)^\frac{1}{q(x)} |t|\right)^{q^\ast(x)}.
\end{align*}
The study of the above generalized critical growth nonlinearity is motivated by embeddings of the corresponding Musielak-Orlicz-Sobolev space $W_0^{1, \mathcal{S}}(\Omega)$ and the above concentration compactness principle. Such types of nonlinearities have not yet been investigated in the literature. In this direction, the current work presents new multiplicity results for quasilinear elliptic problems involving the logarithmic double phase operator given in \eqref{main:operator} and generalized critical growth with logarithmic perturbation.

The paper is organized as follows. In Section \ref{Section-2} we present the general theory about Musielak-Orlicz spaces and introduce our special $N$-function $\mathcal{S}$ including the proofs of some properties of it. Section \ref{Section-3} is devoted to the new continuous and compact embedding results (see Proposition \ref{imp:embedding}) while in Section \ref{Section-4} we are going to develop and prove a concentration compactness principle adapted to logarithmic double phase structures as presented before in this paper, see Theorem \ref{concentration:compactness}. In Section \ref{Section-5} we discuss the properties of the energy functional and the logarithmic double phase operator in \eqref{main:prob}. Finally, in Section \ref{Section-6} we prove our main multiplicity results for problem \eqref{main:prob} in the case of superlinear and sublinear growth, respectively, Theorems \ref{theo-1} and \ref{theo-2}.

\section{The function spaces}\label{Section-2}

We begin with a brief description of Lebesgue and Sobolev spaces with variable exponents and introduce then the necessary definitions for introducing the required Musielak-Orlicz Sobolev spaces. We refer to the monographs of Chlebicka--Gwiazda--\'{S}wierczewska-Gwiazda--Wr\'{o}blewska-Kami\'{n}ska \cite{Chlebicka-Gwiazda-Swierczewska-Gwiazda-Wroblewska-Kaminska-2021}, Diening--Harjulehto--H\"{a}st\"{o}--R$\mathring{\text{u}}$\v{z}i\v{c}ka \cite{Diening-Harjulehto-Hasto-Ruzicka-2011}, Harjulehto--H\"{a}st\"{o} \cite{Harjulehto-Hasto-2019}, Musielak \cite{Musielak-1983}, Papageorgiou--Winkert \cite{Papageorgiou-Winkert-2024} and the paper Fan--Zhao \cite{Fan-Zhao-2001}.

Given a bounded domain $\Omega\subseteq \R^N$ with $N\geq 2$ and Lipschitz boundary $\partial\Omega$, we denote by $L^r(\Omega)$ the usual Lebesgue space for $1 \leq r\leq \infty$ equipped with the norm $\|\cdot\|_r$. Moreover, $W^{1,r}(\Omega)$ and $W^{1,r}_0(\Omega)$ stand for the related Sobolev spaces with the norm $\|\cdot\|_{1,r}=\|\cdot\|_r +\|\nabla \cdot\|_r$ and if $1 \leq r < \infty$, then $W^{1,r}_0(\Omega)$ can be endowed with the equivalent norm $\|\nabla \cdot\|_r$.

For any function $f\colon \Omega \to \R$ we write
\begin{align*}
	\lceil f \rceil(x):= \max\{f(x), 0\}
	\quad \text{and} \quad
	\lfloor f \rfloor(x):= \min\{f(x), 0\}.
\end{align*}
For $r \in C(\overline{\Omega})$, we set $r^- = \min_{x \in \overline{\Omega}} r(x)$ and $r^+ = \max_{x \in \overline{\Omega}} r(x)$ and we define
\begin{align*}
	C_+ (\Omega) = \{ r \in C(\overline{\Omega})\colon  1 < r^- \}.
\end{align*}
Note that for any bounded function $g\colon \Omega \to \mathbb{R}$ we use the same notation, that is, we set
\begin{align*}
	g^+ := \max_{x \in \Omega} g(x)
	\quad \text{and} \quad
	g^- := \min_{x \in \Omega} g(x).
\end{align*}

Now, for $r \in C_+ (\Omega)$ and $M(\Omega)$ being the set of all measurable functions $u\colon\Omega\to\R$, we denote by $\Lp{r(\cdot)}$ the Lebesgue space with variable exponent given by
\begin{align*}
	\Lp{r(\cdot)} = \left\{ u \in M(\Omega) \colon \varrho_{r(\cdot)} (u) < \infty \right\},
\end{align*}
equipped with the norm
\begin{align*}
	\|u\|_{r(\cdot)} = \inf \left\{ \lambda > 0 \colon \varrho_{r(\cdot)} \left( \frac{u}{\lambda} \right)  \leq 1\right\},
\end{align*}
whereby the modular associated with $r$ is
\begin{align*}
	\varrho_{r(\cdot)} (u) = \int_\Omega |u|^{r(x)}\,\mathrm{d}x.
\end{align*}
It is well known that $\Lp{r(\cdot)}$ is a separable and reflexive Banach space and its norm is uniformly convex. Also, it holds $\left[ \Lp{r(\cdot)} \right] ^*=\Lp{r'(\cdot)}$, where $r' \in C_+(\overline{\Omega})$ is the conjugate variable exponent of $r$ given by $r'(x) = r(x) / [r(x) - 1]$ for all $x \in \overline{\Omega}$. Furthermore, we also have a H\"older type inequality of the form
\begin{align*}
	\int_\Omega |uv| \,\mathrm{d}x \leq \left[\frac{1}{r^-}+\frac{1}{r'^-}\right] \|u\|_{r(\cdot)}\|v\|_{r'(\cdot)} \leq 2 \|u\|_{r(\cdot)}\|v\|_{r'(\cdot)} \quad \text{for all } u,v\in \Lp{r(\cdot)}.
\end{align*}
If $r_1, r_2\in C_+(\overline{\Omega})$ and $r_1(x) \leq r_2(x)$ for all $x\in \overline{\Omega}$, we have the continuous embedding $\Lp{r_2(\cdot)} \hookrightarrow \Lp{r_1(\cdot)}$.

The next proposition shows the relation between the norm and its modular, see Fan-Zhao \cite[Theorems 1.2 and 1.3]{Fan-Zhao-2001}.

\begin{proposition}
	Let $r\in C_+(\overline{\Omega})$, $\lambda>0$, and $u\in \Lp{r(\cdot)}$, then
	\begin{enumerate}
		\item[\textnormal{(i)}]
			$\|u\|_{r(\cdot)}=\lambda$  if and only if $ \varrho_{r(\cdot)}\left(\frac{u}{\lambda}\right)=1$ with $u \neq 0$;
		\item[\textnormal{(ii)}]
			$\|u\|_{r(\cdot)}<1$ (resp. $=1$, $>1$)  if and only if $\varrho_{r(\cdot)}(u)<1$ (resp. $=1$, $>1$);
		\item[\textnormal{(iii)}]
			if $\|u\|_{r(\cdot)}<1$, then $\|u\|_{r(\cdot)}^{r^+} \leq \varrho_{r(\cdot)}(u) \leq \|u\|_{r(\cdot)}^{r^-}$;
		\item[\textnormal{(iv)}]
			if $\|u\|_{r(\cdot)}>1$, then $\|u\|_{r(\cdot)}^{r^-} \leq \varrho_{r(\cdot)}(u) \leq \|u\|_{r(\cdot)}^{r^+}$;
		\item[\textnormal{(v)}]
			$\|u\|_{r(\cdot)} \to 0$  if and only if  $\varrho_{r(\cdot)}(u)\to 0$;
		\item[\textnormal{(vi)}]
			$\|u\|_{r(\cdot)}\to +\infty$  if and only if  $\varrho_{r(\cdot)}(u)\to +\infty$.
	\end{enumerate}
\end{proposition}

Next, we introduce the corresponding variable exponent Sobolev space $W^{1,r(\cdot)}(\Omega)$ which is defined by
\begin{align*}
	W^{1,r(\cdot)}(\Omega)=\l\{ u \in \Lp{r(\cdot)}\colon  |\nabla u| \in \Lp{r(\cdot)}\r\}
\end{align*}
endowed with the norm
\begin{align*}
	\|u\|_{1, r(\cdot)} = \inf \left\lbrace \lambda > 0\colon \varrho_{1, r(\cdot)} \left( \frac{u}{\lambda} \right)  \leq 1 \right\rbrace,
\end{align*}
where
\begin{align*}
	\varrho_{1, r(\cdot)} (u) = \varrho_{ r(\cdot)} (u) + \varrho_{ r(\cdot)} ( \nabla u ),
\end{align*}
with $\varrho_{ r(\cdot)} ( \nabla u ) = \varrho_{ r(\cdot)} ( |\nabla u|)$. Moreover, we denote
\begin{align*}
	\Wpzero{r(\cdot)}= \overline{C^\infty _0(\Omega)}^{\|\cdot\|_{1,r(\cdot)}}.
\end{align*}
We know that the spaces $W^{1,r(\cdot)}$ and $\Wpzero{r(\cdot)}$ are both separable and reflexive Banach spaces and the norm $\|\cdot\|_{1,r}$ is uniformly convex. Furthermore, we have the Poincar\'{e} inequality, that is
\begin{align*}
	\|u\|_{r(\cdot)} \leq c_0 \|\nabla u\|_{r(\cdot)}
	\quad\text{for all } u \in \Wpzero{r(\cdot)}.
\end{align*}
for some $c_0>0$. Therefore, we can equip the space $\Wpzero{r(\cdot)}$ with the equivalent norm $\|u\|_{W^{1,r(\cdot)}_0(\Omega)}=\|\nabla u\|_{r(\cdot)}$ which turns out to be uniformly convex as well.

For $r \in C_+ (\overline{\Omega})$ we introduce the critical Sobolev variable exponents $r^*$ and $r_*$ with the following expression
\begin{align*}
	r^*(x) & =
	\begin{cases}
		\frac{N r(x)}{N - r(x) } & \text{if } r(x) < N    \\
		\text{any number }s \in (r(x),\infty) & \text{if } r(x) \geq N
	\end{cases}
	, \quad\text{for all } x \in \overline{\Omega}, \\[1ex]
	r_*(x) & =
	\begin{cases}
		\frac{(N -1) r(x)}{N - r(x) } & \text{if } r(x) < N    \\
		\text{any number }s \in (r(x),\infty) & \text{if } r(x) \geq N
	\end{cases}
	, \quad\text{for all } x \in \overline{\Omega}.
\end{align*}


Let us denote the space $C^{0, \frac{1}{|\log t|}}(\overline{\Omega})$ by the set of all functions $h\colon \overline{\Omega} \to \R$ which are log-H\"{o}lder continuous, i.e.\,there exists $C>0$ such that
\begin{align*}
	|h(x)-h(y)| \leq \frac{C}{|\log |x-y||}\quad\text{for all } x,y\in \overline{\Omega} \text{ with } |x-y|<\frac{1}{2}.
\end{align*}

If $r \in C^{0, \frac{1}{|\log t|}}(\overline{\Omega})$, then the set $C_{c}^\infty(\Omega)$ of smooth functions with finite support is dense in $W^{1,r(\cdot)}_0(\Omega)$. Given a function $u\in W^{1,r(\cdot)}_0(\Omega)$ with {$r \in C^{0, \frac{1}{|\log t|}}(\overline{\Omega})$}, the smooth approximations of $u$ in $W^{1,r(\cdot)}_0(\Omega)$ can be obtained by means of the Friedrichs mollifiers.

The following analog of the Sobolev embedding theorem holds. Given $p,q \in C_+(\overline{\Omega})$ with  $\inf_{\Omega}(p^\ast(x)-q(x))>0$, then for every $u\in W^{1,p(\cdot)}_0(\Omega)$
\begin{align*}
	\|u\|_{q(\cdot),\Omega}\leq C\|u\|_{W^{1,p(\cdot)}_0(\Omega)},\quad\text{with } C=C(p^\pm,q^\pm,|\Omega|,N),
\end{align*}
and the embedding $W^{1,p(\cdot)}_0(\Omega)\subset L^{q(\cdot)}(\Omega)$ is compact.

Next, we begin by recalling some definitions and preliminary results from Fan \cite{Fan-2012}, Harjulehto--H\"{a}st\"{o} \cite{Harjulehto-Hasto-2019} and Musielak \cite{Musielak-1983} in order to introduce Musielak-Orlicz Sobolev spaces and its properties. First, let us denote by $(X,\Sigma,\mu)$ a $\sigma$-finite, complete measure space with $\mu \not \equiv 0$.

\begin{definition}
	Let  $\varphi \colon X \times (0,+\infty) \to \R$. We say that
	\begin{enumerate}
		\item[\textnormal{(i)}]
			$\varphi$ is almost increasing in the second variable if there exists $a \geq 1$ such that $\varphi(x,s) \leq a \varphi(s,t)$ for all $0 < s < t$ and for a.a.\,$x \in X$;
		\item[\textnormal{(ii)}]
			$\varphi$ is almost decreasing in the second variable if there exists $a \geq 1$ such that $a \varphi(x,s) \geq \varphi(x,t)$ for all $0 < s < t$ and for a.a.\,$x \in X$.
	\end{enumerate}
\end{definition}

\begin{definition}
	Let $\varphi \colon X \times (0,+\infty) \to \R$ and $p,q>0$. We say that $\varphi$ satisfies the property
	\begin{enumerate}[leftmargin=2cm]
		\item[\textnormal{(Inc)}$_p$]
			if $t^{-p}\varphi(x,t)$ is increasing in the second variable;
		\item[\textnormal{(aInc)}$_p$]
			if $t^{-p}\varphi(x,t)$ is almost increasing in the second variable;
		\item[\textnormal{(Dec)}$_q$]
			if $t^{-q}\varphi(x,t)$ is decreasing in the second variable;
		\item[\textnormal{(aDec)}$_q$]
			if $t^{-q}\varphi(x,t)$ is almost decreasing in the second variable.
	\end{enumerate}
	Without subindex, that is \textnormal{(Inc)}, \textnormal{(aInc)}, \textnormal{(Dec)} and \textnormal{(aDec)}, it indicates that there exists some $p>1$ or $q<\infty$ such that the condition holds.
\end{definition}

Next, we give the definition of a generalized $\Phi$-function.

\begin{definition}
	A function $\varphi \colon X \times [0,+\infty) \to [0,+\infty]$ is said to be a generalized $\varphi$-function if $\varphi$ is measurable in the first variable, increasing in the second variable and satisfies $\varphi(x,0)=0$, $\lim_{t\to 0^+} \varphi(x,t) = 0$ and $\lim_{t \to +\infty} \varphi(x,t) = +\infty$ for a.a.\,$x \in X$. Moreover, we say that
	\begin{enumerate}
		\item[\textnormal{(i)}]
			$\varphi$ is a generalized weak $\varphi$-function if it satisfies \textnormal{(aInc)}$_1$ on $X \times (0,+\infty)$;
		\item[\textnormal{(ii)}]
			$\varphi$ is a generalized convex $\varphi$-function if $\varphi(x,\cdot)$ is left-continuous and convex for a.a.\,$x \in X$;
		\item[\textnormal{(iii)}]
			$\varphi$ is a generalized strong $\varphi$-function if $\varphi(x,\cdot)$ is continuous in the topology of $[0,\infty]$ and convex for a.a.\,$x \in X$.
	\end{enumerate}
	The set of all generalized strong ${\bf \Phi}$-function is denoted by ${\bf \Phi}(\Omega)$.
\end{definition}

Now we define the conjugate of a generalized $\varphi$-function and its left-inverse.

\begin{definition}\label{def:conjugate}
	Let $\varphi \colon X \times [0,+\infty) \to [0,+\infty]$. We denote by $\varphi^\sharp$ the conjugate function of $\varphi$ which is defined for $x \in X$ and $s \geq 0$ by
	\begin{align*}
		\varphi^\sharp(x,s) = \sup_{t \geq 0} (ts - \varphi(x,t)).
	\end{align*}
	We denote by $\varphi^{-1}$ the left-continuous inverse of $\varphi$, defined for $x \in X$ and $s \geq 0$ by
	\begin{align*}
		\varphi^{-1}(x,s) = \inf \{t \geq 0\colon \varphi(x,t) \geq s\}.
	\end{align*}
\end{definition}

\begin{definition}
	Let $\varphi \colon X \times [0,+\infty) \to [0,+\infty]$, we say that
	\begin{enumerate}
		\item[\textnormal{(i)}]
			$\varphi$ is doubling (or satisfies the $\Delta_2$-condition) if there exists a constant $K \geq 2$ such that
			\begin{align*}
				\varphi(x,2t) \leq K \varphi(x,t)
			\end{align*}
			for all $t \in (0,+\infty]$ and for a.a.\,$x \in X$;
		\item[\textnormal{(ii)}]
			$\varphi$ satisfies the $\nabla_2$ condition if $\varphi^\sharp$ satisfies the $\Delta_2$-condition.
		\item[\textnormal{(iii)}]
			Let $\phi_1, \phi_2 \in {\bf \Phi}(\Omega)$. We say $\phi_1$ is weaker than $\phi_2$, denoted by $ \phi_1 \prec \phi_2 $, if there exist constants $C_1, C_2>0$ and $h \in L^1(\Omega)$, $h \geq 0$ such that
			\begin{align*}
				\phi_1(x,t) \leq C_1 \phi_2(x,C_2 t) + h(x) \quad \text{for a.a.\,}x \in \Omega \text{ and for all } t \geq 0.
			\end{align*}
			We say that the functions $\phi_1, \phi_2$ are equivalent denoted by $\phi_1 \simeq \phi_2$, if there exists $L \geq 1$ such that
			\begin{align*}
				\phi_1(x, tL^{-1}) \leq \phi_2(x,t) \leq \phi_1(x, tL) \quad \text{for a.a.\,}x \in \Omega \text{ and for all } t \geq 0,
			\end{align*}
			or weakly equivalent denoted by $\phi_1 \sim \phi_2$ if there exists $L \geq 1$ and $h \in L^1(\Omega)$ such that
			\begin{align*}
				\phi_2(x,t) \leq \phi_1(x, tL) + h(x)
			\end{align*}
			and
			\begin{align*}
				\phi_1(x,t) \leq \phi_2(x, tL) + h(x)  \quad \text{for a.a.\,} x \in \Omega \text{ and for all } t \geq 0.
			\end{align*}
		\item[\textnormal{(iv)}]
			A function  $\psi \colon [0, \infty) \to [0, \infty)$ is said to be a $N$-function ($N$ stands for Nice function) if $\psi$ is a ${\bf \Phi}(\Omega)$ function and
			\begin{align*}
				\lim_{t \to 0^+} \frac{\psi(t)}{t}=0
				\quad\text{ and }\quad
				\lim_{t \to \infty} \frac{\psi(t)}{t}=\infty.
			\end{align*}
			A function $\psi\colon \Omega \times [0, \infty) \to [0, \infty)$ is said to be a generalized $N$-function if $\psi(\cdot, t)$ is measurable for all $t \geq 0$ and $\psi(x, \cdot)$ is a $N$-function for a.a.\,$x \in \Omega$. We denote by ${\bf N_\Phi}(\Omega)$ the class of all generalized $N$-function on $\Omega$.
		\item[\textnormal{(v)}]
			Given $\phi_1, \phi_2 \in {\bf N_\Phi}(\Omega)$, we say $\phi_1$ increases essentially slower than $\phi_2$ near infinity, denoted by $\phi_1 \ll \phi_2$, if for any $\ell >0$
			\begin{align*}
				\lim_{t \to \infty } \frac{\phi_1(x, \ell t)}{\phi_2(x,t)} = 0 \quad \text{uniformly for a.a.\,} x \in \Omega.
			\end{align*}
	\end{enumerate}
\end{definition}

\begin{definition}
	Let $\varphi \colon \Omega \times [0,+\infty) \to [0,+\infty]$ be a generalized $\varphi$-function, we say that it satisfies the condition
	\begin{enumerate}[leftmargin=1.5cm]
		\item[\textnormal{(A0)}]
			if there exists $0 < \beta \leq 1$ such that $\beta \leq \varphi^{-1} (x,1) \leq \beta^{-1}$ for a.a.\,$x \in \Omega$;
		\item[\textnormal{(A0)'}]
			if there exists $0 < \beta \leq 1$ such that $\varphi (x,\beta) \leq 1 \leq \varphi(x,\beta^{-1})$ for a.a.\,$x \in \Omega$;
		\item[\textnormal{(A1)}]
			if there exists $0 < \beta < 1$ such that $\beta \varphi^{-1}(x,t) \leq \varphi^{-1}(y,t)$ for every $t \in [1 , 1/|B|]$ and  for a.a.\,$x,y \in B \cap \Omega$ with every ball $B$ such that $|B| \leq 1$;
		\item[\textnormal{(A1)'}]
			if there exists $0 < \beta < 1$ such that $\varphi( x,\beta t) \leq \varphi (y,t)$ for every $t \geq 0$ such that $\varphi( y,t) \in [1 , 1/|B|]$ and  for a.a.\,$x,y \in B \cap \Omega$ with every ball $B$ such that $|B| \leq 1$;
		\item[\textnormal{(A2)}]
			if for every $s>0$ there exists $0 < \beta \leq 1$ and $h \in \Lp{1} \cap \Lp{\infty}$ such that $\beta \varphi^{-1}(x,t) \leq \varphi^{-1}(y,t)$ for every $t \in [h(x) + h(y) , s]$ and for a.a.\,$x,y \in \Omega$;
		\item[\textnormal{(A2)'}]
			if there exists $s>0$, $0 < \beta \leq 1$, $\varphi_\infty$ weak $\varphi$-function (that is, constant in the first variable) and $h \in \Lp{1} \cap \Lp{\infty}$ such that $\varphi( x,\beta t) \leq \varphi_\infty(t) + h(x)$ and $\varphi_\infty (\beta t) \leq \varphi(x,t) + h(x)$ for a.a.\,$x \in \Omega$ and for all $t\geq0$ such that $\varphi_\infty(t) \leq s$ and $\varphi( x,t) \leq s$.
	\end{enumerate}
\end{definition}

The following result can be found in Harjulehto--H\"{a}st\"{o} \cite{Harjulehto-Hasto-2019}.

\begin{lemma}
	Let $\varphi \colon \Omega \times [0,+\infty) \to [0,+\infty]$ be a generalized weak $\varphi$-function, then
	\begin{enumerate}
		\item[\textnormal{(i)}]
			$\varphi$ satisfies the $\Delta_2$-condition if and only if $\varphi$ satisfies \textnormal{(aDec)}$_q$ for some $q>1$;
		\item[\textnormal{(ii)}]
			if $\varphi$ is a generalized convex $\varphi$-function, $\varphi$ satisfies the $\Delta_2$-condition if and only if $\varphi$ satisfies \textnormal{(Dec)$_q$} for some $q>1$;
		\item[\textnormal{(iii)}]
			$\varphi$ satisfies the $\nabla_2$ condition if and only if $\varphi$ satisfies \textnormal{(aInc)$_q$} for some $q>1$;
		\item[\textnormal{(iv)}]
			$\varphi$ satisfies the \textnormal{(A0)} condition if and only if $\varphi$ satisfies the \textnormal{(A0)'} condition;
		\item[\textnormal{(v)}]
			if $\varphi$ satisfies the \textnormal{(A0)} condition, the \textnormal{(A1)} condition holds if and only if the \textnormal{(A1)'} condition holds;
		\item[\textnormal{(vi)}]
			$\varphi$ satisfies the \textnormal{(A2)} condition if and only if $\varphi$ satisfies the \textnormal{(A2)'} condition.
	\end{enumerate}
\end{lemma}

For the next result we refer again to Harjulehto--H\"{a}st\"{o} \cite{Harjulehto-Hasto-2019}.

\begin{proposition}\label{Prop:AbstractBanach}
	Let $\varphi \colon X \times [0,+\infty) \to [0,+\infty]$ be a generalized weak $\varphi$-function and let its associated modular be
	\begin{align*}
		\varrho_\varphi (u) = \int_X \varphi(x,|u(x)|) \,\mathrm{d} \mu (x).
	\end{align*}
	Then, the set
	\begin{align*}
		L^\varphi(X) = \{ u \in M(X)\colon \varrho_\varphi (\lambda u) < \infty \text{ for some } \lambda > 0 \}
	\end{align*}
	equipped with the associated Luxemburg quasi-norm
	\begin{align*}
		\|u\|_\varphi = \inf \left\lbrace \lambda > 0 \colon \varrho_\varphi \left( \frac{u}{\lambda} \right)  \leq 1 \right\rbrace
	\end{align*}
	is a quasi Banach space. Furthermore, if $\varphi$ is a generalized convex $\varphi$-function, it is a Banach space; if $\varphi$ satisfies \textnormal{(aDec)$_q$} for some $q > 1$, it holds that
	\begin{align*}
		L^\varphi(X) = \{ u\in M(X) \colon \varrho_\varphi (u) < \infty \};
	\end{align*}
	if $\varphi$ satisfies \textnormal{(aDec)$_q$} for some $q >1$ and $\mu$ is separable, then $L^\varphi(X)$ is separable; and if $\varphi$ satisfies \textnormal{(aInc)$_p$} and \textnormal{(aDec)$_q$} for some $p$, $q > 1$ it possesses an equivalent, uniformly convex norm, hence it is reflexive.
\end{proposition}

The relation of the modular and the norm is stated in the following proposition.

\begin{proposition}\label{Prop:AbstractNormModular}
	Let $\varphi \colon X \times [0,+\infty) \to [0,+\infty]$ be a generalized weak $\varphi$-function that satisfies \textnormal{(aInc)}$_p$ and \textnormal{(aDec)}$_q$, with $1 \leq p \leq q < \infty$. Then
	\begin{align*}
		\frac{1}{a} \min \left\lbrace \|u\|_\varphi^p ,  \|u\|_\varphi^q \right\rbrace
		\leq \varrho_\varphi (u)
		\leq a \max \left\lbrace \|u\|_\varphi^p ,  \|u\|_\varphi^q \right\rbrace
	\end{align*}
	for all measurable functions $u \colon X \to \R$, where $a$ is the maximum of the constants of \textnormal{(aInc)}$_p$ and \textnormal{(aDec)}$_q$.
\end{proposition}

The following characterization of suitable embeddings has been proven by Musielak \cite[Theorems 8.4 and 8.5]{Musielak-1983}.

\begin{proposition}\label{Prop:AbstractEmbedding}
	Let $\varphi, \psi \colon X \times [0,+\infty) \to [0,+\infty]$ be generalized weak $\varphi$-functions and let $\mu$ be atomless. Then $L^\varphi(X) \hookrightarrow L^\psi (X)$ if and only if there exists $K>0$ and a non-negative integrable function $h$ such that for all $t \geq 0$ and for a.a.\,$x \in X$
	\begin{align*}
		\psi\left(x, t \right) \leq K \varphi (x,t) + h(x).
	\end{align*}
\end{proposition}

We also have the following H\"older inequality in Musielak-Orlicz spaces, see Harjulehto--H\"{a}st\"{o} \cite{Harjulehto-Hasto-2019}.

\begin{proposition}\label{holder:ineq}
	Let $\varphi \colon X \times [0,+\infty) \to [0,+\infty]$ be a generalized weak $\varphi$-function, then
	\begin{align*}
		\int_X \abs{u} \abs{v} \,\mathrm{d} \mu (x) \leq 2 \|u\|_{\varphi} \|v\|_{\varphi^\sharp} \quad \text{for all } u \in L^{\varphi}(X),\ v \in L^{\varphi^\sharp}(X).
	\end{align*}
	Moreover, the constant $2$ is sharp.
\end{proposition}

\begin{proposition}\label{modular-conjugate:relation}
	Let $\varphi \colon X \times [0,+\infty) \to [0,+\infty]$ be a generalized weak $\varphi$-function such that $\varphi \in \Delta_2 \cap \nabla_2$, then
	\begin{align*}
		\varphi^\sharp\left(x,\frac{\varphi(x,t)}{t}\right) \leq \varphi(x,t) \leq \varphi^\sharp\left(x,\frac{2 \varphi(x,t)}{t}\right) \quad \text{for all }   t>0.
	\end{align*}
\end{proposition}

Finally, we can also define associated Sobolev spaces to these Musielak-Orlicz spaces analogously to the classical case. We refer to Harjulehto--H\"{a}st\"{o} \cite{Harjulehto-Hasto-2019}.

\begin{proposition}\label{Prop:AbstractSobolev}
	Let $\varphi \colon \Omega \times [0,+\infty) \to [0,+\infty]$ be a generalized weak $\varphi$-function such that $L^{\varphi}(\Omega) \subseteq L^1_{loc} (\Omega)$ and $k\geq 1$. Then, the set
	\begin{align*}
		W^{k,\varphi} (\Omega) = \{ u \in L^{\varphi}(\Omega) \colon \partial_\alpha u \in L^{\varphi}(\Omega) \text{ for all } \abs{\alpha} \leq k \},
	\end{align*}
	where we consider the modular
	\begin{align*}
		\varrho_{k,\varphi} (u) = \sum_{0 \leq \abs{\alpha} \leq k } \varrho_\varphi(\partial_\alpha u)
	\end{align*}
	and the associated Luxemburg quasi-norm
	\begin{align*}
		\|u\|_{k,\varphi} = \inf \left\lbrace \lambda > 0 \colon  \varrho_{k,\varphi} \left( \frac{u}{\lambda} \right)  \leq 1 \right\rbrace
	\end{align*}
	is a quasi Banach space. Analogously, the set
	\begin{align*}
		W^{k,\varphi}_0 (\Omega) = \overline{C_0^\infty (\Omega)}^{\|\cdot\|_{k,\varphi}},
	\end{align*}
	where $C_0^\infty (\Omega)$ are the $C^\infty (\Omega)$ functions with compact support, equipped with the same modular and norm is also a quasi Banach space. Furthermore, if $\varphi$ is a generalized convex $\varphi$-function, both spaces $W^{k,\varphi} (\Omega)$ and $W_0^{k,\varphi} (\Omega)$ are Banach spaces; if $\varphi$ satisfies \textnormal{(aDec)$_q$} for some $q > 1$, then they are separable; and if $\varphi$ satisfies \textnormal{(aInc)$_p$} and \textnormal{(aDec)$_q$} for some $p$, $q > 1$ they possess an equivalent, uniformly convex norm, hence they are reflexive.
\end{proposition}

The next proposition summarizes the relation between the norm in $W^{k,\varphi} (\Omega)$ and its modular.

\begin{proposition}\label{Prop:AbstractoneNormModular}
	Let $\varphi \colon \Omega \times [0,+\infty) \to [0,+\infty]$ be a generalized weak $\varphi$-function that satisfies \textnormal{(aInc)}$_p$ and \textnormal{(aDec)}$_q$, with $1 \leq p \leq q < \infty$. Then
	\begin{align*}
		\frac{1}{a} \min \left\lbrace \|u\|_{k,\varphi}^p ,  \|u\|_{k,\varphi}^q \right\rbrace
		\leq \varrho_{k,\varphi} (u)
		\leq a \max \left\lbrace \|u\|_{k,\varphi}^p ,  \|u\|_{k,\varphi}^q \right\rbrace
	\end{align*}
	for all $u \in W^{k,\varphi} (\Omega)$, where $a$ is the maximum of the constants of \textnormal{(aInc)}$_p$ and \textnormal{(aDec)}$_q$.
\end{proposition}

Now we can give our precise assumptions:
\begin{enumerate}[label=\textnormal{(H$_0$)},ref=\textnormal{H$_0$}]
	\item\label{main:assump}
	\begin{enumerate}
		\item[\textnormal{(i)}]
			$\Omega \subset \mathbb{R}^N$, $N \geq 2$, is a bounded domain with Lipschitz boundary $\partial \Omega$;
		\item[\textnormal{(ii)}]
			$p, q \in C(\overline{\Omega})$ with $1 < p(x), q(x)< N$ for all $x \in \overline{\Omega}$ and $s \in L^\infty(\Omega)$ such that $q(x)+ s(x) \geq r> 1$ for a.a.\,$x \in \Omega$;
		\item[\textnormal{(iii)}]
			$0 \leq a, b\in L^1(\Omega)$ and $a(x) + b(x) \geq d >0$ for a.a.\,$x \in \Omega$.
	\end{enumerate}
\end{enumerate}
We define the function $\mathcal{S}\colon \Omega\times [0,\infty)\mapsto [0,\infty)$ by
\begin{align}\label{eq:H}
	\mathcal{S}(x,t)=a(x) t^{p(x)}+b(x) t^{q(x)}\log^{s(x)}(1 + t) \quad \text{for } (x, t) \in \Omega \times (0, \infty)
\end{align}
and denote
\begin{align*}
	\alpha(x)= \min\{p(x), q(x)\}
	\quad \text{ and } \quad
	\beta(x)= \max\{p(x), q(x)\}.
\end{align*}

The following lemma summarizes the main properties of the $\Phi$-function $\mathcal{S}$.

\begin{lemma}\label{Le:Prop-S}
	Let \eqref{main:assump} be satisfied. Then $\mathcal{S}$ is a generalized strong $\Phi$-function and satisfy \textnormal{(Inc)}$_{\ell^-}$  and \textnormal{(Dec)}$_{\ell^+}$ for $1 < \ell^- = \min\{p^-, (q + \lfloor s \rfloor)^-\}$ and for $\ell^+ = \max\{p^+, (q + \lceil s \rceil)^+\}$.
\end{lemma}

\begin{proof}
	Observe that in light of \eqref{main:assump}, we have
	\begin{align*}
		\begin{cases}
			\text{the map } x \mapsto \mathcal{S}(x, t) \text{ is a locally integrable for all }t \geq 0\\
			\text{the map } t \mapsto \mathcal{S}(x, t) \text{ is increasing and continuous for a.a.\,}x \in \Omega.
		\end{cases}
	\end{align*}
	For all $t>0$ and for a.a.\,$x \in \Omega$, we have
	\begin{align*}
		{\partial_t}^2\mathcal{S} (x,t)
		& = a(x)p(x) ( p(x) - 1) t^{p(x) - 2} + b(x) t^{q(x) -2} \log^{s(x)-2}(1+t) \times\bigg[q(x)(q(x)-1) \log^2(1 + t) \\
		& \quad \quad + 2 q(x)s(x) \log(1+t) \frac{t}{1+t} + s(x)(s(x) - 1)\frac{t^2}{(1+t)^2}  - s(x) \log(1+t) \frac{t^2 }{(1+t)^2}  \bigg] \\
		& :=  a(x)p(x) ( p(x) - 1) t^{p(x) - 2} + b(x) t^{q(x) -2} \log^{s(x)-2}(1+t) \times {\bf M}.
	\end{align*}
	The convexity of the function $\mathcal{S}(x, \cdot)$ follows from the sign of its second derivative. From the above estimate, it is enough to prove that ${\bf M} >0$.\\
	\textbf{Case 1:} $s(x) \geq 0$

	It holds
	\begin{align*}
		{\bf M} & = q(x)(q(x)-1) \log^2(1 + t) + s(x) \frac{t}{(1+t)}\log(1+t) \left[q(x) - \frac{t}{1+t}\right] \\
		& \qquad + s(x) \frac{t}{1+t} \left[q(x) \log(1+t) + (s(x)-1) \frac{t}{1+t}\right]\\
		& \geq q(x)(q(x)-1) \log^2(1 + t) + s(x) \frac{t}{(1+t)}\log(1+t) \left[1- \frac{t}{1+t}\right]  \\
		& \qquad + s(x) q(x) \frac{t}{1+t} \left[\log(1+t) -\frac{t}{1+t}\right] >0,
	\end{align*}
	where in the last inequality we have used the fact that $q(x) + s(x) \geq 1$.\\
	\textbf{Case 2: $s(x) < 0$}

	For a fixed $x$ such that $s(x)<0$, we define
	\begin{align*}
		h_x(t):= 2 q(x) (1+t) \log(1+t) - t \log(1+t) + 2(s(x)-1) t
	\end{align*}
	and
	\begin{align*}
		z_x(t) & := \frac{2 s(x) q(x) t}{(1+t) \log(1+t)}  +  \frac{ s(x) (s(x)-1) t^2}{(1+t)^2 (\log(1+t))^2} -  \frac{s(x) t^2}{(1+t)^2 \log(1+t)}\\
		& =  \frac{s(x) t}{(1+t) \log(1+t)}\left[2 q(x)  +  \frac{ (s(x)-1) t}{(1+t) (\log(1+t))} -  \frac{t}{(1+t)}\right].
	\end{align*}
	Note that
	\begin{align*}
		\lim_{t \to 0} \frac{t^2}{(1+t)^2 (\log(1+t))^2} = \lim_{t \to 0} \frac{t}{(1+t) \log(1+t)} =1, \quad  \lim_{t \to 0} \frac{t^2}{(1+t)^2 \log(1+t)} = 0
	\end{align*}
	and
	\begin{align*}
		{\bf M}= \log^2(1+t) \left[q(x) (q(x)-1) + z_x(t)\right].
	\end{align*}
	Since $q(x) + s(x) \geq 1$ and $s(x) <0$, we have
	\begin{align*}
		z_x(0)= \lim_{t \to 0} z_x(t) = s(x)\left[(s(x)-1) + 2 q(x) \right] \leq 0 \quad \text{and} \quad h_x(0) = 0.
	\end{align*}
	Next, we claim that both $h_x(\cdot)$ and $z_x(\cdot)$ are increasing functions and achieves their infimums $h_x(0)$  and $z_x(0)$ at $0$, respectively. Differentiating with respect to $t$ and using
	\begin{align*}
		(1+t) \log(1+t) \geq t \geq \ln(1+t),\quad q(x) + s(x) \geq 1,\quad q(x) \geq 1,
	\end{align*}
	we obtain
	\begin{align*}
		h_x'(t) & = 2 q(x) [1+ \log(1+t)] - \frac{t}{1+t} - \log(1+t) + 2(s(x)-1)\\
		& = 2 (q(x) + s(x)-1) + (2 q(x)-1) \log(1+t) - \frac{t}{1+t} \\
		& \geq 2 (q(x) + s(x)-1) + 2 (q(x)-1) \geq 0
	\end{align*}
	and
	\begin{align*}
		z_x'(t) &= \frac{s(x) t}{(1+t) \log(1+t)}\left[\frac{ (1-s(x)) [t-\log(1+t)]}{(1+t)^2 (\log^2(1+t))} -  \frac{1}{(1+t)^2} \right]\\
		& \qquad - \frac{s(x) [t-\log(1+t)]}{(1+t)^2 \log^2(1+t)}\left[2 q(x)  +  \frac{ (s(x)-1) t}{(1+t) (\log(1+t))} -  \frac{t}{(1+t)}\right]\\
		& = - \frac{s(x) [t-\log(1+t)]}{(1+t)^3 \log^3(1+t)} \left[ 2 q(x) (1+t) \log(1+t) - t \log(1+t) + 2(s(x)-1) t \right]\\
		& \qquad - \frac{s(x) t}{(1+t)^3 \log(1+t)} \geq 0.
	\end{align*}
	This implies
	\begin{align*}
		{\bf M} &= \log^2(1+t) \left[q(x) (q(x)-1) + z_x(t) \right]  \geq \log^2(1+t) \left[q(x) (q(x)-1) + z_x(0) \right]\\
		& = \log^2(1+t) \left[q(x) (q(x)-1) + s(x)\left[(s(x)-1) + 2 q(x) \right] \right]\\
		& = \log^2(1+t) \left[(q(x) + s(x)) (q(x) + s(x) -1)\right] \geq 0.
	\end{align*}
	Note that, for any $\ell \in \{\ell^-, \ell^+\}$ and $x \in \overline{\Omega}$, we have
	\begin{align*}
		\frac{\mathrm{d}}{\mathrm{d}t}\l(t^{q(x) - \ell} \log^{s(x)}(1 + t)\r) = t^{q(x) - \ell-1} \log^{s(x)-1}(1 + t) \l((q(x)-\ell) \log(1 + t) + s(x) \frac{t}{1+t}\r)
	\end{align*}
	For $\ell= \ell^- = \min\{p^-, (q + \lfloor s \rfloor)^-\}$
	\begin{equation}\label{est:incre}
		\begin{aligned}
			(q(x)-\ell^-) \log(1 + t) + s(x) \frac{t}{1+t}
			& \geq
			\begin{cases}
				(q(x)-\ell^- + s(x)) \frac{t}{(1 + t)}  & \text{if } s(x) \geq 0    \\
				0  & \text{if } s(x) < 0
			\end{cases} \\
			& \quad \geq 0.
		\end{aligned}
	\end{equation}
	For $\ell= \ell^+ = \max\{p^+, (q + \lceil s \rceil)^+\}$
	\begin{equation}\label{est:decre}
		\begin{aligned}
			(q(x)-\ell^+) \log(1 + t) + s(x) \frac{t}{1+t}
			& \leq
			\begin{cases}
				(q(x)-\ell^+ + s(x)) \log(1 + t) & \text{if } s(x) \geq 0    \\
				0   & \text{if } s(x) < 0
			\end{cases} \\
			& \quad \leq 0.
		\end{aligned}
	\end{equation}
	Therefore, in the view of the condition $q(x) + s(x) \geq r>1$ in  \eqref{main:assump}\textnormal{(ii)} as well as \eqref{est:incre} and \eqref{est:decre}, it is easy to verify that, for $1 < \ell^- = \min\{p^-, (q + \lfloor s \rfloor)^-\}$ and $\ell^+ = \max\{p^+, (q + \lceil s \rceil)^+\}$,
	\begin{align*}
		\frac{\mathcal{S} (x,t)}{t^{\ell^-}}
		& = a(x)t^{p(x)- \ell^-} + b(x) t^{q(x) - \ell^-} \log^{s(x)}(1 + t), \\
		\frac{\mathcal{S} (x,t)}{t^{\ell^+}}
		& = t^{p(x) - \ell^+} + \mu(x) t^{q(x) - \ell^+} \log^{s(x)}(1+ t)
	\end{align*}
	are increasing and decreasing functions, respectively.
\end{proof}

As a consequence of the previous Lemma, we obtain the following.

\begin{proposition}\label{pro:spaces-propert}
	Let the assumption \eqref{main:assump} hold. Then, the spaces $L^{\mathcal{S}}(\Omega)$, $W^{1, \mathcal{S}}(\Omega)$ and $W_0^{1, \mathcal{S}}(\Omega)$ are reflexive Banach spaces.
\end{proposition}

\begin{proof}
	Using Lemma \ref{Le:Prop-S} and Propositions \ref{Prop:AbstractBanach} and \ref{Prop:AbstractSobolev}, we know that the spaces $L^{\mathcal{S}}(\Omega)$, $W^{1, \mathcal{S}}(\Omega)$ and $W_0^{1, \mathcal{S}}(\Omega)$ are separable Banach spaces. To prove the remaining claim, it is enough to prove that $L^{\mathcal{S}}(\Omega)$ is reflexive. From Lemma \ref{Le:Prop-S}, we know that the map $t \mapsto \frac{\mathcal{S}(x,t)}{t^{\ell^-}}$ is increasing and the map $t \mapsto \frac{\mathcal{S}(x,t)}{t^{\ell^+}}$ is decreasing, i.e., $\mathcal{S}$ satisfies \textnormal{(Inc)}$_{\ell^-}$ and \textnormal{(Dec)}$_{\ell^+}$. Now, by using Harjulehto--H\"{a}st\"{o} \cite[Proposition 3.6.2,  Theorem 3.6.6 and Lemma 2.1.8]{Harjulehto-Hasto-2019}, there exists a uniformly convex function $\phi \in {\bf \Phi}(\Omega)$ such that $\mathcal{S} \simeq \phi$ and $\phi$ satisfies \textnormal{(Dec)}$_{\ell^+}$. Hence, using again Harjulehto--H\"{a}st\"{o} \cite[Propositions 3.2.4 and 3.6.6]{Harjulehto-Hasto-2019}, $L^\phi(\Omega)$ is uniformly convex, $L^{\mathcal{S}}(\Omega) = L^{\phi}(\Omega)$ and by the Milman-Pettis Theorem it follows that $L^{\mathcal{S}}(\Omega)$ is reflexive.
\end{proof}

Now we denote, for all $u\in L^{\mathcal{S}}(\Omega)$,
\begin{align*}
	\varrho_{\mathcal{S}}(u)
	& = \int_{\Omega} \left(a(x)|u|^{p(x)} + b(x) |u|^{q(x)}  \log^{s(x)}(1+|u|) \right) \,\mathrm{d}x, \\
	\|u\|_{\mathcal{S}}
	& =\inf\left\{\lambda>0\colon \varrho_{\mathcal{S}}\left(\dfrac{u}{\lambda}\right)\leq 1\right\}.
\end{align*}
The following proposition shows the relation between  $\varrho_{\mathcal{S}}(\cdot)$ and  $\|\cdot\|_{\mathcal{S}}$.

\begin{proposition}\label{pro:norm-mod:relation}
	Assume \eqref{main:assump} holds true and $u \in L^{\mathcal{S}}(\Omega)$. Then the following hold:
	\begin{enumerate}
		\item[\textnormal{(i)}]
			If $u \neq 0$, then $\|u\|_{\mathcal{S}} = \eta$ if and only if $\varrho_{\mathcal{S}} \left(\frac{u}{\eta}\right)=1$.
		\item[\textnormal{(ii)}]
			$\varrho_{\mathcal{S}}(u) <1$ $($or $=1$ or $>1)$ $\Leftrightarrow \; \|u\|_{\mathcal{S}}<1$ $($or $=1$ or $>1)$.
		\item[\textnormal{(iii)}]
			If $\|u\|_{\mathcal{S}} < 1$, then $\|u\|^{\max\{p^+, (q + \lceil s \rceil)^+\}}_{\mathcal{S}} \leq \varrho_{\mathcal{S}}(u) \leq \|u\|^{\min\{p^-, (q+ \lfloor s \rfloor )^-\}}_{\mathcal{S}}$.
		\item[\textnormal{(iv)}]
			If $\|u\|_{\mathcal{S}} > 1$, then $\|u\|^{\min\{p^-, (q+\lfloor s \rfloor)^-\}}_{\mathcal{S}} \leq \varrho_{\mathcal{S}}(u) \leq  \|u\|^{\max\{p^+, ( q + \lceil s \rceil)^+\}}_{\mathcal{S}}$.
		\item[\textnormal{(v)}]
			$ \|u_n\|_{\mathcal{S}} \to 0$ $($or $1$ or  $\infty)$ in $L^{\mathcal{S}}(\Omega)$ $\;\Leftrightarrow\; \varrho_{\mathcal{S}}(u_n) \to 0$ $($or $1$ or  $\infty)$.
	\end{enumerate}
\end{proposition}

\begin{proof}
	For any $u \in L^{\mathcal{S}}(\Omega)$, define
	\begin{align*}
		\xi_u\colon [0, \infty) \to \mathbb{R}, \quad \xi_u(\eta) :=\varrho_{\mathcal{S}}(\eta u).
	\end{align*}
	Let $\|u\|_{\mathcal{S}} = \eta$. Now, by using the definition of the norm, the continuity and the strict monotonicity of the map $\xi_u$ because of \eqref{main:assump}, we have for $\varepsilon \in (0, \eta)$
	\begin{align*}
		\xi_u\left(\frac{1}{\eta-\varepsilon}\right) > 1 \quad \text{ and } \quad \xi_u\left(\frac{1}{\eta} \right) \leq 1 \quad  \Longrightarrow \quad \xi_u\left(\frac{1}{\eta} \right) = 1 \quad \text{i.e.} \quad   \varrho_{\mathcal{S}} \left(\frac{u}{\eta}\right) =1.
	\end{align*}
	The converse part holds trivially from the definition of the norm and the strict monotonicity of the map $\xi_u$. Hence, (i) follows. Using the continuity of the function $\mathcal{S}$ and Harjulehto--H\"{a}st\"{o} \cite[Lemma 3.2.3]{Harjulehto-Hasto-2019}, we get (ii). Moreover, (iii) and (iv) follow from Proposition \ref{Prop:AbstractNormModular} and Lemma \ref{Le:Prop-S}. Moreover, from (ii)--(iv), for any $v \in L^{\mathcal{S}}(\Omega)$, we have
    \begin{align*}
    	&\min\bigg\{1, \|v\|^{\min\{p^-, (q+ \lfloor s \rfloor )^-\}}_{\mathcal{S}}, \|v\|^{\max\{p^+, (q + \lceil s \rceil)^+\}}_{\mathcal{S}} \bigg\}\\
    	&\leq \varrho_{\mathcal{S}}(v) \leq \max\bigg\{1, \|v\|^{\min\{p^-, (q+ \lfloor s \rfloor )^-\}}_{\mathcal{S}}, \|v\|^{\max\{p^+, (q + \lceil s \rceil)^+\}}_{\mathcal{S}} \bigg\}
    \end{align*}
    and
    \begin{align*}
    	&\min\bigg\{1, (\varrho_{\mathcal{S}}(v))^{1/\min\{p^-, (q+ \lfloor s \rfloor )^-\}}, (\varrho_{\mathcal{S}}(v))^{1/\max\{p^+, (q + \lceil s \rceil)^+\}} \bigg\}\\
    	&\leq  \|v\|_{\mathcal{S}} \leq \max\bigg\{1, (\varrho_{\mathcal{S}}(v))^{1/\min\{p^-, (q+ \lfloor s \rfloor )^-\}}, (\varrho_{\mathcal{S}}(v))^{1/\max\{p^+, (q + \lceil s \rceil)^+\}}\bigg\}.
    \end{align*}
	Finally, by taking $v=u_n$ in the above estimates, we obtain the required claim in (v).
\end{proof}

For our purposes, we further need to work on the associated Sobolev space, whose properties are summarized in the following statement. Its proof is completely analogous to the proof of Proposition \ref{pro:norm-mod:relation} except that now we use Proposition \ref{Prop:AbstractSobolev} (instead of Proposition \ref{Prop:AbstractBanach}) and Proposition \ref{Prop:AbstractoneNormModular} (instead of Proposition \ref{Prop:AbstractNormModular}).

\begin{proposition}\label{Prop:oneHlogModularNorm}
	Let \eqref{main:assump} be satisfied, Then the following hold:
	\begin{enumerate}
		\item[\textnormal{(i)}]
			If $u \neq 0$, then $\|u\|_{1,\mathcal{S}} = \eta$ if and only if $\varrho_{1,\mathcal{S}} \left(\frac{u}{\eta}\right)=1$.
		\item[\textnormal{(ii)}]
			$\varrho_{1,\mathcal{S}}(u) <1  \text{(resp.\,$=1$;\;$>1$) } \Leftrightarrow \; \|u\|_{1, \mathcal{S}}<1  \text{ (resp.\,$=1$;\; $> 1$)}$.
		\item[\textnormal{(iii)}]
			If $\|u\|_{1,\mathcal{S}} < 1$, then
			$\|u\|^{\max\{p^+, (q + \lceil s \rceil)^+\}}_{1,\mathcal{S}} \leq \varrho_{1, \mathcal{S}}(u) \leq \|u\|^{\min\{p^-, (q+ \lfloor s \rfloor )^-\}}_{1, \mathcal{S}}$.
		\item[\textnormal{(iv)}]
			If $\|u\|_{1,\mathcal{S}} > 1$, then $\|u\|^{\min\{p^-, (q+\lfloor s \rfloor)^-\}}_{1,\mathcal{S}} \leq \varrho_{1,\mathcal{S}}(u) \leq  \|u\|^{\max\{p^+, ( q + \lceil s \rceil)^+\}}_{1, \mathcal{S}}$.
		\item[\textnormal{(v)}]
			$u_n\to 0$ $($or $1$ or  $\infty)$ in $W^{1, \mathcal{S}}(\Omega)$ $\;\Leftrightarrow\; \varrho_{1, \mathcal{S}}(u_n) \to 0$ $($or $1$ or  $\infty)$ respectively as $n\to \infty$.
	\end{enumerate}
\end{proposition}

The next lemma will be useful to prove Sobolev embeddings of Musielak-Orlicz Sobolev spaces. We define
\begin{align*}
	\mathcal{B}_{\theta, \Theta, \Gamma}(x,t) = a(x) t^{\theta(x)} + b(x) t^{\Theta(x)} \log^{\Gamma(x)}(1+t).
\end{align*}
where $\theta, \Theta$ are positive continuous function on $\overline{\Omega}$ and $\Gamma$ is a bounded function on $\Omega$.

\begin{proposition}\label{prop:embedding}
	Let \eqref{main:assump} be satisfied. Then the embeddings
	\begin{align*}
		L^{\mathcal{S}}(\Omega) \hookrightarrow L^{\mathcal{B}_{p,j, m}}(\Omega) \hookrightarrow L^{\ell(\cdot)}(\Omega)
	\end{align*}
	are continuous, where $p,j, \ell \in C_+(\Omega)$, $m \in L^\infty(\Omega)$, $\ell(x) \leq \min\{p(x), j(x)\}$ for all $x \in \overline{\Omega}$,  and $j(\cdot)$ and $m(\cdot)$ are given by
	\begin{equation}\label{embd:assum-1}
		\begin{aligned}
			m(x) \leq s(x), \ j(x) + m(x) \geq 0 \quad \text{for a.a. } x \in \Omega \ \text{and} \ j(x) & =
			\begin{cases}
				q(x)   & \text{if }s(x) > 0, \\
				< q(x) & \text{otherwise},
			\end{cases}
			\quad \text{for all } x \in \overline{\Omega}.
		\end{aligned}
	\end{equation}
\end{proposition}

\begin{proof}
	We will prove the embeddings by applying Proposition \ref{Prop:AbstractEmbedding} to the corresponding $\Phi$-functions. For $j \in C(\overline{\Omega})$ and $m \in L^\infty(\Omega)$ satisfying \eqref{embd:assum-1}, we have
	\begin{equation}\label{embd:est-0}
		\begin{aligned}
			&\mathcal{B}_{p,j,m}\left(x,t\right)\\
			&=	a(x) t^{p(x)} + b(x) t^{j(x)} \log^{m(x)}(1+t) \\
			& \leq a(x) t^{p(x)} + b(x) t^{j(x)} \log^{m(x)}(1+t)  \chi_{\{s(x) \geq 0\}}(x) + b(x) t^{j(x)} \log^{m(x)}(1+t)  \chi_{\{s(x) < 0\}}(x).
		\end{aligned}
	\end{equation}
	Now, we estimate the terms in the right-hand side of the above inequality separately. Note that the condition $j(x) + m(x) \geq 0$ implies that $t^{j(x)} \log^{m(x)}(1+t)$ is an increasing function.\\
	\textbf{Case 1: $x \in \{ s(x) > 0\}$}

	It holds
	\begin{equation}\label{embd:est-1}
		\begin{aligned}
			t^{j(x)} \log^{m(x)}(1+t)
			& \leq
			\begin{cases}
				(e-1)^{j(x)} & \text{if } t \leq e-1,    \\
				t^{q(x)} \log^{s(x)}(1+t)  & \text{if } t \geq e-1,
			\end{cases} \\
			& \leq  (e-1)^{j^+} +    t^{q(x)} \log^{s(x)}(1+t),
		\end{aligned}
	\end{equation}
	where in the above inequalities we have used $m(x) \leq s(x)$ and $j(x) + m(x) \geq 0$ in $\{s(x) >0\}$.\\
	\textbf{Case 2: $x \in \{s(x) \leq 0\}$}\\
	By the continuity of $j(\cdot)$ and $q(\cdot)$, there exists $\varepsilon >0$ such that
	\begin{align*}
		t^{q(x)-j(x)} \geq t^\varepsilon \quad \text{for all }  t \geq 1  \text{ and for all }  x \in \{s(x) \leq 0\}.
	\end{align*}
	It is easy to show that for the above choice of $\varepsilon >0$ there exists $t_\ast = t_\ast(\varepsilon) > e-1$ (independent of $x$) such that
	\begin{equation}\label{embd:est-log}
		\log^{(-s)^+}(1+t) \leq t^\varepsilon \quad \text{for all }  t \geq t^\ast.
	\end{equation}
	Combining the above estimates, we deduce that
	\begin{equation}\label{est:lower:nega}
		t^{q(x)} \geq t^{j(x)} \log^{-s(x)}(1+t) \quad \text{for all }  t \geq t^\ast  \text{ and for a.a.\,} x \in \{x \in \Omega \colon  s(x) <0\}.
	\end{equation}
	Again, by using $m(x) \leq s(x)$ and $j(x) + m(x) \geq 0$  for a.a.\,$x \in \{x\in\Omega\colon  s(x) \leq 0\}$, we obtain
	\begin{equation}\label{embd:est-2}
		\begin{aligned}
			t^{j(x)} \log^{m(x)}(1+t)
			& \leq t^{j(x)} \log^{m(x)}(1+t) \chi_{\{t \leq t^\ast\}}(t)+ t^{j(x)} \log^{m(x)}(1+t) \chi_{\{t \geq t^\ast\}}(t)  \\
			& \leq \left((t^\ast)^{j^+} \log^{(-m)^+}(1+t^\ast)\right) + t^{q(x)} \log^{s(x)}(1+t).
		\end{aligned}
	\end{equation}
	Using the estimates \eqref{embd:est-1}, \eqref{embd:est-log}, \eqref{est:lower:nega} and \eqref{embd:est-2} in \eqref{embd:est-0}, it follows that
	\begin{align*}
			\mathcal{B}_{p,j,m}\left(x,t\right) & \leq \mathcal{S}(x,t) + h(x),
	\end{align*}
	where
	\begin{align*}
		h(x) :=  \log^{(-m)^+}(1+t^\ast) (t^\ast)^{j^+}  + (e-1)^{j^+}.
	\end{align*}
	This concludes the proof.
\end{proof}

As a consequence of the previous result, we have the following embeddings.

\begin{proposition}\label{sobolevembed:1}
	Under the assumption \eqref{main:assump} the following embeddings hold:
	\begin{enumerate}
		\item[\textnormal{(i)}]
			$W^{1, \mathcal{S}}(\Omega) \hookrightarrow W^{1, \mathcal{B}_{p,j, m}}(\Omega) \hookrightarrow W^{1, \ell(\cdot)}(\Omega)$, $W_0^{1, \mathcal{S}}(\Omega) \hookrightarrow W_0^{1, \mathcal{B}_{p,j,m}}(\Omega) \hookrightarrow W_0^{1, \ell(\cdot)}(\Omega)$ are continuous with $1 \leq \ell(x) \leq \min\{p(x), j(x)\}$ for all $x \in \overline{\Omega}$ and $j(\cdot)$ and $m(\cdot)$ are given by \eqref{embd:assum-1}.
		\item[\textnormal{(ii)}] if $\min\{p(\cdot), j(\cdot)\} \in C^{0, \frac{1}{|\log t|}}(\overline{\Omega})$,
			$W^{1, \mathcal{S}}(\Omega) \hookrightarrow L^{ \ell^\ast(\cdot)}(\Omega)$, $W_0^{1, \mathcal{S}}(\Omega) \hookrightarrow L^{\ell^\ast(\cdot)}(\Omega)$ are continuous for all $\ell \in C(\overline{\Omega})$ with $1 \leq \ell(x) \leq \min\{p(x), j(x)\}$ for all $x \in \overline{\Omega}$.
		\item[\textnormal{(iii)}]
			$W^{1, \mathcal{S}}(\Omega) \hookrightarrow L^{ \ell^\ast(\cdot)}(\Omega)$, $W_0^{1, \mathcal{S}}(\Omega) \hookrightarrow L^{\ell^\ast(\cdot)}(\Omega)$ are compact for all $\ell \in C(\overline{\Omega})$ with $1 \leq \ell(x) < \min\{p(x), j(x)\}$ for all $x \in \overline{\Omega}$.
	\end{enumerate}
\end{proposition}

\begin{proof}
	The embeddings are straightforward to verify via Proposition \ref{prop:embedding} and the Sobolev embeddings in variable exponent spaces by finding a weaker $\Phi$ function and comparing the variable exponents in the modular function.
\end{proof}

Under some additional conditions, we prove the following Poincar\'{e} inequality in $W_0^{1, \mathcal{S}}(\Omega)$.

\begin{proposition}\label{Prop:Poincare}
	Let \eqref{main:assump} be satisfied and suppose
	\begin{equation}\label{poin:cond}
		a, b \in L^\infty(\Omega) \quad \text{and} \quad \max\{p(x), q(x)\} < \min\{p^\ast(x), q^{\ast}(x)\} \quad \text{for all }  x \in \overline{\Omega}.
	\end{equation}
	Then $ W^{1, \mathcal{S}}(\Omega) \hookrightarrow L^{\mathcal{S}}(\Omega)$ is a compact embedding and there exists a constant $\mathcal{S}>0$ such that
	\begin{align*}
		\|u\|_{\mathcal{S}} \leq \mathtt{S} \|\nabla u\|_{\mathcal{S}}  \ \text{for all }  u \in W_0^{1, \mathcal{S}}(\Omega),
	\end{align*}
	where $\mathtt{S}$ is independent of $u$ and so $\|\nabla \cdot\|_{\mathcal{S}}$ is an equivalent norm on $W_0^{1, \mathcal{S}}(\Omega)$.
\end{proposition}

\begin{proof}
	Note that $\mathcal{S}(x, \cdot)$ is an increasing and continuous function for a.a.\,$x \in \Omega$ and for $\varepsilon >0$ there exists $t_\ast = t_\ast(\varepsilon)$ (independent of $x$) such that
	\begin{equation}\label{log:est}
		\log^{s(x)}(1+t) \leq \log^{s^+}(1+t) \leq t^\varepsilon \quad \text{for all } t \geq t^\ast \  \text{and for a.a.\,}  x \in \{x \in \Omega \colon  s(x) \geq 0\}.
	\end{equation}
	This directly implies that
	\begin{equation}\label{poin-est-1}
		\mathcal{S}(x,t) \leq C_1 + a(x) t^{p(x)} + b(x) t^{q(x) + \varepsilon} \quad \text{for a.a.\,} x \in \{x\in \Omega \colon s(x) \geq 0\}
	\end{equation}
	and
	\begin{equation}\label{poin-est-2}
		\begin{aligned}
			\mathcal{S}(x,t)
			& = \mathcal{S}(x,t) \chi_{\{t < e-1\}}(t) + \mathcal{S}(x,t) \chi_{\{t \geq e-1\}}(t)    \\
			& \leq C_2 + a(x) t^{p(x)} + b(x) t^{q(x)} \quad \text{for a.a }  x \in \{x \in \Omega \colon  s(x) < 0\}.
		\end{aligned}
	\end{equation}
	where $C_1$ and $C_2$ depend on $t^\ast$ and $L^\infty$-norms of $q,b$ and $s$. The uniform continuity of $p$ and $q$ and the sharp inequality in \eqref{poin:cond} gives the choice of $\varepsilon$ in \eqref{log:est} such that
	\begin{equation}\label{log-e-est}
		\max\{p(x), q(x)\} + 2 \varepsilon< \min\{p^\ast(x), q^{\ast}(x)\} \quad \text{for all }  x \in \overline{\Omega}.
	\end{equation}
	Now, by combining the estimates in \eqref{poin-est-1} and \eqref{poin-est-2} with the inequality in \eqref{log-e-est}, we obtain
	\begin{align*}
		\mathcal{S}(x,t) \leq C_3 t^{\min\{p^\ast(x), q^{\ast}(x)\} -\varepsilon} + C_4.
	\end{align*}
	Finally, by applying Proposition \ref{Prop:AbstractEmbedding} and Proposition \ref{sobolevembed:1} $\text{(iii)}$, we obtain the required embedding.
\end{proof}

\section{Musielak-Orlicz Sobolev embeddings}\label{Section-3}

For a parameter $\ell \geq 1$, we denote
\begin{align}\label{perturbed:function}
	\hat{\mathcal{S}}(x,t)=
	\begin{cases}
		\frac{t\mathcal{S}(x,\ell) }{\ell} & \text{if $ 0\leq t \leq \ell$}, \\
		\mathcal{S}(x,t)& \text{if $t >\ell$}.
	\end{cases}
	\quad \text{for all }(x,t) \in \Omega \in [0, \infty).
\end{align}
Since $\Omega \subset \mathbb{R}^N$ is bounded and $\hat{\mathcal{S}}(x,t)= \mathcal{S}(x,t)$ for all $(x,t) \in \Omega \times [\ell, \infty)$, we get $L^{\mathcal{S}}(\Omega)= L^{\hat{\mathcal{S}}}(\Omega)$ and $W^{1, \mathcal{S}}(\Omega)= W^{1, \hat{\mathcal{S}}}(\Omega)$ and their norms are comparable. Therefore, in light of the embedding results, we may use $\hat{\mathcal{S}}$ in place of $\mathcal{S}$. For the sake of brevity, we write $\mathcal{S}$ instead of $\hat{\mathcal{S}}$. Now, we define the Sobolev conjugate function of $\mathcal{S}$ in order to study the continuous and compact embeddings of certain Musielak-Orlicz Sobolev spaces into Musielak-Orlicz Lebesgue spaces.

\begin{definition}
	A function $\mathcal{S}_*$ is called the Sobolev conjugate function of $\mathcal{S}$ if the inverse of the Sobolev conjugate function $\mathcal{S}_*^{-1} \colon  \Omega \times [0, \infty) \to [0, \infty)$ is given by
	\begin{equation}\label{inverse:conj:def}
		\mathcal{S}_*^{-1}(x,s)= \int_0^s \frac{{\mathcal{S}}^{-1}(x, \tau)}{\tau^{\frac{N+1}{N}}} \,\mathrm{d}\tau \quad \text{for all } (x,s) \in \Omega \times [0, \infty),
	\end{equation}
	where $\mathcal{S}_*\colon  (x, t) \in \Omega \times [0, \infty) \to s \in [0, \infty)$ is such that $\mathcal{S}_*^{-1}(x,s)=t$ and $\mathcal{S}^{-1}(x, \cdot)\colon  [0, \infty) \to [0, \infty)$ is the inverse function of $\mathcal{S}(x, \cdot)$ for all $x \in \Omega$.
\end{definition}

In addition to the assumption \eqref{main:assump}, we suppose the following regularity and oscillation conditions:
\begin{enumerate}[label=\textnormal{(H$_1$)},ref=\textnormal{H$_1$}]
	\item\label{main:assump-1}
		\begin{enumerate}
			\item[\textnormal{(i)}]
				$p, q, s \in C^{0,1}(\overline{\Omega})$ and $a, b \in C^{0,1}(\overline{\Omega})$;
			\item[\textnormal{(ii)}]
				$ \frac{\max\{p(x),q(x)\}}{\min\{p(x), q(x)\}} < 1 + \frac{1}{N}$ in $\overline{\Omega}$.
		\end{enumerate}
\end{enumerate}

Now, we prove some continuous and compact embedding results using ideas from the papers by Cianchi--Diening \cite{Cianchi-Diening-2024}, Colasuonno--Perera \cite{Colasuonno-Squassina-2016}, Crespo-Blanco--Gasi\'{n}ski--Harjulehto--Winkert \cite{Crespo-Blanco-Gasinski-Harjulehto-Winkert-2022}, Fan \cite{Fan-2012}, and Ho--Winkert \cite{Ho-Winkert-2023}.

\begin{proposition}\label{pro:embed}
	Let \eqref{main:assump} and \eqref{main:assump-1} be satisfied. Then, the following hold:
	\begin{enumerate}
		\item[\textnormal{(i)}]
			$W^{1, \mathcal{S}}(\Omega) \hookrightarrow L^{\mathcal{S}_{*}}(\Omega)$ is a continuous embedding.
		\item[\textnormal{(ii)}]
			Suppose $\mathcal{L} \in {\bf \Phi}(\Omega)$ and $\mathcal{L} \ll \mathcal{S}_\ast$, then there exists a compact embedding  $W^{1, \mathcal{S}}(\Omega) \hookrightarrow L^{\mathcal{L}}(\Omega)$.
	\end{enumerate}
\end{proposition}

\begin{proof}
	The proof follows from Fan  \cite[Theorem 1.1 and 1.2]{Fan-2012} provided $\mathcal{S}$ satisfy the condition (P$_4$) and condition (2) in \cite[Proposition 3.1]{Fan-2012}, i.e.
	\begin{align*}
		\lim_{t \to \infty} \mathcal{S}_{*}^{-1}(x,t) = \infty \quad \text{for all }  x \in \Omega
	\end{align*}
	and there exist positive constants $c,t_0$ and $\delta<\frac{1}{N}$ such that for each $j=1,2, \dots, N$
	\begin{align*}
		\left| \frac{\partial \mathcal{S}(x,t)}{\partial x_j}\right| \leq c \left(\mathcal{S}(x,t)\right)^{1+ \delta}  \text{for all }  x \in \Omega  \text{ and for all } t \geq t_0
	\end{align*}
	for which $\nabla a(x), \nabla b(x), \nabla p(x), \nabla q(x), \nabla s(x)$ exist and so $\left| \frac{\partial \mathcal{S}(x,t)}{\partial x_j}\right|$ does.

	For $(x,t) \in \Omega \times [1, \infty)$, note that for any $\varepsilon>0$ and $x \in \Omega$,
	\begin{align*}
		\lim_{t \to \infty} \frac{a(x) t^{p(x)}+b(x)t^{q(x)} \log^{s(x)}(1 + t)}{t^{\beta(x) + \varepsilon}} = 0 \quad \text{with } \beta(x)= \max\{p(x), q(x)\},
	\end{align*}
	that is, for any $\eta >0$, there exists a constant $m_1>0$ (independent of $x$) such that
	\begin{align*}
		\mathcal{S}(x,t)= a(x) t^{p(x)}+b(x) t^{q(x)}\log^{s(x)}(1 + t) \leq \eta t^{\beta(x)+ \varepsilon} \quad \text{for all }  x \in \Omega  \text{ and for all } t \geq m_1.
	\end{align*}
	Replacing $t$ by $\mathcal{S}^{-1}(x,t)$ in the above inequality gives
	\begin{align*}
		\mathcal{S}^{-1}(x, t) \geq \left(\frac{t}{\eta }\right)^{\frac{1}{\beta(x) + \varepsilon}} \quad \text{for all }  x \in \Omega \text{ and for all } t \geq \mathcal{S}(x, m_1).
	\end{align*}
	Hence, by choosing $\varepsilon$ small enough such that $\beta(x) + \varepsilon < N$ for all $x \in \overline{\Omega}$, we obtain
	\begin{align*}
		\mathcal{S}_{*}^{-1}(x,t) & \geq \mathcal{S}_{*}^{-1}(x,\mathcal{S}(x, m_1)) + \frac{1}{\eta ^{\frac{1}{\beta(x) +\varepsilon}}}\int_{\mathcal{S}(x, m_1)}^t \tau^{\frac{1}{\beta(x)+\varepsilon} - \frac{N+1}{N}} \,\mathrm{d}\tau \\
		& \geq C(\eta, \beta^\pm, N) \left(t^{\frac{1}{\beta(x)+\varepsilon} - \frac{1}{N}} - \left(\mathcal{S}(x, m_1)\right)^{\frac{1}{\beta(x)+\varepsilon} - \frac{1}{N}}\right) \to \infty \quad \text{as } t \to \infty.
	\end{align*}
	For $\zeta>0$ there exists $t_0= t_0(\zeta) \gg 1$ large enough such that
	\begin{equation}\label{log:growth}
		1 \leq \log(\log(1+t))\leq \log(t) \leq \log(1+t) \leq c_0 t^\frac{\zeta}{\max\{2, 2 \|s\|_\infty\}} \leq c_0 t^\frac{\zeta}{2} \quad \text{for }  t \geq t_0.
	\end{equation}
	Differentiating the function $\mathcal{S}$ with respect to $x_j$, we get
	\begin{equation}\label{est:deri}
		\begin{aligned}
			\left| \frac{\partial \mathcal{S}(x,t)}{\partial x_j}\right| & \leq a(x) t^{p(x)} \left| \frac{\partial p}{\partial x_j}\right| \log(t) + \left| \frac{\partial a}{\partial x_j}\right| t^{p(x)}  + \left| \frac{\partial b}{\partial x_j}\right| t^{q(x)} \log^{s(x)}(1+t) \\
			& \qquad + b(x) t^{q(x)} \log(t) \left| \frac{\partial q}{\partial x_j}\right| \log^{s(x)}(1+t)\\
			&\qquad + b(x) t^{q(x)} |\log(\log(1+t))| \left| \frac{\partial s }{\partial x_j}\right| \log^{s(x)}(1+t).
		\end{aligned}
	\end{equation}
	Now by using \eqref{log:growth} multiple times in \eqref{est:deri} for $t \geq t_0 > 1$, we obtain
	\begin{equation}\label{embed:est-5}
		\begin{aligned}
			\left| \frac{\partial \mathcal{S}(x,t)}{\partial x_j}\right| & \leq C_1 \left(t^{p(x) + \frac{\zeta}{2}} + t^{p(x)} \right) + C_2
			\begin{cases}
				t^{q(x)+ \frac{\zeta}{2}} \log^{s^+}(1+t) & \text{if $s(x) \geq 0$} \\
				t^{q(x)+ \frac{\zeta}{2}}\log^{s^-}(1+t) & \text{if $s(x) <0$}
			\end{cases}\\
			& \leq C_3
			\begin{cases}
				t^{\max\{p(x), q(x)\} + \zeta } & \text{if    $ s(x) \geq 0$} \\
				t^{\max\{p(x), q(x)\} + \zeta}  & \text{if $s(x) <0$}
			\end{cases} \\
			& \leq C_4 \left(a(x) t^{p(x)} + b(x) t^{q(x)} + b(x) t^{q(x)} \right)^{\textstyle \frac{\max\{p(x), q(x)\}+ \zeta}{\min\{p(x), q(x)\}}}.
		\end{aligned}
	\end{equation}
	Now, by using the uniform continuity of the functions $p(\cdot)$, $q(\cdot)$ in $\overline{\Omega}$ and \eqref{main:assump-1} \text{(ii)}, we can find $\zeta= \zeta(\delta)$ such that
	\begin{align}\label{embed:est-6}
		\frac{\max\{p(x), q(x)\}+\zeta}{\min\{p(x), q(x)\}} < 1+ \delta < 1+ \frac{1}{N}  \quad \text{in } \overline{\Omega}.
	\end{align}
	Using \eqref{embed:est-6} in \eqref{embed:est-5}, we get
	\begin{align*}
		\left| \frac{\partial \mathcal{S}(x,t)}{\partial x_j}\right| & \leq C \left(\mathcal{S}(x,t)\right)^{1+\delta} \quad \text{for } (x,t) \in \overline{\Omega} \times (t_0, \infty).
	\end{align*}
	This shows the assertions of the proposition.
\end{proof}

\begin{remark}
	With a similar argumentation as in Arora--Crespo-Blanco-Winkert \cite[Lemma 3.10 and Theorem 3.12]{Arora-Crespo-Blanco-Winkert-2023}, it can be concluded that the function $\mathcal{S}$ satisfies \textnormal{(A0)} and \textnormal{(A1)}, provided the conditions
    \begin{equation}\label{assumption:stronger}
        b \in C^{0, \gamma}(\overline{\Omega}) \quad \text{and} \quad \left(\frac{q}{p}\right)_+ < 1 + \frac{\gamma}{N}
    \end{equation}
    hold true with $a(x) \equiv 1$, $1 < p(x) < q(x)$, and $s(x) \equiv 1$. Consequently, this embedding result is a specific instance of Cianchi--Diening \cite[Theorem 3.5]{Cianchi-Diening-2024}, assuming the more stringent conditions stated in \eqref{assumption:stronger} compared to those in \eqref{main:assump} and \eqref{main:assump-1}.
\end{remark}

Next, we provide estimates for the Sobolev conjugate function $\mathcal{S}_\ast$ of $\mathcal{S}$, which is crucial for studying the concentration-compactness principle in the next section. Specifically, we intend to look for a function $\mathcal{S}^\ast$ weaker than $\mathcal{S}_\ast$ and of polynomial growth or polynomial growth perturbed with a logarithmic function governed by the support of the modulating coefficients $a(\cdot)$ and $b(\cdot)$. In the following we assume that \eqref{main:assump} and \eqref{main:assump-1} hold.

\begin{proposition}\label{prop:esti-near0}
	Let $S_\ast$ be the Sobolev conjugate function of $\mathcal{S}$ given by \eqref{inverse:conj:def}. Then, the following hold:
	\begin{enumerate}
		\item[\textnormal{(i)}]
			For $(x,s) \in \overline{\Omega} \times [0, \mathcal{S}(x,\ell)]$
			\begin{align*}
				\mathcal{S}^{-1}(x,s) = \frac{s  \ell}{\mathcal{S}(x,\ell)} \quad \text{ and } \quad \mathcal{S}_\ast^{-1}(x,s)= \frac{N \ell}{(N-1) \mathcal{S}(x,\ell)} s^{\frac{N-1}{N}}.
			\end{align*}
		\item[\textnormal{(ii)}]
			For $(x,t) \in \overline{\Omega} \times [0,\mathcal{S}_\ast^{-1}(x,\mathcal{S}(x,\ell))]$
			\begin{align*}
				\mathcal{S}_\ast(x,t) = \left(\frac{\mathcal{S}(x,\ell) (N-1)}{N \ell}\right)^\frac{N}{(N-1)} t^{\frac{N}{N-1}}
			\end{align*}
			where $\mathcal{S}_\ast^{-1}(x,\mathcal{S}(x,\ell)) = \frac{N \ell}{N-1} (\mathcal{S}(x,\ell))^{\frac{-1}{N}}$.
	\end{enumerate}
\end{proposition}

\begin{proof}
	From the definition of the Sobolev conjugate function $\mathcal{S}_\ast$ of $\mathcal{S}$ defined in \eqref{perturbed:function}, we get
	\begin{align*}
		\mathcal{S}^{-1}(x,s) = \frac{s  \ell}{\mathcal{S}(x,\ell)} \quad \text{ and } \quad \mathcal{S}_\ast^{-1}(x,s)= \frac{N \ell}{(N-1) \mathcal{S}(x,\ell)} s^{\frac{N-1}{N}} \quad \text{for }(x,s) \in \overline{\Omega} \times [0, \mathcal{S}(x,\ell)].
	\end{align*}
	By taking the inverse of the function $\mathcal{S}_\ast^{-1}$, we get
	\begin{align*}
		\mathcal{S}_\ast(x,t) = \left(\frac{\mathcal{S}(x,\ell) (N-1)}{N \ell}\right)^\frac{N}{(N-1)} t^{\frac{N}{N-1}} \quad \text{for }(x,t) \in \overline{\Omega} \times [0,\mathcal{S}_\ast^{-1}(x,\mathcal{S}(x,\ell))].
	\end{align*}
\end{proof}

Now, we prove a series of results offering the lower estimates of the Sobolev conjugate function $\mathcal{S}_\ast$ in the components $\{A_i\}_{i=1,2} \subset \Omega$ where $A_i$, $i=1,2$ is defined as
\begin{align*}
	A_1:= \{x \in \Omega\colon  a(x) \neq 0\}
	\quad\text{ and } \quad
	A_2:= \{x \in \Omega\colon  b(x) \neq 0\} .
\end{align*}

\begin{proposition}
	Let $S_\ast$ be the Sobolev conjugate function given by \eqref{inverse:conj:def}. Then, for $(x,t) \in A_1 \times \left[\mathcal{S}_\ast^{-1}(x,\mathcal{S}(x,\ell)), \infty\right)$, we have
	\begin{align*}
		\mathcal{S}_\ast(x,t) & \geq \left[\frac{a(x)^\frac{1}{p(x)}}{p^\ast(x)} \left(t+ (\mathcal{S}(x,\ell))^{\frac{-1}{N}}\left(p^\ast(x) \left(\frac{\mathcal{S}(x,\ell)}{a(x)}\right)^{\frac{1}{p(x)}} - \frac{N \ell}{N-1}\right)\right)\right]^{p^\ast(x)} \geq \left[ \frac{a(x)^\frac{1}{p(x)}}{p^\ast(x)} t \right]^{p^\ast(x)},
	\end{align*}
	where equality holds in the region $A_1^0 \subset A_1$ with $A_1^0:= \{x \in \Omega\colon  b(x) =0\}$.
\end{proposition}

\begin{proof}
	From \eqref{perturbed:function}, we know that for $(x,t) \in \overline{\Omega} \times [\ell, \infty)$, the function $\mathcal{S}$ is given by
	\begin{align*}
		\mathcal{S}(x,t) = a(x) t^{p(x)} + b(x) t^{q(x)} \log^{s(x)}(1+ t).
	\end{align*}
	By taking $t= \mathcal{S}^{-1}(x,y)$ for $y \geq \mathcal{S}(x,\ell)$ and $x \in \overline{\Omega}$, we get
	\begin{equation}\label{inverse:est-1}
		y= a(x) \left(\mathcal{S}^{-1}(x,y)\right)^{p(x)} + b(x) \left(\mathcal{S}^{-1}(x,y)\right)^{q(x)} \log^{s(x)} (1+ \left(\mathcal{S}^{-1}(x,y)\right)
	\end{equation}
	and
	\begin{align*}
		\mathcal{S}^{-1}(x, y) \geq \ell \quad \text{for all } (x,y) \in \overline{\Omega} \times [\mathcal{S}(x, \ell), \infty).
	\end{align*}
	The non-negativity of the modulating coefficient $b(\cdot)$ implies
	\begin{equation}\label{inverse:est-2}
		\mathcal{S}^{-1}(x,y) \leq \left(\frac{y}{a(x)}\right)^\frac{1}{p(x)} \quad \text{for }
		(x, y) \in A_1 \times [\mathcal{S}(x,\ell), \infty),
	\end{equation}
	where equality holds in $A_1^0 \times [\mathcal{S}(x,\ell), \infty)$. Now, by using \eqref{inverse:est-2} and Proposition \ref{prop:esti-near0} in the definition of the inverse of the Sobolev conjugate function $\mathcal{S}_\ast$ for $(x,z) \in A_1 \times [\mathcal{S}(x, \ell), \infty)$, we obtain
	\begin{align*}
		\mathcal{S}_*^{-1}(x,z) & =  \mathcal{S}_*^{-1}(x,\mathcal{S}(x,\ell)) + \int_{\mathcal{S}(x,\ell)}^z \frac{{\mathcal{S}}^{-1}(x, y)}{y^{\frac{N+1}{N}}} \,\mathrm{d}y \\
		& \leq \mathcal{S}_*^{-1}(x,\mathcal{S}(x,\ell)) + \int_{\mathcal{S}(x,\ell)}^z  \left(\frac{y}{a(x)}\right)^\frac{1}{p(x)} \frac{1}{y^{1+ \frac{1}{N}}} \,\mathrm{d}y \\
		& = \frac{p^\ast(x)}{(a(x))^\frac{1}{p(x)}} z^{\frac{1}{p^\ast(x)}} + \frac{1}{(\mathcal{S}(x, \ell))^\frac{1}{N}} \left[\frac{N \ell}{N-1} - p^\ast(x) \left(\frac{\mathcal{S}(x, \ell)}{a(x)}\right)^\frac{1}{p(x)}\right].
	\end{align*}
	Replacing $z= \mathcal{S}_\ast(x, \tau)$ for $x \in A_1$ and $\tau \geq \mathcal{S}_\ast^{-1}(x, \mathcal{S}(x, \ell))$, we get
	\begin{align*}
		\tau \leq  \frac{p^\ast(x)}{(a(x))^\frac{1}{p(x)}} \left(\mathcal{S}_\ast(x, \tau) \right)^{\frac{1}{p^\ast(x)}} + \frac{1}{(\mathcal{S}(x, \ell))^\frac{1}{N}} \left[\frac{N \ell}{N-1} - p^\ast(x) \left(\frac{\mathcal{S}(x, \ell)}{a(x)}\right)^\frac{1}{p(x)}\right].
	\end{align*}
	This further gives
	\begin{align*}
		\mathcal{S}_\ast(x,\tau) & \geq \left[\frac{a(x)^\frac{1}{p(x)}}{p^\ast(x)} \left(\tau+ (\mathcal{S}(x,\ell))^{\frac{-1}{N}}\left(p^\ast(x) \left(\frac{\mathcal{S}(x,\ell)}{a(x)}\right)^{\frac{1}{p(x)}} - \frac{N \ell}{N-1}\right)\right)\right]^{p^\ast(x)} \\
		& \geq \left[ \frac{a(x)^\frac{1}{p(x)}}{p^\ast(x)} \tau \right]^{p^\ast(x)},
	\end{align*}
	where the last inequality follows from the fact that  $\mathcal{S}(x, \ell) \geq a(x) \ell^{p(x)}$ and $p^\ast(x) \geq \frac{N}{N-1}$ for $ x \in A_1$.
\end{proof}

The non-negativity of the modulating coefficient $a(\cdot)$ in \eqref{inverse:est-1} and $\mathcal{S}^{-1}(x, y) \geq \ell \geq 1$ for $(x, y) \in A_2 \times [\mathcal{S}(x,\ell), \infty)$ gives
\begin{align}\label{inverse:conju:1}
	\mathcal{S}^{-1}(x,y) \leq \left(\frac{y}{b  (x) \log^{s(x)}(1+ \mathcal{S}^{-1}(x,y))}\right)^\frac{1}{q(x)}.
\end{align}
Now, depending upon the sign of $s(\cdot)$, we further partition the set $A_2$ into two disjoint components and estimate the Sobolev conjugate function over the following disjoint components
\begin{align*}
	A_2^{(1)}:=  A_2 \cap \{s(x) >0\} \quad \text{and} \quad A_2^{(2)}:=  A_2 \cap \{s(x) \leq 0\}
\end{align*}

\begin{proposition}
	Let $S_\ast$ be the Sobolev conjugate function given by \eqref{inverse:conj:def}. Then, there exist $l_0 \gg 1$ and $C>0$ such that for $\ell \geq \ell_0$ and for $x \in A_2^{(1)}$ the following estimates hold true:
	\begin{enumerate}
		\item[\textnormal{(i)}]
			$\mathcal{S}^{-1}(x,y) \leq \left(\frac{y}{b(x) }\right)^\frac{1}{q(x)} \log^{\frac{-s(x)}{q(x)}}\left(1+ \left(\frac{y}{h(x)}\right)^{\frac{1}{h_1(x)}}\right)$ for any $y \in [\mathcal{S}(x,\ell), \infty)$ where
			\begin{align*}
				h(x):= a(x) + b(x) \quad \text{ and } \quad h_1(x):= \max\{p(x), q(x) + \varepsilon\}, \varepsilon >0.
			\end{align*}
		\item[\textnormal{(ii)}]
			$\mathcal{S}_*^{-1}(x,z) \leq \frac{q^\ast(x)}{\zeta (b(x))^\frac{1}{q(x)}} \frac{z^\frac{1}{q^\ast(x)}}{\log^{\frac{s(x)}{q(x)}}\left(1+ \left(\frac{z}{h(x)}\right)^{\frac{1}{h_1(x)}}\right)} + (\mathcal{S}(x,\ell))^{\frac{-1}{N}} \ell \left(\frac{N}{N-1} -  \frac{q^\ast(x)}{\zeta} \right)$ for any $\zeta \in (0,1)$ and $z \in  [\mathcal{S}(x, \ell), \infty).$
		\item[\textnormal{(iii)}]
			$\mathcal{S}_\ast(x,t) \geq  C \left(\frac{(b(x))^\frac{1}{q(x)}}{ q^\ast(x) } \log^{\frac{s(x)}{q(x)}}\left(1+t \right)\right)^{q^\ast(x)} t^{q^\ast(x)}$ for any $t \in \left[\mathcal{S}_\ast^{-1}(x,\mathcal{S}(x,\ell)), \infty \right).$
	\end{enumerate}
\end{proposition}

\begin{proof}
	For any $\delta>0$, we know that $\log(1+t) \leq t^\delta$ for $t \geq t_0(\delta)$. Then, by taking $\delta = \frac{\varepsilon}{\lceil s \rceil ^+}$ for some $\varepsilon>0$ and $\ell$ large enough such that $\ell \geq \ell_0(\varepsilon, |s|^+) \geq 1$ gives that
	\begin{align*}
		\log^{s(x)}(1+\mathcal{S}^{-1}(x,y)) \leq (\mathcal{S}^{-1}(x,y))^{\varepsilon \frac{ s(x)}{\lceil s \rceil ^+}} \leq (\mathcal{S}^{-1}(x,y))^{\varepsilon},
	\end{align*}
	which further implies
	\begin{equation}\label{inverse:conju:2}
		y \leq a(x) \left(\mathcal{S}^{-1}(x,y)\right)^{p(x)} + b(x)\left(\mathcal{S}^{-1}(x,y)\right)^{q(x) + \varepsilon} \leq h(x) \left(\mathcal{S}^{-1}(x,y)\right)^{h_1(x)}.
	\end{equation}
	Now, by using the fact that $h(x) >d$ for all $x \in \overline{\Omega}$ in \eqref{inverse:conju:2}, we get
	\begin{equation}\label{inverse:log:est}
		\l(\frac{y}{h(x)}\r)^\frac{1}{h_1(x)} \leq \mathcal{S}^{-1}(x,y) \quad \text{and} \quad \log^{s(x)}\left(1+ \left(\frac{y}{h(x)}\right)^{\frac{1}{h_1(x)}}\right) \leq \log^{s(x)}(1+\mathcal{S}^{-1}(x,y)).
	\end{equation}
	Combining \eqref{inverse:conju:1} and \eqref{inverse:log:est} for $(x, y) \in A_2^{(1)} \times [\mathcal{S}(x,\ell), \infty)$  it follows that
	\begin{equation}\label{inverse:log:est:0}
		\mathcal{S}^{-1}(x,y) \leq \left(\frac{y}{b(x) }\right)^\frac{1}{q(x)} \log^{\frac{-s(x)}{q(x)}}\left(1+ \left(\frac{y}{h(x)}\right)^{\frac{1}{h_1(x)}}\right).
	\end{equation}
	Using \eqref{inverse:log:est:0} and Proposition \ref{prop:esti-near0} in the definition of the inverse of the Sobolev conjugate function $\mathcal{S}_\ast$ for $(x,z) \in A_2^{(1)} \times [\mathcal{S}(x, \ell), \infty)$, we have
	\begin{equation}\label{inverse:conju:3}
		\begin{aligned}
			\mathcal{S}_*^{-1}(x,z) & \leq \mathcal{S}_*^{-1}(x,\mathcal{S}(x,\ell)) + \frac{1}{(b(x))^\frac{1}{q(x)}} \int_{\mathcal{S}(x,\ell)}^z  y^{\frac{1}{q^\ast(x)} -1} \log^{\frac{-s(x)}{q(x)}}\left(1+ \left(\frac{y}{h(x)}\right)^{\frac{1}{h_1(x)}}\right) \,\mathrm{d}y \\
			& : = \mathcal{S}_*^{-1}(x,\mathcal{S}(x,\ell)) + \frac{1}{(b(x))^\frac{1}{q(x)}} g(x,z).
		\end{aligned}
	\end{equation}
	By applying integration by parts formula for $(x,z) \in A_2^{(1)} \times [\mathcal{S}(x, \ell), \infty)$, we estimate
	\begin{align*}
		g(x,z) & \leq   q^\ast(x) \left[ \frac{z^\frac{1}{q^\ast(x)}}{\log^{\frac{s(x)}{q(x)}}\left(1+ \left(\frac{z}{h(x)}\right)^{\frac{1}{h_1(x)}}\right)} - \frac{(\mathcal{S}(x ,\ell))^\frac{1}{q^\ast(x)}}{\log^{\frac{s(x)}{q(x)}}\left(1+ \left(\frac{\mathcal{S}(x ,\ell)}{h(x)}\right)^{\frac{1}{h_1(x)}}\right)}\right]   \\
		& \qquad + \frac{s(x) q^\ast(x)}{q(x) h_1(x)}\int_{\mathcal{S}(x,\ell)}^z  y^{\frac{1}{q^\ast(x)} -1} \frac{\left(\frac{y}{h(x)}\right)^\frac{1}{h_1(x)}}{1+ \left(\frac{y}{h(x)}\right)^\frac{1}{h_1(x)}} \log^{-\left(\frac{s(x)}{q(x)} +1 \right)}\left(1+ \left(\frac{y}{h(x)}\right)^{\frac{1}{h_1(x)}}\right) \,\mathrm{d}y \\
		& \leq   q^\ast(x) \left[ \frac{z^\frac{1}{q^\ast(x)}}{\log^{\frac{s(x)}{q(x)}}\left(1+ \left(\frac{z}{h(x)}\right)^{\frac{1}{h_1(x)}}\right)} - \frac{(\mathcal{S}(x ,\ell))^\frac{1}{q^\ast(x)}}{\log^{\frac{s(x)}{q(x)}}\left(1+ \left(\frac{\mathcal{S}(x ,\ell)}{h(x)}\right)^{\frac{1}{h_1(x)}}\right)}\right]   \\
		& \qquad + \frac{s(x) q^\ast(x)}{q(x) h_1(x) \log\left(1+ \left(\frac{\mathcal{S}(x, \ell)}{h(x)}\right)^{\frac{1}{h_1(x)}}\right)} g(x,z).
	\end{align*}
	Taking $\ell \gg 1$ such that $\ell \geq \l(e^\frac{s^+ (q^\ast)^+}{r^- h_1^- (1-\zeta)} -1\r)^{h_1^+}$ (where the precise value of $\zeta \in (0,1)$ is chosen later), we get
	\begin{align*}
		\frac{s(x) q^\ast(x)}{q(x) h_1(x)} \leq \frac{s^+ (q^\ast)^+}{q^- h_1^-} \leq (1-\zeta) \log\left(1+ \left(\ell\right)^{\frac{1}{h_1^+}}\right) \leq (1-\zeta) \log\left(1+ \left(\frac{\mathcal{S}(x, \ell)}{h(x)}\right)^{\frac{1}{h_1(x)}}\right)
	\end{align*}
	and
	\begin{equation}\label{inverse:conju:4}
		g(x,z) \leq   \frac{q^\ast(x)}{\zeta} \left[ \frac{z^\frac{1}{q^\ast(x)}}{\log^{\frac{s(x)}{q(x)}}\left(1+ \left(\frac{z}{h(x)}\right)^{\frac{1}{h_1(x)}}\right)} - \frac{(\mathcal{S}(x ,\ell))^\frac{1}{q^\ast(x)}}{\log^{\frac{s(x)}{q(x)}}\left(1+ \left(\frac{\mathcal{S}(x ,\ell)}{h(x)}\right)^{\frac{1}{h_1(x)}}\right)}\right].
	\end{equation}
	Using \eqref{inverse:conju:4} in \eqref{inverse:conju:3} and the fact that $b(x) \neq 0$ for $x \in A_2^{(1)}$ implies
	\begin{align*}
		\mathcal{S}_*^{-1}(x,z) & \leq \frac{q^\ast(x)}{\zeta (b(x))^\frac{1}{q(x)}} \frac{z^\frac{1}{q^\ast(x)}}{\log^{\frac{s(x)}{q(x)}}\left(1+ \left(\frac{z}{h(x)}\right)^{\frac{1}{h_1(x)}}\right)} \\
		& \qquad + (\mathcal{S}(x,\ell))^{\frac{-1}{N}}  \left(\frac{N \ell}{N-1} -  \frac{q^\ast(x)}{\zeta (b(x))^\frac{1}{q(x)}} \frac{(\mathcal{S}(x ,\ell))^\frac{1}{q(x)}}{\log^{\frac{s(x)}{q(x)}}\left(1+ \left(\frac{\mathcal{S}(x ,\ell)}{h(x)}\right)^{\frac{1}{h_1(x)}}\right)} \right) \\
		& \leq \frac{q^\ast(x)}{\zeta (b(x))^\frac{1}{q(x)}} \frac{z^\frac{1}{q^\ast(x)}}{\log^{\frac{s(x)}{q(x)}}\left(1+ \left(\frac{z}{h(x)}\right)^{\frac{1}{h_1(x)}}\right)} + (\mathcal{S}(x,\ell))^{\frac{-1}{N}} \ell \left(\frac{N}{N-1} -  \frac{q^\ast(x)}{\zeta} \right),
	\end{align*}
	where in the last inequality we have used $b(x) \ell^{q(x)} \ln^{s(x)}(1+\ell) \leq \mathcal{S}(x, \ell) \leq h(x) \ell^{h_1(x)}$. Now, by choosing $\ell \gg 1$ large enough and $0 < \zeta < 1$ small enough such that $\zeta < \frac{q^-(N-1)}{N-q^-}$, we rewrite the above estimate. Now by replacing $z= \mathcal{S}_\ast(x, \tau)$ for $x \in A_2^{(1)}$ and $\tau \geq \mathcal{S}_\ast^{-1}(x, \mathcal{S}(x, \ell))$, we get
	\begin{equation}\label{inverse:conju:5}
		\begin{aligned}
			&\mathcal{S}_\ast(x, \tau)\\
			& \geq \left[\frac{\zeta (b(x))^\frac{1}{q(x)}}{q^\ast(x)} \log^{\frac{s(x)}{q(x)}}\left(1+ \left(\frac{\mathcal{S}_\ast(x, \tau)}{h(x)}\right)^{\frac{1}{h_1(x)}}\right) \left[ \tau + (\mathcal{S}(x,\ell))^{\frac{-1}{N}} \ell \left(\frac{q^\ast(x)}{\zeta} -\frac{N}{N-1} \right)\right] \right]^{q^\ast(x)} \\
			& \geq \left[\frac{\zeta (b(x))^\frac{1}{q(x)}}{q^\ast(x)}  \tau\right]^{q^\ast(x)} \log^{\frac{s(x) q^\ast(x)}{q(x)}}\left(1+ \left(\frac{\mathcal{S}_\ast(x, \tau)}{h(x)}\right)^{\frac{1}{h_1(x)}}\right).
		\end{aligned}
	\end{equation}
	To estimate $\mathcal{S}_\ast(x, \tau)$ in the logarithmic term in the right-hand side of the above estimate, we use the fact that $\mathcal{S}^{-1}(x,y) \geq  \ell \gg 1$ in \eqref{inverse:conju:1} and obtain
	\begin{align*}
		y \geq  \left(a(x) + b(x) \right) \left(\mathcal{S}^{-1}(x,y)\right)^{\min\{p(x), q(x)\}} = h(x) \left(\mathcal{S}^{-1}(x,y)\right)^{\alpha(x)}.
	\end{align*}
	This yields
	\begin{equation}\label{inverse:conju:10}
		\mathcal{S}^{-1}(x,y) \leq \left(\frac{y}{h(x)}\right)^\frac{1}{\alpha(x)} \quad \text{for }(x, y) \in A_2 \cap \{s(x) \geq 0\} \times [\mathcal{S}(x,\ell), \infty).
	\end{equation}
	Now, by repeating the arguments for estimating the upper bound of inverse of the Sobolev conjugate function $\mathcal{S}_\ast^{-1} (x, y)$ in the light of \eqref{inverse:conju:10}, we obtain
	\begin{align*}
		\mathcal{S}_*^{-1}(x,z) & \leq \mathcal{S}_*^{-1}(x,\mathcal{S}(x,\ell)) + \frac{1}{(h(x))^\frac{1}{\alpha(x)}} \int_{\mathcal{S}(x,\ell)}^z  y^{\frac{1}{\alpha^\ast(x)} -1} \,\mathrm{d}y   \\
		& =  \frac{N \ell}{N-1} (\mathcal{S}(x,\ell))^{\frac{-1}{N}} + \frac{\alpha^\ast(x)}{(h(x))^\frac{1}{\alpha(x)}} \left(z^\frac{1}{\alpha^\ast(x)} - (\mathcal{S}(x,\ell))^\frac{1}{\alpha^\ast(x)} \right).
	\end{align*}
	Replacing $z= \mathcal{S}_\ast(x, \tau)$ for $x \in A_2 \cap \{s(x) \geq 0\}$ and $\tau \geq \mathcal{S}_\ast^{-1}(x, \mathcal{S}(x, \ell))$, and using $\mathcal{S}(x, \ell) \geq h(x) \ell^{\alpha(x)}$, we have
	\begin{align*}
		\mathcal{S}_\ast(x, \tau) & \geq \left[ \frac{(h(x))^\frac{1}{\alpha(x)}}{\alpha^\ast(x)} \left(\tau + (\mathcal{S}(x,\ell))^\frac{-1}{N} \left(\alpha^\ast(x) \left(\frac{ \mathcal{S}(x,\ell)}{h(x)}\right)^\frac{1}{\alpha(x)} - \frac{N \ell}{N-1}\right)\right) \right]^{\alpha^\ast(x)} \\
		& \geq \left(\frac{(h(x))^\frac{1}{\alpha(x)}}{\alpha^\ast(x)} \right)^{\alpha^\ast(x)} \tau^{\alpha^\ast(x)}.
	\end{align*}
	Merging the above lower estimate of $\mathcal{S}_\ast(x, \tau)$ in \eqref{inverse:conju:5} and using the assumption \eqref{main:assump-1} which implies $\frac{\alpha^\ast(x)}{h_1(x)} > 1$, we obtain
	\begin{align*}
		\mathcal{S}_\ast(x, \tau) & \geq \left[\frac{\zeta (b(x))^\frac{1}{q(x)}}{q^\ast(x)} \left[ \tau + (\mathcal{S}(x,\ell))^{\frac{-1}{N}} \ell \left(\frac{q^\ast(x)}{\zeta} -\frac{N}{N-1} \right)\right] \right]^{q^\ast(x)}   \\
		& \qquad \qquad \qquad \times \log^{\frac{s(x) q^\ast(x)}{q(x)}}\left(1+ \left(\frac{(h(x))^{\frac{1}{\alpha(x)} - \frac{1}{\alpha^\ast(x)}}}{\alpha^\ast(x)} \right)^{\frac{\alpha^\ast(x)}{h_1(x)}} \tau^\frac{\alpha^\ast(x)}{h_1(x)}\right) \\
		& \geq C \left[(b(x))^\frac{1}{q(x)}  \log^{\frac{s(x)}{q(x)}}\left(1+ \tau \right)  \left[ \tau + (\mathcal{S}(x,\ell))^{\frac{-1}{N}} \ell \left(\frac{q^\ast(x)}{\zeta} -\frac{N}{N-1} \right)\right] \right]^{q^\ast(x)}  \\
		& \geq C \left[(b(x))^\frac{1}{q(x)}  \log^{\frac{s(x)}{q(x)}}\left(1+ \tau \right)  \right]^{q^\ast(x)} \tau^{{q^\ast(x)} },
	\end{align*}
	where the constant $C$ depends upon $\zeta, (q^\ast)_\pm, (\alpha^\ast)_\pm, {h_1}_{\pm}, h_\pm$.
\end{proof}

\begin{proposition}\label{prop:lower:est-4}
	Let $S_\ast$ be the Sobolev conjugate function given by \eqref{inverse:conj:def}. Then, there exist $l_0 \gg 1$ and $C>0$ such that for $\ell \geq \ell_0$ and for $x \in A_2^{(2)}$ the following estimates hold true:
	\begin{enumerate}
		\item[\textnormal{(i)}]
			$\mathcal{S}^{-1}(x,y) \leq \left(\frac{y}{b(x) }\right)^\frac{1}{q(x)} \log^{\frac{-s(x)}{q(x)}}\left(1+ \left(\frac{y}{h(x)}\right)^{\frac{1}{h_2(x)}}\right)$ for any $y \in [\mathcal{S}(x,\ell), \infty)$ where
			\begin{align*}
				h(x):= a(x) + b(x) \quad \text{ and } \quad h_2(x):= \min\{p(x), q(x) - \varepsilon\}, \varepsilon >0.
			\end{align*}
		\item[\textnormal{(ii)}]
			$\mathcal{S}_*^{-1}(x,z) \leq (\mathcal{S}(x,\ell))^{\frac{-1}{N}} \ell \left[\frac{N}{N-1} - q^\ast(x)\right] + \frac{ q^\ast(x) }{(b(x))^\frac{1}{q(x)}}   z^\frac{1}{q^\ast(x)} \log^{\frac{-s(x)}{q(x)}}\left(1+ \left(\frac{z}{h(x)}\right)^{\frac{1}{h_2(x)}}\right)$ for any $z \in  [\mathcal{S}(x, \ell), \infty)$.
		\item[\textnormal{(iii)}]
			$\mathcal{S}_\ast(x,t) \geq  C \left(\frac{(b(x))^\frac{1}{q(x)}}{ q^\ast(x) } \log^{\frac{s(x)}{q(x)}}\left(1+t \right)\right)^{q^\ast(x)} t^{q^\ast(x)}$ for any $t \in \left[\mathcal{S}_\ast^{-1}(x,\mathcal{S}(x,\ell)), \infty \right).$
	\end{enumerate}
\end{proposition}

\begin{proof}
	Using $\log(1+t) \leq t^\frac{\varepsilon}{\lceil -s \rceil^+}$ for $t \geq t_0(\varepsilon, \lceil -s \rceil^+)$ and $\varepsilon>0$ in \eqref{inverse:est-1}, we get
	\begin{align}\label{lower:ineq:est}
		\mathcal{S}(x, \ell) \geq a(x) t^{p(x)} + b(x) \ell^{q(x)-\varepsilon} \geq h(x) t^{h_2(x)} \quad \text{for }  \ell \geq \ell_0 \gg 1,
	\end{align}
	which further implies, by replacing $y=\mathcal{S}(x, \ell)$ in the above inequalities,
	\begin{equation}\label{inverse:conju:6}
		\mathcal{S}^{-1}(x,y) \leq \left(\frac{y}{h(x)}\right)^{\frac{1}{h_2(x)}} \quad \text{for }  (x, y) \in A_2^{(2)} \times [\mathcal{S}(x,\ell), \infty).
	\end{equation}
	where
	\begin{align*}
		h_2(x) := \min\{p(x), q(x)- \varepsilon\}.
	\end{align*}
	From \eqref{inverse:conju:1} and \eqref{inverse:conju:6}, we get
	\begin{align*}
		\mathcal{S}^{-1}(x,y) \leq  \left(\frac{y}{b(x) }\right)^\frac{1}{q(x)} \log^{\frac{-s(x)}{q(x)}}\left(1+ \left(\frac{y}{h(x)}\right)^{\frac{1}{h_2(x)}}\right).
	\end{align*}
	Now, by using the fact that $x \in A_2^{(2)}$ in the definition of the inverse of the Sobolev conjugate function, we obtain
	\begin{align*}
		\mathcal{S}_*^{-1}(x,z) &  \leq \frac{N \ell}{N-1} (\mathcal{S}(x,\ell))^{\frac{-1}{N}} + \frac{1}{(b(x))^\frac{1}{q(x)}} \int_{\mathcal{S}(x,\ell)}^z  y^{\frac{1}{q^\ast(x)}-1} \log^{\frac{-s(x)}{q(x)}}\left(1+ \left(\frac{y}{h(x)}\right)^{\frac{1}{h_2(x)}}\right) \,\mathrm{d}y  \\
		& \leq (\mathcal{S}(x,\ell))^{\frac{-1}{N}} \left[\frac{N \ell}{N-1} - \frac{ q^\ast(x) }{(b(x))^\frac{1}{q(x)}}  (\mathcal{S}(x,\ell))^\frac{1}{q(x)} \log^{\frac{-s(x)}{q(x)}}\left(1+ \left(\frac{\mathcal{S}(x, \ell)}{h(x)}\right)^{\frac{1}{h_2(x)}}\right)\right] \\
		& \qquad + \frac{ q^\ast(x) }{(b(x))^\frac{1}{q(x)}}   z^\frac{1}{q^\ast(x)} \log^{\frac{-s(x)}{q(x)}}\left(1+ \left(\frac{z}{h(x)}\right)^{\frac{1}{h_2(x)}}\right)   \\
		& \leq (\mathcal{S}(x,\ell))^{\frac{-1}{N}} \ell \left[\frac{N}{N-1} - q^\ast(x)\right] + \frac{ q^\ast(x) }{(b(x))^\frac{1}{q(x)}}   z^\frac{1}{q^\ast(x)} \log^{\frac{-s(x)}{q(x)}}\left(1+ \left(\frac{z}{h(x)}\right)^{\frac{1}{h_2(x)}}\right),
	\end{align*}
	where in the last inequality we have used \eqref{lower:ineq:est} and $\mathcal{S}(x, \ell) \geq b(x) \ell^{q(x)} \log^{s(x)}(1+\ell)$. Now, by replacing $z= \mathcal{S}_\ast(x, \tau)$ for $x \in A_2 \cap \{s(x) \leq 0\}$ and $\tau \geq \mathcal{S}_\ast^{-1}(x, \mathcal{S}(x, \ell))$, we get
	\begin{align*}
		\mathcal{S}_\ast(x, \tau) & \geq \left(\frac{(b(x))^\frac{1}{q(x)}}{ q^\ast(x) } \log^{\frac{s(x)}{q(x)}}\left(1+ \left(\frac{\mathcal{S}_\ast(x, \tau)}{h(x)}\right)^{\frac{1}{h_2(x)}}\right)\right)^{q^\ast(x)} \\
		& \qquad \times \left[ \tau + (\mathcal{S}(x,\ell))^{\frac{-1}{N}} \ell \left[q^\ast(x) - \frac{N}{N-1} \right]\right]^{q^\ast(x)}.
	\end{align*}
	Next, by repeating the above arguments for the lower bound, we get
	\begin{align*}
		\mathcal{S}(x, \ell) \leq h(x) \ell^{h_3(x)} \quad \text{for }  \ell \geq  e-1, \text{ where } h_3(x) = \max\{p(x), q(x)\}.
	\end{align*}
	By replacing $y=\mathcal{S}(x, \ell)$, we have $\mathcal{S}^{-1}(x, y) \geq \left(\frac{y}{h(x)}\right)^\frac{1}{h_3(x)}$ and
	\begin{equation}\label{inverse:conju:8}
		\begin{aligned}
			\mathcal{S}_*^{-1}(x,z)
			& \geq \frac{N \ell}{N-1} (\mathcal{S}(x,\ell))^{\frac{-1}{N}} + \frac{1}{(h(x))^\frac{1}{h_3(x)}}\int_{\mathcal{S}(x,\ell)}^z  y^{\frac{1}{h_3^\ast(x)}-1}  \,\mathrm{d}y  \\
			& \geq (\mathcal{S}(x,\ell))^{\frac{-1}{N}} \left[\frac{N \ell}{N-1} - h_3^\ast(x) \left(\frac{\mathcal{S}(x,\ell)}{(h(x))}\right)^\frac{1}{h_3(x)}  \right] +  \frac{h_3^\ast(x)}{(h(x))^\frac{1}{h_3(x)}}   z^\frac{1}{h_3^\ast(x)} \\
			& \geq  \frac{h_3^\ast(x)}{(h(x))^\frac{1}{h_3(x)}}   z^\frac{1}{h_3^\ast(x)}.
		\end{aligned}
	\end{equation}
	Again, by replacing $z= \mathcal{S}_\ast(x, \tau)$ for $x \in A_2^{(2)}$ and $\tau \geq \mathcal{S}_\ast^{-1}(x, \mathcal{S}(x, \ell))$ and using $d \leq a(x) + b(x)$, we get
	\begin{align*}
		\mathcal{S}_\ast(x, \tau) & \leq  \left[ \frac{(h(x))^\frac{1}{h_3(x)}}{h_3^\ast(x)} \right]^{h_3^\ast(x)} \tau^{h_3^\ast(x)} \leq C h(x) \tau^{h_3^\ast(x)},
	\end{align*}
	where $C$ depends upon $d, h_\pm, {h_3}_\pm, (h_3^\ast)_\pm$.
	The above estimate and the fact that $\frac{h_3^\ast(x)}{h_2(x)} > 1$ implies
	\begin{equation}\label{inverse:conju:9}
		\log\left(1+ \left(\frac{\mathcal{S}_\ast(x, \tau)}{h(x)}\right)^{\frac{1}{h_2(x)}}\right) \leq \log\left(1+ \left( C_3 \tau^{h_3^\ast(x)}\right)^{\frac{1}{h_2(x)}}\right) \leq C \log(1 + \tau).
	\end{equation}
	Finally, by using \eqref{inverse:conju:9} in \eqref{inverse:conju:8} shows that
	\begin{align*}
		\mathcal{S}_\ast(x, \tau) & \geq C \left(\frac{(b(x))^\frac{1}{q(x)}}{ q^\ast(x) } \log^{\frac{s(x)}{q(x)}}\left(1+\tau\right)\right)^{q^\ast(x)}  \left[ \tau + (\mathcal{S}(x,\ell))^{\frac{-1}{N}} \ell \left[q^\ast(x)  - \frac{N}{N-1} \right]\right]^{q^\ast(x)} \\
		& \geq C \left(\frac{(b(x))^\frac{1}{q(x)}}{ q^\ast(x) } \log^{\frac{s(x)}{q(x)}}\left(1+\tau\right)\right)^{q^\ast(x)} \tau^{q^\ast(x)}.
	\end{align*}
\end{proof}

For $(x, t) \in \overline{\Omega} \times [0, \infty)$, we define
\begin{equation}\label{critical-MOf}
	\mathcal{S}^\ast(x,t):= \left( \left(a(x)\right)^\frac{1}{p(x)} t\right)^{p^\ast(x)}  + \left(\left(b(x) \log^{s(x)}(1+t)\right)^\frac{1}{q(x)} t  \right)^{q^\ast(x)}.
\end{equation}
Now we obtain the following embedding results.
\begin{proposition}\label{imp:embedding}
	Let \eqref{main:assump} and \eqref{main:assump-1} be satisfied. Then, the Musielak-Orlicz Sobolev space
	$W^{1, \mathcal{S}}(\Omega)$ is continuously embedded into the Musielak-Orlicz Lebesgue space $ L^{\mathcal{S}^{\ast}}(\Omega)$. Moreover, the Musielak-Orlicz Sobolev space
	$W^{1, \mathcal{S}}(\Omega)$ is compactly embedded into the Musielak-Orlicz Lebesgue space $ L^{\mathcal{N}^{\ast}}(\Omega)$ if $\mathcal{N}^\ast \ll \mathcal{S}^\ast$.
\end{proposition}

\begin{proof}
	Combining the lower estimates of the Sobolev conjugate function $\mathcal{S}_\ast$ in Propositions \ref{prop:esti-near0}-\ref{prop:lower:est-4}, the continuity of the function $\mathcal{S}^\ast$ and Propositions \ref{Prop:AbstractEmbedding} and \ref{pro:embed} we obtain the required claim.
\end{proof}

\begin{remark}
	Note that the embedding $W^{1, \mathcal{S}}(\Omega) \hookrightarrow L^{\mathcal{S}^{\ast}}(\Omega)$ is sharp is the sense that, for each fixed x, it coincides with the sharp Sobolev conjugate in classical Orlicz spaces. The sharpness in the case when $x$ is fixed and $s(x) \equiv 0$ can be justified by using the arguments by Cianchi--Diening \cite[Example 3.11]{Cianchi-Diening-2024} and Ho--Winkert \cite[Proposition 3.5]{Ho-Winkert-2023} while the case $s(x) \not \equiv 0$  can be justified by using the arguments by Cianchi \cite[Example 1.2]{Cianchi-2004} and Cianchi--Diening \cite[Example 3.11]{Cianchi-Diening-2024}.
\end{remark}

\section{Concentration-compactness principle}\label{Section-4}

In this section, we prove the concentration compactness principle for the Musielak-Orlicz Sobolev space $W_0^{1, \mathcal{S}}(\Omega)$ having logarithmic double phase modular function structure. For this, we set
\begin{align*}
	\mathtt{m}_-(x)&:= \min\left\{p(x) , q(x) + s^-(x) \right\}, & \mathtt{n}_+(x)&:= \max\left\{p(x), q(x) + s^+(x) \right\},\\
	\mathtt{m}_\varepsilon(x)&:= \min\left\{p(x) ,q(x) + \varepsilon \right\}, & \mathtt{n}_\varepsilon(x)&:= \max\left\{p(x) , q(x) + \varepsilon \right\},\\
	\mathtt{m}_\varepsilon^\ast(x)&:= \min\left\{p^{\ast}(x), q^\ast(x) + \varepsilon \right\}, & \mathtt{n}^\ast_\varepsilon(x)&:= \max\left\{p^{\ast}(x),q^\ast(x) + \varepsilon \right\}\quad \text{for}  \varepsilon \in \mathbb{R},\\
	\mathtt{m}^\ast_-(x)&:= \min\left\{p^{\ast}(x), q^\ast(x) \left(1 + \frac{s^-(x)}{q(x)}\right) \right\},& \mathtt{n}^\ast_+(x)&:= \max\left\{p^{\ast}(x), q^\ast(x) \left(1 + \frac{s^+(x)}{q(x)}\right) \right\}.
\end{align*}

Straightforward computations give directly the following lemma.

\begin{lemma}\label{upp-low:est}
	Let $S$ and $S^\ast$ be the functions defined in \eqref{eq:H} and \eqref{critical-MOf}. Then, the following hold:
	\begin{enumerate}
		\item[\textnormal{(i)}]
			for $x \in \Omega $ and $t,z \in (0, \infty)$
			\begin{align*}
				\mathtt{m}_-(x) \leq \frac{t S'(x,t)}{S(x,t)} \leq \mathtt{n}_+(x),
			\end{align*}
			and
			\begin{align*}
				\min\{t^{\mathtt{m}_-(x)}, t^{\mathtt{n}_+(x)}\} \mathcal{S}(x,z) \leq \mathcal{S}(x, tz) \leq \max\{t^{\mathtt{m}_-(x)}, t^{\mathtt{n}_+(x)}\} \mathcal{S}(x,z)
			\end{align*}
			where $S'$ represents the partial derivative of $S$ with respect to the second variable $t$.
		\item[\textnormal{(ii)}]
			for  $\varepsilon>0$, there exists $t=t(\varepsilon, s^\pm) >1$ such that
			\begin{align*}
				\frac{t S'(x,t)}{S(x,t)} \leq \mathtt{n}_\varepsilon(x) \quad \text{ and } \quad \mathcal{S}(x, tz) \leq t^{\mathtt{n}_\varepsilon(x)} \mathcal{S}(x,z)
			\end{align*}
			for $(x,t) \in \Omega \times (t_\varepsilon, \infty)$ and $z>0$.
	\end{enumerate}
	The same estimates hold for the function $\mathcal{S}^\ast$ by replacing $\mathtt{m}_-$, $\mathtt{n}_+$ and $\mathtt{n}_\varepsilon$ with $\mathtt{m}_-^\ast$, $\mathtt{n}_+^\ast$ and $\mathtt{n}_\varepsilon^\ast$, respectively.
\end{lemma}

\begin{proposition}
	Let $g \in C(\overline{\Omega})$ and $\mathcal{N} \in {\bf \Phi}(\Omega)$  satisfies the $\Delta_2$-condition,
	\begin{align}\label{upper-growth:control}
		\mathcal{N}(x,t w) \leq w^{g(x)} \mathcal{N}(x,t)  \quad \text{for } w \geq 1  \text{ and for all } (x,t) \in \Omega \times (0, \infty),
	\end{align}
	and for a.a.\,$x \in \Omega$, the map $t \to \mathcal{N}(x,t)$ is non-decreasing. Then, for every $\varepsilon>0$ there exists $C_\varepsilon >0$ such that
	\begin{equation}\label{sum-control}
		\left|\mathcal{N}(x, |t+m|) - \mathcal{N}(x, |t|) \right| \leq C_\varepsilon \mathcal{N}(x, |m|) + \varepsilon  \mathcal{N}(x, |t|)
	\end{equation}
	for every $t, m \geq 0$ and for a.a.\,$x \in \Omega$.
\end{proposition}

\begin{proof}
	Let $t , m \geq 0$ and $\delta>0$. If $m \geq \delta t$, then the monotonicity of the map $t \to \mathcal{N}(x, t)$ and the $\Delta_2$-condition give
	\begin{equation}\label{sc:est-1}
		\mathcal{N}(x, t+m) \leq \mathcal{N}(x, m \delta^{-1} +m ) \leq \mathcal{N}(x, 2^k m ) \leq C^k \mathcal{N}(x, m)
	\end{equation}
	where $k=k(\delta) >0$ and $C$ is the constant obtained in the $\Delta_2$-condition. If $m \leq \delta t$, then using \eqref{upper-growth:control}, we get
	\begin{equation}\label{sc:est-2}
		\mathcal{N}(x, t+m) \leq \mathcal{N}(x, (1+\delta) t) \leq (1+\delta)^{g(x)} \mathcal{N}(x,t).
	\end{equation}
	Combining  \eqref{sc:est-1} and \eqref{sc:est-2} yields
	\begin{align*}
		\mathcal{N}(x, t +m) \leq (1+\delta)^{g(x)} \mathcal{N}(x, t) + C_\delta \mathcal{N}(x,m).
	\end{align*}
	Replacing $t$ by $|t|$ and $m$ by $|m|$ and using the triangle inequality and the monotonicity of the map $s \to \mathcal{N}(x, s)$, we obtain
	\begin{equation}\label{sc:est-3}
		\begin{aligned}
			\mathcal{N}(x, |t +m|) - \mathcal{N}(x, |t|) & \leq ((1+\delta)^{g(x)}-1) \mathcal{N}(x, |t|) + C_\delta \mathcal{N}(x,|m|)   \\
			& \leq \varepsilon \mathcal{N}(x, |t|) + C_\varepsilon \mathcal{N}(x,|m|) \quad \text{for all }  t,s \in \mathbb{R} \text{ and  for a.a.\,} x \in \Omega.
		\end{aligned}
	\end{equation}
	Replacing $t$ by $t-m$ and then $m$ by $-m$ in \eqref{sc:est-3}, it follows
	\begin{equation}\label{sc:est-4}
		\begin{aligned}
			 & \mathcal{N}(x, |t|) - \mathcal{N}(x, |t+m|)\\
			 &\leq ((1+\delta)^{g(x)}-1) \mathcal{N}(x, |t+m|) + C_\delta \mathcal{N}(x,|m|)   \\
			 & \leq ((1+\delta)^{g(x)}-1) (1+\delta)^{g(x)} \mathcal{N}(x, |t|)
			 + C_\delta (1+\delta)^{g(x)} \mathcal{N}(x,|m|)
		\end{aligned}
	\end{equation}
	for all $t,s \in \mathbb{R}$ and for all a.a.\,$x \in \Omega$. Then, for given $\varepsilon>0$, choosing $\delta$ small enough such that $((1+\delta)^{g^+} -1) (1+\delta)^{g^+} \leq \varepsilon$ in \eqref{sc:est-3} and \eqref{sc:est-4}, we obtain
	\begin{align*}
		\left|\mathcal{N}(x, |t+m|) - \mathcal{N}(x, |t|) \right| \leq C_\varepsilon \mathcal{N}(x, |m|) + \varepsilon  \mathcal{N}(x, |t|).
	\end{align*}
\end{proof}

\begin{lemma}[Br\'{e}zis-Lieb Lemma]\label{BL-lemma}
	Let $\mathcal{N} \in {\bf \Phi}(\Omega)$  satisfies \eqref{sum-control}.  Then, for $f_n \to f$ a.e.\,in $\Omega$ and $f_n \rightharpoonup f$ in $L^\mathcal{N}(\Omega)$, we have
	\begin{align*}
		\lim_{n \to \infty} \left(\int_{\Omega} \mathcal{N}(x, |f_n|) \phi \,\mathrm{d}x - \int_{\Omega} \mathcal{N}(x, |f-f_n|) \phi \,\mathrm{d}x \right) = \int_{\Omega} \mathcal{N}(x, |f|) \phi \,\mathrm{d}x \quad \text{for every }\phi \in L^\infty(\Omega).
	\end{align*}
\end{lemma}

\begin{proof}
	Define a sequence of functions $B_{\varepsilon,n} \colon  \Omega \to \mathbb{R}^+$ as
	\begin{align*}
		B_{\varepsilon,n}(x) = \bigg[(1-\varepsilon)\mathcal{N}(x, |f_n(x)|) - \mathcal{N}(x, |f(x) - f_n(x)|) - \mathcal{N}(x, |f(x)|) \bigg]^+.
	\end{align*}
	From \eqref{sum-control}, we note that
	\begin{align*}
		& \left|\mathcal{N}(x, |f_n(x)|) - \mathcal{N}(x, |f(x) - f_n(x)|) - \mathcal{N}(x, |f(x)|)\right| - \varepsilon \mathcal{N}(x, |f_n(x)|)  \\
		& \leq \left|\mathcal{N}(x, |f_n(x)|) - \mathcal{N}(x, |f(x) - f_n(x)|) \right| + \mathcal{N}(x, |f(x)|) - \varepsilon \mathcal{N}(x, |f_n(x)|) \\
		& \leq (C_\varepsilon +1) \mathcal{N}(x, |f(x)|).
	\end{align*}
	It follows that
	\begin{align*}
		0 \leq B_{\varepsilon, n}(x) \leq (C_\varepsilon +1) \mathcal{N}(x, |f(x)|).
	\end{align*}
	Now, by using $B_{\varepsilon, n}(x) \to 0$ a.e.\,in $\Omega$ and Lebesgue's dominated convergence theorem, we get
	\begin{equation}\label{BL:est-1}
		\lim_{n \to \infty} \int_{\Omega} B_{\varepsilon, n}(x) \phi(x) \,\mathrm{d}x =0 \quad \text{for any } \phi \in L^\infty(\Omega).
	\end{equation}
	On the other hand, $f_n \rightharpoonup f$ in $L^\mathcal{N}(\Omega)$ implies
	\begin{equation}\label{BL:est-2}
		\begin{aligned}
			& \left|\int_{\Omega} \mathcal{N}(x, |f_n(x)|) \phi(x) \,\mathrm{d}x -  \int_{\Omega} \mathcal{N}(x, |f(x) - f_n(x)|) \phi(x) \,\mathrm{d}x - \int_{\Omega} \mathcal{N}(x, |f(x)|) \phi(x) \,\mathrm{d}x \right| \\
			&\leq \int_{\Omega} B_{\varepsilon, n}(x) |\phi(x)| \,\mathrm{d}x + \varepsilon \int_{\Omega}  \mathcal{N}(x, |f_n(x)|) |\phi(x)| \,\mathrm{d}x.  \\
			& \leq \int_{\Omega} B_{\varepsilon, n}(x) |\phi(x)| \,\mathrm{d}x + \varepsilon \|\phi\|_{\infty, \Omega} C
		\end{aligned}
	\end{equation}
	for some constant $C>0$. Hence, the claim follows by combining \eqref{BL:est-1} and \eqref{BL:est-2}.
\end{proof}

\begin{proposition}\label{pro:ln-est}
	The following inequalities hold true:
    \begin{enumerate}
        \item[\textnormal{(i)}] $\log(1 + tz) \geq \log(1+t) \log(1+z)$ for $z \in [0, 1]$ and $t \geq 0.$
        \item[\textnormal{(ii)}] $\log(1+t) \log(1+z) \geq \log(2)  \log(1 + tz)$ for $z \geq 1$ and $t \geq 1.$
    \end{enumerate}
\end{proposition}

\begin{proof}
	For a fixed $z \in [0,1]$, define $h_z \colon  \mathbb{R}^+ \to \mathbb{R}$ given by
	\begin{align*}
		h_z(t) = \log(1+tz) - \log(1+t) \log(1+z).
	\end{align*}
	Note that $h(0)=0$ and
	\begin{align*}
		h_z'(t) = \frac{z}{1+tz} - \frac{\log(1+z)}{1+t} \geq 0 \quad \text{for all }  t \geq 0  \text{ and }  z \in [0,1].
	\end{align*}
	Hence, the claim in \text{(i)} follows. Observe that the function $k\colon  [1, \infty) \to \mathbb{R}$ given by
	\begin{equation}\label{obser}
		k(z) = z \left(\log(1+z) -\log(2)\right) + \log(1+z) - z \log(2) \quad \text{is non-negative for all }  z \in [1, \infty).
	\end{equation}
	Now, to prove \text{(ii)}, we fix $z \geq 1$ and define $g_z\colon  [1, \infty) \to \mathbb{R}$ given by
	\begin{align*}
		g_z(t) = \log(1+t) \log(1+z) - \log(2)  \log(1 + tz).
	\end{align*}
	Clearly $g_z(1) =0$ and
	\begin{align*}
		g_z'(t) & = \frac{(\log(1+z)-z \log(2)) + (\log(1+z) - \log(2))tz}{(1+t)(1+tz)}  \\
		& \geq \frac{(\log(1+z)-z \log(2)) + (\log(1+z) - \log(2))z}{(1+t)(1+tz)} \geq 0 \quad \text{for }  z, t \geq 1,
	\end{align*}
	where in the last inequality, we have used \eqref{obser}. This further implies $g_z \geq 0$ in $[1, \infty)$ and for all $z \geq 1$, and hence the claim in \text{(ii)} is shown.
\end{proof}

For $\varepsilon \in (0,1)$, we define sub-multiplicative functions as follows:
\begin{equation}\label{pert-MOF}
	\mathtt{M}_{\varepsilon}^\ast(x, z) := \min\{z^{p^\ast(x)}, z^{q^\ast(x)} \mathtt{M}_{\varepsilon, \log}(x,z)\}  \quad \text{ and } \quad \mathtt{M}_{\varepsilon}(x, z) :=
	\begin{cases}
		z^{\mathtt{m}_-(x)}  & \text{if }z < 1, \\
		z^{\mathtt{n}_\varepsilon(x)} & \text{if }z\geq 1,
	\end{cases}
\end{equation}
where
\begin{align*}
	\mathtt{M}_{\varepsilon, \log}(x, z) :=
	\begin{cases}
		\left(\log(1+z) \right)^\frac{s(x) q^\ast(x)}{q(x)}  & \text{if $s(x) \geq 0$  and  $z \leq 1$},\\
		\left(\log(2) \right)^\frac{s(x) q^\ast(x)}{q(x)} z^\varepsilon  & \text{if $s(x) \geq 0$  and  $z \geq 1$}, \\
		1 & \text{if $s(x) < 0$  and  $z \leq 1$},\\
		\left(\frac{C(\varepsilon)}{\log(2)}\right)^{\frac{s(x) q^\ast(x)}{q(x)}} z^{-\varepsilon} & \text{if $s(x) < 0$  and  $z \geq 1$}.
	\end{cases}
\end{align*}

\begin{proposition}\label{lower:est-conju}
	Let $S^\ast$ and $\mathtt{M}_\varepsilon^\ast$ be the functions defined in \eqref{critical-MOf} and \eqref{pert-MOF}.
	Then, for every $(x,t) \in \overline{\Omega} \times [1, \infty)$ and $z \geq 0$, we have
	\begin{align*}
		\mathcal{S}^\ast(x, tz) \geq \mathcal{S}^\ast(x,t) \mathtt{M}_\varepsilon^\ast(x,z).
	\end{align*}
\end{proposition}

\begin{proof}
	By Proposition \ref{pro:ln-est} (i) for $(x,t) \in \{x\in\overline{\Omega}\colon s(x) \geq 0\} \times [0, \infty)$ we deduce that
	\begin{equation}\label{critical-lower:est-1}
		\begin{aligned}
			\log^{\frac{s(x) q^\ast(x)}{q(x)}}(1+ tz) & \geq
			\begin{cases}
				\left(\log(2) \log(1+ t)\right)^{\frac{s(x) q^\ast(x)}{q(x)}}, & \text{if }z \geq 1. \\
				\log^{\frac{s(x) q^\ast(x)}{q(x)}}(1+z)  \log^{\frac{s(x) q^\ast(x)}{q(x)}}(1+ t), & \text{if }z < 1,
			\end{cases}
		\end{aligned}
	\end{equation}
	and by Proposition \ref{pro:ln-est} (ii) for $(x,t) \in \{x\in\overline{\Omega} \colon s(x) <0 \} \times [1, \infty)$
	\begin{equation}\label{critical-lower:est-2}
		\begin{aligned}
			\log^{\frac{s(x) q^\ast(x)}{q(x)}}(1+ tz) & \geq
			\begin{cases}
				\left(\frac{\log(1+z)}{\log(2)}\right)^{\frac{s(x) q^\ast(x)}{q(x)}}  \log^{\frac{s(x) q^\ast(x)}{q(x)}}(1+ t), &\text{if } z \geq 1. \\
				\log^{\frac{s(x) q^\ast(x)}{q(x)}}(1+ t),   &\text{if } z < 1.
			\end{cases}
		\end{aligned}
	\end{equation}
	Note that for $\delta> 0$ the inequality
	\begin{align*}
		\log(1+z) \leq C(\delta) z^\delta \quad \text{for }   z \geq 1
	\end{align*}
	holds.
	Now, by using \eqref{critical-lower:est-1} and \eqref{critical-lower:est-2} in the  definition of $\mathcal{S}^\ast$ and choosing $\delta = \varepsilon \frac{r^+}{(|s|)^+( r^\ast)^+}$ in the above inequality, we obtain
	\begin{align*}
		\mathcal{S}^\ast(x,tz) & = \left(\left(a(x)\right)^\frac{1}{p(x)} t \right)^{p^\ast(x)} z^{p^\ast(x)} + \left( (b(x))^\frac{1}{q(x)} t z\right)^{q^\ast(x)} \log^{\frac{s(x)q^\ast(x)}{q(x)}}(1+tz) \geq \mathcal{S}^\ast(x,t)  \mathtt{M}_\varepsilon^\ast(x,z).
	\end{align*}
\end{proof}

Now, we can state the main result of this section.

\begin{theorem}\label{concentration:compactness}
	Let hypothesis \eqref{main:assump} and \eqref{main:assump-1} be satisfied and $\{v_k\}_{k \in \mathbb{N}}$ be a weakly convergent sequence in $W^{1, \mathcal{S}}(\Omega)$ with weak limit $v$ such that
	\begin{align*}
		S^\ast(x, |v_k|) \rightharpoonup \Theta  \ \text{ and }  \ S(x, |\nabla v_k|) \rightharpoonup \theta \quad \text{weakly-}\ast  \text{ in the sense of measures},
	\end{align*}
	where $\Theta$ and $\theta$ are signed Radon measures with finite mass. Then there exists a countable index set $I$, positive numbers $\{\Theta_i\}_{i \in I}$ and $\{\theta_i\}_{i \in I}$, and ${\bf C}^\ast >0$ such that
	\begin{align*}
		\Theta= S^\ast(x, |v|) + \sum_{i \in I} \Theta_i \delta_{x_i}, \quad \theta \geq  S(x, |\nabla v|) + \sum_{i \in I} \theta_i \delta_{x_i}
	\end{align*}
	and
	\begin{align*}
		\min\{\Theta_j^\frac{1}{\mathtt{n}_\varepsilon^\ast(x_j)}, \Theta_j^\frac{1}{\mathtt{m}_0^\ast(x_j)}\}  \leq {\bf C}^\ast \max\{\theta_j^\frac{1}{\mathtt{m}_-(x_j)}, \theta_j^\frac{1}{\mathtt{n}_\varepsilon(x_j)}\}.
	\end{align*}
\end{theorem}

To prove the above result, first we set $u_k:= v_k-v$ such that $u_k \rightharpoonup 0$ in $W_0^{1, \mathcal{S}}(\Omega)$ and
\begin{align*}
	\nu_k:=S^\ast(x, |u_k|) \rightharpoonup \nu  \ \text{ and }  \ S(x, |\nabla u_k|) \rightharpoonup \mu  \quad\text{weakly-}\ast  \text{ in the sense of measures}
\end{align*}
and prove some crucial Lemmas involving the measures $\mu$ and $\nu$.

\begin{lemma}\label{lem:conv:meas}
	Let $\{\nu_k\}_{k \in \mathbb{N}}$ be a non-negative and finite radon measure in $\Omega$ such that $\nu_k \rightharpoonup \nu  \text{weakly-}\ast  \text{in the sense of measures}$. Then,
	\begin{align*}
		\|\phi\|_{\mathtt{M}_\varepsilon^\ast, \nu_k} \to \|\phi\|_{\mathtt{M}_\varepsilon^\ast, \nu}\quad \text{as }  k \to \infty
	\end{align*}
	for any $\phi \in C(\overline{\Omega})$.
\end{lemma}

\begin{proof}
	The proof follows by repeating the same arguments as in Fern\'{a}ndez Bonder--Silva \cite[Lemma 3.1]{Fernandez-Bonder-Silva-2010}.
\end{proof}

Next, we prove the following reverse H\"older type inequality between the measures $\mu$ and $\nu$.

\begin{lemma}\label{reverse:holder:ineq}
	Let $\mathtt{M}_\varepsilon^\ast$ and $\mathtt{M}_\varepsilon$ be the functions defined in \eqref{pert-MOF}. Then, for every $\phi \in C^\infty(\overline{\Omega})$ the following inequality holds:
	\begin{equation}\label{reverse:holder}
		\|\phi\|_{\mathtt{M}_\varepsilon^\ast, \nu} \leq C \|\phi\|_{\mathtt{M}_\varepsilon, \mu}.
	\end{equation}
\end{lemma}

\begin{proof}
	Let $\phi \in C^\infty(\overline{\Omega})$. It is easy to see that $\phi u \in W_0^{1, \mathcal{S}}(\Omega)$ for every $u \in W_0^{1, \mathcal{S}}(\Omega)$. Using the Sobolev embedding in Proposition \ref{imp:embedding}, we obtain
	\begin{equation}\label{sobo:est-1}
		\|\phi u_k\|_{\mathcal{S}^\ast} \leq C \|\nabla (\phi u_k)\|_{\mathcal{S}}.
	\end{equation}
	Set $\delta_k= \|\phi\|_{\mathtt{M}_\varepsilon^\ast, \nu_k}$ for $k \in \mathbb{N}$ and $\delta:= \|\phi\|_{\mathtt{M}_\varepsilon^\ast, \nu}$. Then, by Lemma \ref{lem:conv:meas}, it holds
	\begin{align*}
		\lim_{k \to \infty} \delta_k =\delta.
	\end{align*}
	If $\delta =0$, \eqref{reverse:holder} holds. Without any loss of generality, we can assume that $\delta_k \neq 0$ for $k \in \mathbb{N}$. Using Proposition \ref{lower:est-conju}, we have
	\begin{equation}\label{sobo:est-2}
		\begin{aligned}
			\int_{\Omega} \mathcal{S}^\ast\left(x, \frac{|\phi u_k|}{\delta_k}\right) \,\mathrm{d}x & \geq \int_{|u_k| > 1} \mathtt{M}_\varepsilon^\ast\left(x, \frac{|\phi|}{\delta_k}\right) \mathcal{S}^\ast(x, |u_k|)    \,\mathrm{d}x  \\
			& = \int_{\Omega} \mathtt{M}_\varepsilon^\ast\left(x, \frac{|\phi|}{\delta_k}\right) \,\mathrm{d}\nu_k  -  \int_{|u_k| \leq 1} \mathtt{M}_\varepsilon^\ast\left(x, \frac{|\phi|}{\delta_k}\right) \mathcal{S}^\ast(x, |u_k|)    \,\mathrm{d}x \\
			& =  1 -  \int_{|u_k| \leq 1} \mathtt{M}_\varepsilon^\ast\left(x, \frac{|\phi|}{\delta_k}\right) \mathcal{S}^\ast(x, |u_k|)    \,\mathrm{d}x.
		\end{aligned}
	\end{equation}
	It is easy that
	\begin{equation}\label{upper:control:est}
		\int_{|u_k| \leq 1}  \mathtt{M}_\varepsilon^\ast\left(x, \frac{|\phi|}{\delta_k}\right) \mathcal{S}^\ast(x, |u_k|)   \,\mathrm{d}x \leq  \int_{\Omega} \mathtt{M}_\varepsilon^\ast\left(x, \frac{|\phi|}{\delta_k}\right) \mathcal{S}^\ast(x, |u_k|)    \,\mathrm{d}x = \left\|\frac{\phi}{\delta_k}\right\|_{\mathtt{M}_\varepsilon^\ast, \nu_k} =1.
	\end{equation}
	Applying Vitali's convergence theorem in light of \eqref{upper:control:est} and $u_k \to 0$ a.e.\,in $\Omega$, we deduce that
	\begin{equation}\label{sobo:est-3}
		\lim_{k \to \infty} \int_{|u_k| \leq 1} \mathtt{M}_\varepsilon^\ast\left(x, \frac{|\phi|}{\delta_k}\right)\mathcal{S}^\ast(x, |u_k|)    \,\mathrm{d}x = 0.
	\end{equation}
	Now, by using \eqref{sobo:est-3} in \eqref{sobo:est-2} and Lemma \ref{upp-low:est}, there exists $k_0$ (independent of the choice of $\varphi$) such that for $k \geq k_0$
	\begin{equation}\label{sobo:est-6}
		\int_{\Omega} \mathcal{S}^\ast\left(x, \frac{|\phi u_k|}{\delta_k}\right) \,\mathrm{d}x \geq  \frac{1}{2}.
	\end{equation}
	Applying Proposition \ref{pro:norm-mod:relation} \textnormal{{(ii)}}, Lemma \ref{upp-low:est} and passing $k \to \infty$ in \eqref{sobo:est-6} gives
	\begin{equation}\label{sobo:est-new}
		\lim\inf_{k \to \infty} \|\phi u_k\|_{\mathcal{S}^\ast} \geq C \lim\inf_{k \to \infty} \|\phi\|_{\mathtt{M}_\varepsilon^\ast, \nu_k} = C \|\phi\|_{\mathtt{M}_\varepsilon^\ast, \nu}
	\end{equation}
	for some $C>0$ independent of $k$. Next, we set $\lambda_k=\|\phi \nabla u_k\|_{\mathcal{S}}$ and
	\begin{align*}
		M:= 1 + \sup_{k \in \mathbb{N}} \int_{\Omega} \mathcal{S}\left(x, |\nabla u_k|\right) \,\mathrm{d}x.
	\end{align*}
	Because of Proposition \ref{imp:embedding} and the fact that $u_k \rightharpoonup 0$ in $W_0^{1, \mathcal{S}}(\Omega)$, we have $u_k \to 0$ in $L^\mathcal{S}(\Omega)$.
	From Lemma \ref{upp-low:est} $\textnormal{(ii)}$, we know that for $\varepsilon>0$ and $z>0$, there exists $K_\varepsilon>0$ such that $\mathcal{S}(x, tz) \leq t^{\mathtt{n}_\varepsilon(x)} \mathcal{S}(x,z)$ for $(x,t) \in \Omega \times (K_\varepsilon, \infty)$. Therefore, applying Lemma \ref{upp-low:est} (i), (ii) and Young's inequality, it follows that
	\begin{align*}
		1 &= \int_{\Omega} \mathcal{S}\left(x, \frac{|\nabla u_k| |\phi|}{\lambda_k}\right) \,\mathrm{d}x\\
		 & = \int_{\Omega \cap \{|\phi| \leq \lambda_k\}} \mathcal{S}\left(x, \frac{|\nabla u_k| |\phi|}{\lambda_k}\right) \,\mathrm{d}x\\
		 &\quad + \int_{\Omega \cap \{\lambda_k < |\phi| \leq \lambda_k K_\varepsilon\}} \mathcal{S}\left(x, \frac{|\nabla u_k| |\phi|}{\lambda_k}\right) \,\mathrm{d}x  + \int_{\Omega \cap \{|\phi| > \lambda_k K_\varepsilon\}} \mathcal{S}\left(x, \frac{|\nabla u_k| |\phi|}{\lambda_k}\right) \,\mathrm{d}x  \\
		& \leq  \int_{\Omega \cap \{|\phi| \leq \lambda_k,\ |\phi|> \lambda_k K_\varepsilon\}} \mathcal{S}\left(x, |\nabla u_k|\right) \mathtt{M}_{\varepsilon}\left(x, \frac{|\phi|}{\lambda_k}\right) \,\mathrm{d}x \\
		& \quad \qquad + \int_{\Omega \cap \{\lambda_k < |\phi| \leq \lambda_k K_\varepsilon\}} \left( \frac{|\phi|}{\lambda_k}\right)^{\mathtt{n}_+(x)} \mathcal{S}\left(x, |\nabla u_k|\right) \,\mathrm{d}x  \\
		& \leq  \frac{1}{2M} \int_{\Omega} \mathcal{S}\left(x, |\nabla u_k|\right) \,\mathrm{d}x + C(\varepsilon, M) \int_{\Omega} \mathtt{M}_{\varepsilon}\left(x, \frac{|\phi|}{\lambda_k}\right) \,\mathrm{d}\nu_k \\
		& \leq \frac{1}{2} +  C(\varepsilon, M) \int_{\Omega} \mathtt{M}_{\varepsilon}\left(x, \frac{|\phi|}{\lambda_k}\right) \,\mathrm{d}\nu_k.
	\end{align*}
	Again, applying Proposition \ref{pro:norm-mod:relation} \textnormal{{(ii)}} and  passing to the  $\lim\sup_{k \to \infty}$ in the above estimate, we obtain
	\begin{equation}\label{sobo:est-4}
		\lim\sup_{k \to \infty} \|\phi \nabla u_k\|_{\mathcal{S}} \leq C \|\phi\|_{\mathtt{M}_{\varepsilon}, \nu}.
	\end{equation}
	Note that,
	\begin{align*}
		\|\nabla (\phi u_k)\|_{\mathcal{S}} \leq \|u_k \nabla \phi\|_{\mathcal{S}} + \| \phi \nabla u_k\|_{\mathcal{S}}.
	\end{align*}
	Since $u_k \to 0$ in $L^{\mathcal{S}}(\Omega)$, we have $|u_k \nabla \phi| \to 0$ in $L^{\mathcal{S}}(\Omega)$. Now, by passing to the $\lim\sup\limits_{k \to \infty}$ in the above inequality and using \eqref{sobo:est-4}, we get
	\begin{equation}\label{sobo:est-5}
		\lim\sup_{k \to \infty}  \|\nabla (\phi u_k)\|_{\mathcal{S}} \leq C \|\phi\|_{\mathtt{M}_{\varepsilon}, \nu}.
	\end{equation}
	Finally, combining \eqref{sobo:est-1},\eqref{sobo:est-new} and \eqref{sobo:est-5}, we complete the proof.
\end{proof}

\begin{remark}\label{control:rema}
	The continuity of $p$ and $q$ and the oscillation condition \eqref{main:assump-1}\textnormal{(ii)} imply that $\mathtt{M}_\varepsilon \ll \mathtt{M}_\varepsilon^\ast$ for $\varepsilon$ small enough. In order to see this, observe that the following inequality
	\begin{equation}\label{lower:compa}
		\mathtt{M}_\varepsilon^\ast(x, z) \geq C z^{\mathtt{m}_{-\varepsilon}^\ast(x)} \quad \text{for }(x,z) \in \Omega \times (1, \infty)
	\end{equation}
	holds for some $C>1.$ In light of the above inequality, to prove our claim it is enough to show that $ \mathtt{n}_\varepsilon(x) < \mathtt{m}_{-\varepsilon}^\ast(x)$ for all $x \in \overline{\Omega}$ and for some $\varepsilon>0$.
	Then, by using the continuity of $p$ and $q$ in $\overline{\Omega}$, we can find $\varepsilon > 0$  such that
	\begin{align*}
		\frac{\mathtt{n}_\varepsilon(x)}{\mathtt{m}_{-\varepsilon}(x)} =\frac{\max\{p(x), q(x) + \varepsilon\}}{\min\{p(x), q(x)-\varepsilon\}} < 1+ \frac{1}{N} \quad \text{in }  \overline{\Omega}.
	\end{align*}
	This gives
	\begin{equation}\label{control:est}
			\mathtt{n}_\varepsilon(x) < \mathtt{m}_{-\varepsilon}(x) + \frac{\mathtt{m}_{-\varepsilon}(x)}{N} < \frac{N \mathtt{m}_{-\varepsilon}(x)}{N-\mathtt{m}_{-\varepsilon}(x)} = \mathtt{m}_{-\varepsilon}^\ast(x) \quad \text{in }  \overline{\Omega}.
	\end{equation}
\end{remark}

\begin{lemma}\label{lem:non-tri:mes}
	Let $\nu$ be a non-negative, bounded Borel measure and
	\begin{equation}\label{control:measure}
		\|\phi\|_{\mathtt{M}_\varepsilon^\ast, \nu} \leq C \|\phi\|_{\mathtt{M}_\varepsilon, \nu}
	\end{equation}
	for every $\phi \in C_c^\infty(\Omega)$ and for some constant $C>0$. Then, the following hold:
	\begin{enumerate}
		\item[\textnormal{(i)}]
			There exists $\delta>0$ such that for all Borel sets $U \subset \overline{\Omega}$, either $\nu(U)=0$ or $\nu(U) \geq \delta$.
		\item[\textnormal{(ii)}]
			There exists a countable index set $I$, scalars $\{\nu_i\}_{i \in I}$ and points $\{x_i\}_{i \in I}$ such that
			\begin{align*}
				\nu= \sum_{i \in I} \nu_i \delta_{x_i}.
			\end{align*}
	\end{enumerate}
\end{lemma}

\begin{proof}
	It is easy to see that the inequality \eqref{control:measure} holds for characteristic functions on Borel sets. Since $\Omega$ is bounded, $p,q \in C(\overline{\Omega})$ and \eqref{control:est} holds true, there exists a finite cover $\{\Omega_i\}_{i \in I}$ of $\overline{\Omega}$ such that
	\begin{equation}\label{eq:parti:est}
		m_i:= \min_{\overline{\Omega}_i} \mathtt{m}_{-\varepsilon}^\ast(x)  > n_i:= \max_{\overline{\Omega}_i} \mathtt{n}_\varepsilon(x) \quad \text{for all }  i \in I.
	\end{equation}
	For any $U \subset \Omega$ being a Borel set, denote $
		V_i:= \Omega_{i} \cap U$ for $i \in I$.
	If $\nu(U) \geq 1$, we are done. Suppose $\nu(U) < 1$. By taking $\phi(x)= \chi_{V_i}(x) \nu(V_i)^\frac{-1}{m_i}$ for $i \in I$ in \eqref{control:measure} and using \eqref{lower:compa}, we deduce that
	\begin{align*}
		\int_{\Omega} \mathtt{M}_\varepsilon^\ast(x, \phi) \,\mathrm{d}\nu \geq C \int_{\Omega} \left(\chi_{V_i}(x) \nu(V_i)^\frac{-1}{m_i}\right)^{\mathtt{m}_{-\varepsilon}^\ast(x)} \,\mathrm{d}\nu \geq C \int_{\Omega} \left(\chi_{V_i}(x) \nu(V_i)^\frac{-1}{\mathtt{m}_{-\varepsilon}^\ast(x)}\right)^{\mathtt{m}_{-\varepsilon}^\ast(x)} \,\mathrm{d}\nu = C >1.
	\end{align*}
	It follows that $\|\chi_{V_i}\|_{\mathtt{M}_\varepsilon^\ast, \nu} \geq C \nu(V_i)^\frac{1}{m_i}$. Analogously, we get $\|\chi_{V_i}\|_{\mathtt{M}_\varepsilon, \nu} \leq \nu(V_i)^\frac{1}{n_i}$ for $i \in I$. Combining the above estimates with \eqref{control:measure}, we arrive at
	\begin{align*}
		\nu(V_i)^\frac{1}{m_i} \leq C_1 \nu(V_i)^\frac{1}{n_i}.
	\end{align*}
	From \eqref{eq:parti:est}, we know that $m_i > n_i$ for every $i \in I$, then either $\nu(V_i) =0$ or
	\begin{align*}
		\nu(V_i) \geq \max\left\{1, C_1^\frac{-m_i n_i}{m_i-n_i}\right\}.
	\end{align*}
	Therefore, either $\nu(U)=0$ or
	\begin{align*}
		\nu(U) = \sum_{i \in I} \nu(V_i)
		\geq \sum_{i \in I} \max\{1, C_1^\frac{-m_i n_i}{m_i-n_i}\}:= \delta.
	\end{align*}
	Note that the constant $\delta$ is independent from the choice $U$.
\end{proof}

\begin{lemma}\label{lem:countable-dirac}
	Let $\mathtt{M}_\varepsilon^\ast$ and $\mathtt{M}_\varepsilon$ be the functions defined in \eqref{pert-MOF}. Then, there exist an at most countable index set $I$, families $\{x_i\}_{i \in I}$ of distinct points in $\Omega$ and scalars $\nu_{i} \in (0, \infty)$, such that
	\begin{align*}
		\nu = \sum_{ i \in I} \nu_{i} \delta_{x_i}.
	\end{align*}
\end{lemma}

\begin{proof}
	The reverse H\"older inequality in \eqref{reverse:holder} implies that the measure $\nu$ is absolutely continuous with respect to the measure $\mu$, i.e.\,$\nu \ll \mu$. Thus, by applying the Radon-Nikodym theorem, there exists a non-negative function $f$ such that $f \in L^1_\mu(\Omega)$ and
	\begin{equation}\label{def:nu}
		\nu = \mu\lfloor f \quad \text{i.e.} \quad \nu(E) = \int_{E} f \,\mathrm{d}\mu \quad \text{for any Borel set }  E.
	\end{equation}
	Let $U$ be a Borel set. By repeating the same arguments as in the proof of Lemma \ref{lem:non-tri:mes} with the covering $\{\Omega_i\}_{i \in I}$ such that $\nu(\Omega_i) + \mu(\Omega_i) < 1$ and taking $\phi(x)= \chi_{V_i}(x)$, $V_i = \Omega_i \cap U$ for $i \in I$ in \eqref{reverse:holder}, we obtain
	\begin{align}\label{meas:control}
		\nu(V_i)^\frac{1}{m_i}\leq C \mu(V_i)^\frac{1}{n_i},
	\end{align}
	where $m_i$, $n_i$ are defined in \eqref{eq:parti:est}.
	Using \eqref{meas:control} and the fact that $n_i < m_i$, we have
	\begin{align*}
		\fint_{U} f \,\mathrm{d}\mu = \frac{\nu(U)}{\mu(U)} \leq \sum_{\substack{i \in I \\ \mu(V_i) \neq 0 }} \frac{\nu(V_i)}{\mu(V_i)} = \sum_{\substack{i \in I \\ \mu(V_i) \neq 0 }} \mu(V_i)^{\frac{m_i}{n_i}-1}  \leq C(\mu(\Omega), p^\pm, q^\pm, N).
	\end{align*}
	Therefore, by Lebesgue's differentiation theorem, $f \in L^\infty_\mu(\Omega)$. On the other hand, we can decompose the measure $\mu$ as
	\begin{align*}
		\mu= \mu_0 + \mu_1 \quad \text{such that} \quad \mu_0(E) = \mu(E \cap X_0)  \text{ and }  \mu_1(E) = \mu(E \cap X_f),
	\end{align*}
	where
	\begin{align*}
		X_0:= \{x \in \Omega\colon  f(x)=0\} \quad \text{ and} \quad X_f:= \{x \in \Omega\colon  f(x) \neq 0\}.
	\end{align*}
	Observe that the measures $\mu_0$ and $\mu_1$ are singular and absolutely continuous with respect to $\nu$, respectively. Indeed, the sets $X_0$ and $X_f$ are mutually disjoint and $\mu_0(X_f) = 0 = \mu_1(X_0)$. Let $E$ be a Borel set such that $\nu(E)=0$. Then, from \eqref{def:nu} and using the fact that $f>0$ in $X_f$, we get
	\begin{align*}
		0= \nu(E) = \int_{E} f \,\mathrm{d}\mu = \int_{E \cap X_f} f \,\mathrm{d}\mu \quad \Longrightarrow \quad \mu_1(E) =\mu(E \cap X_f) =0.
	\end{align*}
	Now, by applying the Radon-Nikodym theorem to the measures $\mu_1$ and $\nu$, there exists a non-negative function $g$ such that $g \in L^1_\nu(\Omega)$ and $\mu_1 = \nu \lfloor g$ , and using the fact that $\nu(X_0) =0$, we get
	\begin{align*}
		\mu_1(E) = \int_{E} g \,\mathrm{d}\nu = \int_{E \cap X_f} g \,\mathrm{d}\nu = \int_{E} \tilde{g} \,\mathrm{d}\nu, \quad  \text{with} \quad \tilde{g}(x):= g(x) \chi_{X_f}(x)
	\end{align*}
	for any Borel set $E$. Finally, we have $\mu_0(X_{\tilde{g}}) =0$ and $\mu= \nu \lfloor \tilde{g} + \mu_0$. Using the fact $\mathtt{m}_{-\varepsilon}^\ast(x) > \mathtt{n}_\varepsilon(x)$ proved in Remark \ref{control:rema}, we define
	\begin{align*}
		\Phi(x,t) :=
		\begin{cases}
			t^\frac{1}{\mathtt{n}_\varepsilon^\ast(x) - \mathtt{m}_-(x)}           &  \text{if }  t \leq 1, \\
			t^\frac{1}{\mathtt{m}_{-\varepsilon}^\ast(x) - \mathtt{n}_\varepsilon(x)} &  \text{if }  t > 1,
		\end{cases}
	\end{align*}
	such that
	\begin{equation}\label{def:pi}
		\Pi(x, \tilde{g}) := \max\left\{ \Phi(\tilde{g})^{\mathtt{n}_\varepsilon(x)}, \Phi(\tilde{g})^{\mathtt{m}_-(x)} \right\} \tilde{g} = \min\left\{ \Phi(\tilde{g})^{\mathtt{n}_\varepsilon^\ast(x)}, \Phi(\tilde{g})^{\mathtt{m}_{-\varepsilon}^\ast(x)}  \right\}.
	\end{equation}
	By setting $\tilde{\nu}_k:= \Pi(x, \tilde{g}) \chi_{\{\tilde{g} \leq k\}} \nu$, we are going to prove that $\tilde{\nu}_k$ is given by a finite number of Dirac masses and this will prove that $\nu \chi_{\{\tilde{g} \leq k\}}$ is a finite number of Dirac masses for all $k< \infty$ and since $\nu(\{\tilde{g} = + \infty\}) =0$, the claim on $\nu$ is true. Taking $\phi:= \Phi(\tilde{g})  \psi \chi_{\{\tilde{g} \leq k\}}$ with $\psi \in C_c^\infty(\Omega)$ as a test function in the reverse H\"older inequality, we get
	\begin{equation}\label{new:reverse:1}
		\begin{aligned}
			\int_{\Omega} \mathtt{M}_\varepsilon\left(x, \frac{\Phi(\tilde{g}) \psi \chi_{\{\tilde{g} \leq k\}}}{\lambda}\right) \,\mathrm{d}\mu & =  \int_{\Omega} \mathtt{M}_\varepsilon\left(x, \frac{\Phi(\tilde{g}) \psi \chi_{\{\tilde{g} \leq k\}}}{\lambda}\right) \tilde{g} \,\mathrm{d}\nu  + \int_{\Omega} \mathtt{M}_\varepsilon\left(x, \frac{\Phi(\tilde{g}) \psi \chi_{\{\tilde{g} \leq k\}}}{\lambda}\right) \,\mathrm{d}\mu_0 \\
			& = \int_{\Omega} \mathtt{M}_\varepsilon\left(x, \frac{\Phi(\tilde{g}) \psi \chi_{\{\tilde{g} \leq k\}}}{\lambda}\right) \tilde{g} \,\mathrm{d}\nu \\
			& \leq \int_{\Omega} \mathtt{M}_\varepsilon\left(x, \frac{\psi}{\lambda}\right) \mathtt{M}_\varepsilon\left(x, \Phi(\tilde{g})\right) \tilde{g}  \chi_{\{\tilde{g} \leq k\}} \,\mathrm{d}\nu \\
			& \leq \int_{\Omega} \mathtt{M}_\varepsilon\left(x, \frac{\psi}{\lambda}\right) \Pi(x, \tilde{g})   \chi_{\{\tilde{g} \leq k\}} \,\mathrm{d}\nu,
		\end{aligned}
	\end{equation}
	where the last two inequalities follow from the fact that $\mathtt{M}_\varepsilon$ is a sub-multiplicative function and the definition of the function $\Pi$ in \eqref{def:pi}. Moreover, we have
	\begin{equation}\label{new:reverse:2}
		\begin{aligned}
			\int_{\Omega} \mathtt{M}_\varepsilon^\ast\left(x,  \frac{\Phi(\tilde{g}) \psi \chi_{\{\tilde{g} \leq k\}}}{\lambda} \right) \,\mathrm{d}\nu & \geq C \int_\Omega \mathtt{M}_\varepsilon^\ast\left(x, \frac{\psi}{\lambda}\right) \min\left\{ \left(\Phi(\tilde{g})\right)^{\mathtt{n}_\varepsilon^\ast(x)}, \left(\Phi(\tilde{g})\right)^{\mathtt{m}_{-\varepsilon}^\ast(x)}  \right\}  \chi_{\{\tilde{g} \leq k\}} \,\mathrm{d}\nu \\
			& = C \int_{\Omega} \mathtt{M}_\varepsilon^\ast\left(x, \frac{\psi}{\lambda}\right) \Pi(x, \tilde{g})   \chi_{\{\tilde{g} \leq k\}} \,\mathrm{d}\nu,
		\end{aligned}
	\end{equation}
	where $C = C(\varepsilon, s^\pm, q^\pm, N)$. Combining \eqref{new:reverse:1} and \eqref{new:reverse:2} yields the following reverse inequality
	\begin{align*}
		\|\psi\|_{\mathtt{M}_\varepsilon^\ast, \tilde{\nu}_k} \leq C \|\psi\|_{\mathtt{M}_\varepsilon, \tilde{\nu}_k}.
	\end{align*}
	Using Lemma \ref{lem:non-tri:mes}, there exist scalars $\{K_i\}_{i \in I_k}$ and points $\{x_i\}_{i \in I_k}$ such that $\tilde{\nu}_k= \sum_{i \in I_k} K_i  \delta_{x_i}$. Notice that, the sequence of measure $\tilde{\nu}_k \nearrow \Pi(x, \tilde{g}) \nu$. Letting $k \to \infty$ and using the fact that $\tilde{g} \in L^1_{\nu}(\Omega)$, there exists an at most countable set $I$, families $\{x_i\}_{i \in I}$ of distinct points in $\Omega$ and scalars $\nu_{i} \in (0, \infty)$ such that $ \nu = \sum_{ i \in I} \nu_{i} \delta_{x_i}$ and this concludes the proof.
\end{proof}

Now we are ready to give the proof of Theorem \ref{concentration:compactness}.

\begin{proof}[Proof of Theorem \ref{concentration:compactness}]
	From Lemmas \ref{reverse:holder:ineq}--\ref{lem:countable-dirac}, we know that
	\begin{align*}
		\mathcal{S}^\ast(x, |u_k|) \rightharpoonup \nu = \sum_{ i \in I} \nu_{i} \delta_{x_i} \quad  \text{weakly-}\ast  \text{ in the sense of measures}.
	\end{align*}
	By applying the Br\'{e}zis-Lieb lemma (Lemma \ref{BL-lemma}) to $\mathcal{S}^\ast$, we get
	\begin{align*}
		\lim_{k \to \infty} \left(\int_{\Omega} \mathcal{S}^\ast(x, |v_k|) \phi \,\mathrm{d}x - \int_{\Omega} \mathcal{S}^\ast(x, |u_k|) \phi \,\mathrm{d}x \right) = \int_{\Omega} \mathcal{S}^\ast(x, |v|) \phi \,\mathrm{d}x \quad \text{for every }\phi \in L^\infty(\Omega),
	\end{align*}
	which further gives the following representation
	\begin{align*}
		\mathcal{S}^\ast(x, |v_k|)  \rightharpoonup \Theta = \mathcal{S}^\ast(x, |v|)  + \sum_{ i \in I} \nu_{i} \delta_{x_i}  \text{weakly-}\ast  \text{ in the sense of measures}.
	\end{align*}
	Let $\phi \in C_c^\infty(\Omega)$ such that $\phi(0)=1$, $0 \leq \phi \leq 1$ and $\operatorname{supp}(\phi) \subset B_1(0)$. Define $\phi_{\gamma, i}(x) = \phi\left(\frac{x-x_i}{\gamma}\right)$ for each $i \in I$, $x \in \mathbb{R}^N$ and $\gamma>0$. Taking $\phi_{\gamma, i}$ as a test function in the reverse H\"older inequality \eqref{reverse:holder} for measures $\mu$ and $\nu$, we obtain
	\begin{equation}\label{reverse:proto}
		\|\phi_{\gamma, i}\|_{\mathtt{M}_\varepsilon^\ast, \nu} \leq C \|\phi_{\gamma, i}\|_{\mathtt{M}_\varepsilon, \mu}.
	\end{equation}
	With the representation of the measure $\nu$ and for any $i \in I$, we estimate the left- and right-hand terms of the above inequality by
	\begin{align*}
		\varrho_{\mathtt{M}_\varepsilon^\ast}(\phi_{\gamma,i}):= \int_{\Omega} \mathtt{M}_\varepsilon^\ast(x, \phi_{\gamma,i}(x)) \,\mathrm{d}\nu =  \sum_{i \in I} \nu_i \mathtt{M}_\varepsilon^\ast(x_i, \phi_{\gamma,i}(x_j)) \geq \nu_i
	\end{align*}
	If $\varrho_{\mathtt{M}_\varepsilon^\ast}(\phi_{\gamma,i}) \leq 1$, then
	\begin{equation} \label{reverse:proto:est-1}
		\|\phi_{\gamma, i}\|_{\mathtt{M}_\varepsilon^\ast, \nu} \geq \nu_i^{\frac{1}{\alpha_\gamma^-}}
	\end{equation}
	Analogously, if $\varrho_{\mathtt{M}_\varepsilon^\ast}(\phi_{\gamma,i}) > 1$, then
	\begin{align*}
		\|\phi_{\gamma, i}\|_{\mathtt{M}_\varepsilon^\ast, \nu} \geq \nu_i^{\frac{1}{\alpha_\gamma^+}}.
	\end{align*}
	Similarly,
	\begin{equation}\label{reverse:proto:est-2}
		\begin{aligned}
			1=\varrho_{\mathtt{M}_\varepsilon}\left(\frac{\phi_{\gamma,i}}{\|\phi_{\gamma,i}\|_{\mathtt{M}_\varepsilon, \mu}}\right) & := \int_{\Omega} \mathtt{M}_\varepsilon\left(x, \frac{\phi_{\gamma,i}(x)}{\|\phi_{\gamma,i}\|_{\mathtt{M}_\varepsilon, \mu}}\right) \,\mathrm{d}\mu \leq \int_{B_{\gamma}(x_i)} \mathtt{M}_\varepsilon\left(x, \frac{1}{\|\phi_{\gamma,i}\|_{\mathtt{M}_\varepsilon, \nu}}\right) \,\mathrm{d}\mu \\
			& \leq \mu(B_{\gamma}(x_i)) \max\left\{\|\phi_{\gamma_j,j}\|^{-\beta_{\gamma_j}^-}_{\mathtt{M}_\varepsilon, \mu}, \|\phi_{\gamma_j,j}\|^{-\beta_{\gamma_j}^+}_{\mathtt{M}_\varepsilon, \mu}\right\}
		\end{aligned}
	\end{equation}
	where
	\begin{align*}
		\alpha_{\gamma}^- &:= \min\limits_{B_{\gamma}(x_i)} \mathtt{m}_{-\varepsilon}^\ast(x), &   \alpha_{\gamma}^+ &:= \max\limits_{ B_{\gamma}(x_i)} \mathtt{n}_\varepsilon^\ast(x),\\
		\beta_{\gamma}^- &:= \min\limits_{ B_{\gamma}(x_i)} \mathtt{m}_-(x), & \beta_{\gamma}^+ &:= \max\limits_{B_{\gamma}(x_i)} \mathtt{n}_\varepsilon(x).
	\end{align*}
	Now, by collecting the estimates in \eqref{reverse:proto:est-1} and \eqref{reverse:proto:est-2} in \eqref{reverse:proto} and by letting $\gamma \to 0$, we get
	\begin{align*}
		\min\{\nu_j^\frac{1}{\mathtt{n}_\varepsilon^\ast(x_j)}, \nu_j^\frac{1}{\mathtt{m}_{-\varepsilon}^\ast(x_j)}\} \leq C \max\left\{\mu_j^\frac{1}{\mathtt{m}_-(x_j)}, \mu_j^\frac{1}{\mathtt{n}_\varepsilon(x_j)}\right\},
	\end{align*}
	where $\mu_j:= \mu(x_j)= \lim\limits_{\gamma_j \to 0} \mu(B_{\gamma_j}(x_j))$. In particular, $\{x_i\}_{i \in I}$ are atoms of $\nu$. Note that for any $\phi \in C(\overline{\Omega})$, $\phi \geq 0$, the functional
	\begin{align*}
		u \to \int_{\Omega} \left(a(x) |\nabla u|^{p(x)} + b(x) |\nabla u|^{q(x)} \log^{s(x)}(1+| \nabla u|)\right) \phi \,\mathrm{d}x
	\end{align*}
	is convex and continuous in $W_0^{1, \mathcal{S}}(\Omega)$. Hence, it is weakly lower semicontinuous and therefore,
	\begin{align*}
		\int_{\Omega} & \left(a(x) |\nabla u|^{p(x)} + b(x) |\nabla u|^{q(x)} \log^{s(x)}(1+| \nabla u|)\right) \phi \,\mathrm{d}x \\
		& \leq \liminf_{t \to \infty} \int_{\Omega} \left(a(x) |\nabla u_k|^{p(x)} + b(x) |\nabla u_k|^{q(x)} \log^{s(x)}(1+| \nabla u_k|)\right) \phi \,\mathrm{d}x = \int_{\Omega} \phi \,\mathrm{d}\mu.
	\end{align*}
	Thus, $\mu \geq \mathcal{S}(x, |\nabla u|)$. Finally, by extracting $\mu$ to its atoms, we conclude our result.
\end{proof}

\section{Properties of the energy functional and the double phase operator}\label{Section-5}

In this section we investigate the properties of the associated energy functional and the logarithmic operator given in \eqref{main:prob}. We begin by defining the energy functional $\mathcal{E}_{\Lambda, \lambda}\colon W_0^{1, \mathcal{S}}(\Omega) \to \mathbb{R}$ corresponding to our new logarithmic double phase operator as
\begin{align*}
	\mathcal{E}_{\Lambda, \lambda}(u) = \mathcal{E}_1(u) - \Lambda \mathcal{E}_2(u) - \lambda \mathcal{E}_3(u),
\end{align*}
where
\begin{align*}
	\mathcal{E}_1(u) = \int_{\Omega} \mathcal{M}(x, |\nabla u|) \,\mathrm{d}x, \qquad
	\mathcal{E}_2(u) = \int_{\Omega} \mathcal{M}^\ast(x, |\nabla u|) \,\mathrm{d}x,  \qquad
	\mathcal{E}_3(u) = \int_{\Omega} \mathcal{M}_\star(x, |\nabla u|) \,\mathrm{d}x.
\end{align*}
Furthermore we introduce $\mathcal{J}\colon  W_0^{1,\mathcal{S}}(\Omega) \to (W_0^{1,\mathcal{S}}(\Omega))^\ast$ which is given by
\begin{equation}\label{multi-phase}
	\langle \mathcal{J}(u), \phi\rangle_{\mathcal{S}}:= \langle \mathcal{J}_1(u), \phi\rangle_{\mathcal{S}} - \Lambda \langle \mathcal{J}_2(u), \phi\rangle_{\mathcal{S}} - \lambda \langle \mathcal{J}_3(u), \phi\rangle_{\mathcal{S}}
	\quad \text{for }  u, \phi \in W_0^{1, \mathcal{S}}(\Omega),
\end{equation}
where
\begin{align*}
	\langle \mathcal{J}_1(u), \phi\rangle_{\mathcal{S}}
	& := \int_{\Omega} a(x) |\nabla u|^{p(x)-2} \nabla u \cdot \nabla \phi \,\mathrm{d}x  \\
	& \quad+ \int_{\Omega} b(x) |\nabla u|^{q(x)-2} \log^{s(x)-1 }(1+ |\nabla u|) \left(\log(1+ |\nabla u|) + \frac{s(x)}{q(x)} \frac{|\nabla u|}{1+ |\nabla u|} \right) \nabla u \cdot \nabla \phi \,\mathrm{d}x,\\
	\langle \mathcal{J}_2(u), \phi\rangle_{\mathcal{S}} & := \int_{\Omega} \left(a(x)\right)^\frac{p^\ast(x)}{p(x)} |u|^{p^\ast(x)-2}  u  \phi \,\mathrm{d}x \\
	& \quad + \int_{\Omega} \left(b(x) \log^{s(x)}(1+ |u|) \right)^\frac{q^\ast(x)}{q(x)} |u|^{q^\ast(x)-2} \left(1+  \frac{s(x) q^\ast(x)}{q(x)} \frac{|u|}{\log(1+ |u|)(1+ |u|)}\right) u \phi \,\mathrm{d}x,\\
	\langle \mathcal{J}_3(u), \phi\rangle_{\mathcal{S}} & := \int_{\Omega} \left(a(x)\right)^\frac{p_\star(x)}{p(x)} |u|^{p_\star(x)-2}  u  \phi \,\mathrm{d}x \\
	& \quad + \int_{\Omega} \left(b(x) \log^{s_\star(x)}(1+ |u|) \right)^\frac{q_\star(x)}{q(x)} |u|^{q_\star(x)-2} \left(1+  \frac{s_\star(x) q_\star(x)}{q(x)} \frac{|u|}{\log(1+ |u|)(1+ |u|)}\right) u \phi \,\mathrm{d}x,
\end{align*}
where $\langle \cdot, \cdot \rangle_{\mathcal{S}}$ denotes the duality paring between the space $W_0^{1, \mathcal{S}}(\Omega)$ and its dual space $(W_0^{1, \mathcal{S}}(\Omega))^\ast$. Here $\mathcal{J}_1$ is the new logarithmic double phase operator.

\begin{proposition}\label{well-defined-1}
	Let \eqref{main:assump} and \eqref{main:assump-1} be satisfied. Then, the energy functional $\mathcal{E}_1$ is well-defined and belongs to the class $C^1$ with $\mathcal{E}_1'(u) = \mathcal{J}_1(u)$.
\end{proposition}

\begin{proof}
	Observe that, for any $u \in W_0^{1, \mathcal{S}}(\Omega)$, we have
	\begin{align*}
		0 \leq \frac{\varrho_{\mathcal{S}}(|\nabla u|)}{\mathfrak{n}^+_0} \leq \mathcal{E}_1(u) \leq  \frac{\varrho_{\mathcal{S}}(|\nabla u|)}{\mathfrak{m}^-_0} < \infty \quad \text{with $\mathfrak{n}^+_0:= \max\limits_{x \in \overline{\Omega}}\mathfrak{n}_0(x)$ and $\mathfrak{m}^-_0:= \min\limits_{x \in \overline{\Omega}}\mathfrak{m}_0(x)$.}
	\end{align*}
	Hence, the functional $\mathcal{E}_1$ is well defined. Let $u, \phi \in W_0^{1, \mathcal{S}}(\Omega)$. By using the mean-value theorem, we get
	\begin{align*}
		\frac{\mathcal{E}_1(u+ t\phi) - \mathcal{E}_1(u)}{t} = J_1 + J_2 + J_3.
	\end{align*}
	where
	\begin{align*}
		J_1&:=  \int_{\Omega} a(x) |\nabla u + \eta_{x,t} t \nabla \phi|^{p(x)-2} (\nabla u + \eta_{x,t}  t \nabla \phi) \cdot \nabla \phi \,\mathrm{d}x,\\
		J_2&:=\int_{\Omega} b(x) |\nabla u + \eta_{x,t} t \nabla \phi|^{q(x)-2} \log^{s(x)}(1 + |\nabla u + \eta_{x,t}  t \nabla \phi|) (\nabla u + \eta_{x,t}  t \nabla \phi) \cdot \nabla \phi \,\mathrm{d}x,\\
		J_3&:= \int_{\Omega} \frac{b(x)s(x)}{q(x)} \frac{|\nabla u + \eta_{x,t} t \nabla \phi|^{q(x)-1}}{ 1+ |\nabla u + \eta_{x,t}  t \nabla \phi|} \log^{s(x)-1}\left(1 + |\nabla u + \eta_{x,t}  t \nabla \phi|\right) (\nabla u + \eta_{x,t}  t \nabla \phi) \cdot \nabla \phi \,\mathrm{d}x
	\end{align*}
	for some $t \in \mathbb{R}$ and $\eta_{x,t} \in (0,1)$.\\
	\textbf{Claim:} $\lim\limits_{t \to 0} J_1 + J_2 + J_3 = \langle \mathcal{J}(u) , \phi\rangle$.

	To prove the above claim, we have to find integrable functions dominating the integrand of $J_i$, $i=1,2,3$ and use Lebesgue's dominated convergence theorem. For $0< |t| \leq t_0 \leq 1$, we have
	\begin{align*}
		\left|  a(x) |\nabla u + \eta_{x,t} t \nabla \phi|^{p(x)-2} (\nabla u + \eta_{x,t}  t \nabla \phi) \cdot \nabla \phi \right| & \leq a(x) |\nabla u + \eta_{x,t} t \nabla \phi|^{p(x)-1} |\nabla \phi| \\
		& \leq 2^{p^+ -1} a(x) \left(|\nabla u|^{p(x)-1} + t_0 |\nabla \phi|^{p(x)-1} \right) |\nabla \phi|
	\end{align*}
	and
	\begin{align*}
		\int_{\Omega} a(x) |\nabla u|^{p(x)-1} |\nabla \phi| \,\mathrm{d}x & \leq  \int_{\{|\nabla u| \geq |\nabla \phi| \}} a(x) |\nabla u|^{p(x)-1} |\nabla \phi| \,\mathrm{d}x +  \int_{\{|\nabla u| < |\nabla \phi|\}} a(x) |\nabla u|^{p(x)-1} |\nabla \phi| \,\mathrm{d}x  \\
		& \leq \int_{\Omega} a(x) |\nabla u|^{p(x)} \,\mathrm{d}x + \int_{\Omega} a(x) |\nabla \phi|^{p(x)} \,\mathrm{d}x \leq \varrho_{\mathcal{S}}(|\nabla u|) + \varrho_{\mathcal{S}}(|\nabla \phi|) < \infty.
	\end{align*}
	To estimate the integrands in $J_2$ and $J_3$, we use condition $q(x) + s(x) \geq r >1$ in \eqref{main:assump}. In particular, this implies that
	\begin{equation}\label{incr:prop}
		t^{q(x)-1} \log^{s(x)}(1+t)  \quad\text{is increasing for }  t \geq 0,
	\end{equation}
	because
	\begin{equation}\label{control:integrand}
		0 \leq s(x) t +  (q(x)-1) (1+t) \log(1+t) \quad \text{for all }  t \geq 0  \ \text{and } x \in \Omega.
	\end{equation}
	Now, by using \eqref{control:integrand} and  Proposition \ref{pro:norm-mod:relation}, we obtain the estimate
	\begin{align*}
		& \left|\frac{b(x)s(x)}{q(x)} \frac{|\nabla u + \eta_{x,t} t \nabla \phi|^{q(x)-1}}{ 1+ |\nabla u + \eta_{x,t}  t \nabla \phi|} \log^{s(x)-1}\left(1 + |\nabla u + \eta_{x,t}  t \nabla \phi|\right) (\nabla u + \eta_{x,t}  t \nabla \phi) \cdot \nabla \phi\right| \\
		& \leq b(x) |s(x)| \left(|\nabla u + \eta_{x,t} t \nabla \phi|\right)^{q(x)-1} \log^{s(x)}(1 + |\nabla u + \eta_{x,t}  t \nabla \phi|) |\nabla \phi|     \\
		& \leq b(x) |s(x)| \left(|\nabla u| + \eta_{x,t} t |\nabla \phi|\right)^{q(x)-1} \log^{s(x)}(1 + |\nabla u| + \eta_{x,t}  t |\nabla \phi|) |\nabla \phi|  \\
		& \leq C(r^+, |s|^+) b(x) \left(|\nabla u|^{q(x)-1} + (\eta_{x,t} t)^{q(x)-1} |\nabla \phi|^{q(x)-1} \right) |\nabla \phi|   \\
		& \qquad \times
			\begin{cases}
				\log^{s(x)}(1 + |\nabla u|) + \log^{s(x)}(1 + \eta_{x,t} t |\nabla \phi|) &  \text{if }  s(x) \geq 0,   \\
				\log^{s(x)}(1 + \eta_{x,t} t |\nabla \phi|) &  \text{if }  s(x) < 0  \text{ and }  |\nabla u| \leq \eta_{x,t} t |\nabla \phi|, \\
				\log^{s(x)}(1 + |\nabla u|)   &  \text{if }  s(x) < 0  \text{ and }  |\nabla u| \geq \eta_{x,t} t |\nabla \phi|.
			\end{cases}
			\\
			 & \leq C'(r^+, s^+) b(x)
			\begin{cases}
				|\nabla u|^{q(x)} \log^{s(x)}(1 + |\nabla u|) &  \text{if }  |\nabla u| \geq |\nabla \phi|, \\
				|\nabla \phi|^{q(x)} \log^{s(x)}(1 +  |\nabla \phi|) &  \text{if }  |\nabla u| \leq |\nabla \phi|,
			\end{cases}
	\end{align*}
	and
	\begin{align*}
		&\int_{|\nabla u| \geq |\nabla \phi|}  b(x) |\nabla u|^{q(x)} \log^{s(x)}(1 + |\nabla u|) \,\mathrm{d}x + \int_{|\nabla u| \leq |\nabla \phi|} b(x) |\nabla \phi|^{q(x)} \log^{s(x)}(1 + |\nabla \phi|) \,\mathrm{d}x \\
		& \leq \varrho_{\mathcal{S}}(|\nabla u|) +  \varrho_{\mathcal{S}}(|\nabla \phi|) < \infty.
	\end{align*}
	Finally, by using
	\begin{align*}
		a(x) |\nabla u + \eta_{x,t} t \nabla \phi|^{p(x)-2} (\nabla u + \eta_{x,t}  t \nabla \phi) \cdot \nabla \phi \to  a(x) |\nabla u|^{p(x)-2} \nabla u \cdot \nabla \phi  \quad \text{a.e.\,in }\Omega
	\end{align*}
	and Lebesgue's dominated convergence theorem, we obtain the required claim.

	For the $C^1$-property, let $u_n \to u$ in $W_0^{1, \mathcal{S}}(\Omega)$ and $\phi \in W_0^{1, \mathcal{S}}(\Omega)$ with $\|\nabla \phi\|_{\mathcal{S}} \leq 1$ and we claim that
	\begin{align*}
		\langle \mathcal{J}_1(u_n) - \mathcal{J}_1(u), \phi \rangle \to 0 \quad \text{as }  n \to \infty.
	\end{align*}
	For the following computations, we define
	\begin{align*}
		\Omega_0 & = \{x \in \Omega\colon  |\nabla u| =0\},    \\
		g(u(x)) & =
		\begin{cases}
			\left[ \log (1 + |\nabla u|) + \frac{s(x)|\nabla u|}{q(x) (1 + |\nabla u|)}\right] \log^{s(x)-1}(1+ |\nabla u|) |\nabla u|^{q(x) -2 }\nabla u, & \text{if }  \Omega \setminus \Omega_0, \\
			0 & \text{if }  x \in \Omega_0,
		\end{cases}
		\\
		v_n(x)   & = g(u_n) - g(u)  \\
		h_n(x)  & = \log^{\frac{s(x)}{q(x)}} ( 1 + \max\{ |\nabla u|, |\nabla u_n| \}),  \\
		\Omega_u^{\geq}     & = \{ x \in \Omega  \setminus \Omega_0 \colon  s(x) \geq 0,   \max\{ |\nabla \phi|, |\nabla u|, |\nabla u_n| \} = |\nabla u|\},   \\
		\Omega_u^{<}        & = \{ x \in \Omega  \setminus \Omega_0 \colon  s(x) < 0,   \max\{ |\nabla \phi|, |\nabla u|, |\nabla u_n| \} = |\nabla u|\},  \\
		\Omega_{u_n}^{\geq} & = \{ x \in \Omega  \setminus \Omega_0
		\colon s(x) \geq 0, \,  |\nabla u| < \max\{ |\nabla \phi|, |\nabla u|, |\nabla u_n| \} = |\nabla u_n|\},   \\
		\Omega_{u_n}^{<}    & = \{ x \in \Omega  \setminus \Omega_0
		\colon s(x) < 0, \,  |\nabla u| < \max\{ |\nabla \phi|, |\nabla u|, |\nabla u_n| \} = |\nabla u_n|\},  \\
		\Omega_\phi^{\geq}  & = \{ x \in \Omega  \setminus \Omega_0\colon s(x) \geq 0,\,  |\nabla u|, |\nabla u_n| < \max\{ |\nabla \phi|, |\nabla u|, |\nabla u_n| \} = |\nabla \phi|\}, \\
		\Omega_\phi^{<}   & = \{ x \in \Omega  \setminus \Omega_0 \colon s(x) < 0, \,  |\nabla u|, |\nabla u_n| < \max\{ |\nabla \phi|, |\nabla u|, |\nabla u_n| \} = |\nabla \phi|\}.
	\end{align*}
	Using \eqref{incr:prop}, inequality $(1+t) \log(1+t) \geq t$ for $t \geq 0$ and Young's inequality, we have
	\begin{align*}
		|b(x) g(u_n) \nabla \phi| & \leq |b(x)| \left[ \log (1 + |\nabla u_n|) + \frac{|s(x)||\nabla u_n|}{q(x) (1 + |\nabla u_n|)}\right] \log^{s(x)-1}(1+ |\nabla u_n|)|\nabla u_n|^{q(x)-1}|\nabla \phi| \\
		& \leq C_0 b(x) \log^{s(x)}(1+ |\nabla u_n|)|\nabla u_n|^{q(x)-1}|\nabla \phi|   \\
		& \leq C_1 b(x) \log^{s(x)}(1+ |\nabla u_n|)|\nabla u_n|^{q(x)} +  C_2 b(x) \log^{s(x)}(1+ |\nabla \phi|) |\nabla \phi|^{q(x)} \in L^1(\Omega),
	\end{align*}
	where $C_i$ are independent of $n$. Therefore, using $\nabla u_n \to \nabla u$ a.e., it is easy to see that
	\begin{align*}
		\int_{\Omega_0} b(x) v_n  \cdot \nabla \phi \,\mathrm{d}x  \to 0.
	\end{align*}
	By H\"older's inequality in $L^{q(\cdot)}(\Omega \setminus \Omega_0)$, we have
	\begin{align*}
		\left|\int_{\Omega  \setminus \Omega_0} b(x) v_n  \cdot \nabla \phi \,\mathrm{d}x \right| \leq 2 \left\|(b (x))^{\frac{q(x) - 1}{q(x)}} \frac{|v_n|}{h_n}\right\|_{\frac{q(\cdot)}{q(\cdot) - 1}}  \left\|(b(x))^{1/q(x)} |\nabla \phi| h_n\right\|_{q(\cdot)}.
	\end{align*}
	The second factor is uniformly bounded in $n$ and $v$ by Proposition \ref{pro:norm-mod:relation} \textnormal{(v)} and
	\begin{align*}
		\int_{\Omega \setminus \Omega_0} \left( (b(x))^{1/q(x)} |\nabla \phi| h_n \right)^{q(x)} \,\mathrm{d}x
		& \leq  \int_{\Omega_u^{\geq} \cup \Omega_{\phi}^{<} \cup \Omega_{u_n}^{<}} b(x) |\nabla \phi|^{q(x)} \log^{s(x)}( 1 + |\nabla u| ) \,\mathrm{d}x \\
		& \quad + \int_{\Omega_u^{<} \cup \Omega_{u_n}^{\geq}} b(x) |\nabla \phi|^{q(x)} \log^{s(x)}( 1 + |\nabla u_n| ) \,\mathrm{d}x    \\
		& \quad +  \int_{\Omega_\phi^{\geq}} b(x) |\nabla \phi|^{q(x)} \log^{s(x)}( 1 + |\nabla \phi| ) \,\mathrm{d}x   \\
		& \leq \varrho_{\mathcal{S}}(\nabla u) + \varrho_{\mathcal{S}}(\nabla u_n) + \varrho_{\mathcal{S}}(\nabla \phi) < + \infty.
	\end{align*}
	Therefore, we only need to prove that the first factor converges to zero. By Proposition \ref{pro:norm-mod:relation}  \textnormal{(v)}, it is enough to see that this happens in the modular of $L^{\frac{q(\cdot)}{q(\cdot) - 1}}(\Omega \setminus \Omega_0)$, that is
	\begin{align*}
		\int_{\Omega \setminus \Omega_0} \left( (b (x))^{\frac{q(x) - 1}{q(x)}} \frac{|v_n|}{h_n} \right)^{\frac{q(x)}{q(x)-1}} \,\mathrm{d}x
		= \int_{\Omega \setminus \Omega_0} b(x) \left(  \frac{|v_n|}{h_n} \right) ^{\frac{q(x)}{q(x) - 1}}  \,\mathrm{d}x \; \xrightarrow{n \to \infty} 0.
	\end{align*}
	We prove this convergence by using Vitali's convergence theorem.  For the uniform integrability, using \eqref{incr:prop} and the inequality $(1+t) \log(1+t) \geq t$ for $t \geq 0$ note that
	\begin{align*}
		b(x) \left(\frac{|v_n|}{h_n} \right)^{\frac{q(x)}{q(x)-1}} & \leq C_0 b(x) \left(\frac{\log^{s(x)}(1+ |\nabla u_n|) |\nabla u_n|^{q(x)-1} + \log^{s(x)}(1+ |\nabla u|) |\nabla u|^{q(x)-1}}{\log^{\frac{s(x)}{q(x)}} (1 + \max\{ |\nabla u|, |\nabla u_n| \})}\right)^{\frac{q(x)}{q(x)-1}} \\
		& \leq C b(x) \left(  |\nabla u_n|^{q(x)} \log^{s(x)} (1 + |\nabla u_n|) +  |\nabla u|^{q(x)}  \log^{s(x)} (1 + |\nabla u|)\right).
	\end{align*}
	As $\nabla u_n \to \nabla u$ in measure and $u_n \to u$ in $W_0^{1, \mathcal{S}}(\Omega)$, we know that $ b(x)|\nabla u_n|^{q(x)} \log^{s(x)} (1 + \abs{ \nabla u_n })$ is uniformly integrable, hence we also know that our sequence is uniformly integrable and this finishes the proof.
\end{proof}

We suppose the following assumptions:
\begin{enumerate}[label=\textnormal{(H$_\star$)},ref=\textnormal{H$_\star$}]
	\item \label{main:assump-1-2}
		$p_\star, q_\star \in C(\overline{\Omega})$, $s_\star \in L^\infty(\Omega)$, $1 < p_\star(x)$, $q_\star(x)< N$, $p_{\star}(x) < p^\ast(x)$ and $q_{\star}(x) < q^\ast(x)$ for all $x \in \overline{\Omega}$, and $ s_{\star}(x) < s(x)$, $q_\star(x)+ s_\star(x) \geq r> 1$ for a.a.\,$x \in \Omega$.
\end{enumerate}

\begin{proposition}\label{well-defined-2}
	Let \eqref{main:assump}, \eqref{main:assump-1} and \eqref{main:assump-1-2} be satisfied. Then, the functionals $\mathcal{E}_2$ and $\mathcal{E}_3$ are well-defined and belong to class $C^1$ with $\mathcal{E}_2'(u) = \mathcal{J}_2(u)$ and $\mathcal{E}_3'(u) = \mathcal{J}_3(u)$.
\end{proposition}

\begin{proof}
	By Young's inequality with exponents $\frac{p^\ast(x)}{p_\star(x)}$ and $\frac{q^\ast(x)}{q_\star(x)}$ and its conjugates, we get
	\begin{align*}
		\left( \left(a(x)\right)^\frac{1}{p(x)} t\right)^{p_\star(x)}  \leq  \left( \left(a(x)\right)^\frac{1}{p(x)} t\right)^{p_\ast(x)}   + 1
	\end{align*}
	and by \eqref{embd:est-2},
	\begin{align*}
		& \left(\left(b(x) \log^{s_\star(x)}(1+t)\right)^\frac{1}{q(x)} t  \right)^{q_\star(x)}   \\
		& \leq  \left(\left(b(x) \log^{s_\star(x)}(1+t)\right)^\frac{1}{q(x)} t  \right)^{q^\ast(x)}  \chi_{\{s_\star (x) > 0\}} + 1 + \left(\left(b(x) \log^{s_\star(x)}(1+t)\right)^\frac{1}{q(x)} t  \right)^{q_\star(x)} \chi_{\{s_\star (x) \leq 0\}} \\
		&  \leq  \left(\left(b(x) \log^{s(x)}(1+t)\right)^\frac{1}{q(x)} t  \right)^{q^\ast(x)} + C \left(\left(b(x) \log^{s(x)}(1+t)\right)^\frac{1}{q(x)} t  \right)^{q^\ast(x)} + h(x).
	\end{align*}
	This gives
	\begin{equation}\label{cond:AbstractEmbedding}
		\mathcal{S}_\star(x,t) \leq C \mathcal{S}^\ast(x,t) + h(x)  \quad \text{for all }  t \geq 0,  \text{ for a.a.\,}  x \in \Omega  \ \text{and for some }  C>0
	\end{equation}
	with
	\begin{align*}
		\mathcal{S}_\star(x,t):= \left( \left(a(x)\right)^\frac{1}{p(x)} t\right)^{p_\star(x)}  + \left(\left(b(x) \log^{s_\star(x)}(1+t)\right)^\frac{1}{q(x)} t  \right)^{q_\star(x)}.
	\end{align*}
	Therefore, by Propositions \ref{prop:embedding},   \ref{sobolevembed:1} and \ref{imp:embedding}, we have $W_0^{1, \mathcal{S}}(\Omega)  \hookrightarrow L^{\mathcal{S}^\ast}(\Omega) \hookrightarrow L^{\mathcal{S}_\star}(\Omega)$. Then, for any $u \in W_0^{1, \mathcal{S}}(\Omega)$, it holds
	\begin{align*}
		0 \leq \frac{\varrho_{\mathcal{S}^\ast}( u)}{\mathfrak{n}^{+,\ast}_0} \leq \mathcal{E}_2(u) \leq  \frac{\varrho_{\mathcal{S}^\ast}(u)}{\mathfrak{m}^{-, \ast}_0} < \infty \quad \text{with} \quad \mathfrak{n}^{+,\ast}_0:= \max\limits_{x \in \overline{\Omega}} \mathfrak{n}^\ast_0(x)  \text{ and }  \mathfrak{m}^{-, \ast}_0:= \min\limits_{x \in \overline{\Omega}}\mathfrak{m}_0^\ast(x)
	\end{align*}
	and
	\begin{align*}
		0 \leq \frac{\varrho_{\mathcal{S}_\star}(u)}{\mathfrak{n}^{+}_\star} \leq \mathcal{E}_3(u) \leq  \frac{\varrho_{\mathcal{S}_\star}(u)}{\mathfrak{m}^{-}_\star} < \infty
	\end{align*}
	where
	\begin{align*}
		\mathfrak{n}^{+}_\star:= \max\limits_{x \in \overline{\Omega}} \max\{p_\star(x), q_\star(x) \} \quad \text{and} \quad \mathfrak{m}^{-}_\star:= \min\limits_{x \in \overline{\Omega}} \min\{p_\star(x), q_\star(x) \}.
	\end{align*}
	Hence, the functionals $\mathcal{E}_2$ and $\mathcal{E}_3$ are well-defined. The remaining proof can be done by adopting the same arguments as in Proposition \ref{well-defined-1}.
\end{proof}

Combining Proposition \ref{well-defined-1} and Proposition \ref{well-defined-2}, we assert that the energy functional $\mathcal{E}$ is well-defined and belong to class $C^1$ with $\mathcal{E}_{\Lambda, \lambda}'(u) = \mathcal{J}(u)$.

Next, we are concerned with the properties of the operator $\mathcal{J}_1$. For this purpose, we need the following two lemmas. The first one is concerned with the monotonicity of terms that are not power laws and the second lemma is concerned with a version of Young's inequality specially tailored for our line of work.

\begin{lemma}
	Let $q >1$ and $s \in \mathbb{R}$ such that $q +s \geq 2-\delta$ for some $\delta \in (1,2)$. Then, for any $\xi, \eta \in \R^N$, we have the following inequalities: \\
    If $ s \geq 0$,
	\begin{equation}\label{est:s-posi-qgeq2}
		\left(|\xi|^{q-2} \xi \log^{s}(1+|\xi|) - |\eta|^{q-2} \eta \log^{s}(1+|\eta|) \right) \cdot \left( \xi - \eta \right)  \geq  C_q |\xi - \eta|^{q} \log^s(1+ |m|) \quad \text{if }  q \geq 2.
	\end{equation}
	and
	\begin{equation}\label{est:s-posi-qleq2}
		\begin{aligned}
			&\left( |\xi| + |\eta| \right)^{2 - q} \left( |\xi|^{q-2} \xi \log^s(1+|\xi|)- |\eta|^{q-2} \eta \log^s(1+|\eta|)\right)
			\cdot \left( \xi - \eta \right)  \\
			& \geq C_q |\xi -\eta|^2 \log^s(1+ |m|) \quad \text{if }  1<q<2
		\end{aligned}
	\end{equation}
	where $m = \min \{ \abs{\xi}, \abs{\eta} \}$ and
	\begin{align*}
		C_q =
		\begin{cases}
			\min \{ 2^{2-q}, 2^{-1} \} & \text{if } q \geq 2,  \\
			q-1 & \text{if } 1 < q < 2.
		\end{cases}
	\end{align*}
	If $s<0$ and $q >1$
	\begin{equation}\label{est:s-negi}
		\begin{aligned}
			&\left(|\xi + |\eta| \right)^{\delta}  \left(|\xi|^{q-2} \xi \log^{s}(1+|\xi|) - |\eta|^{q-2} \eta \log^{s}(1+|\eta|) \right) \cdot \left( \xi - \eta \right)  \\
			& = \left(|\xi + |\eta| \right)^{\delta} \left(|\xi|^{-\delta} \xi |\xi|^{q-2 + \delta}  \log^{s}(1+|\xi|) -  |\eta|^{-\delta} \eta |\eta|^{q-2 + \delta} \log^{s}(1+|\eta|) \right) \cdot \left( \xi - \eta \right) \\
			& \geq C_\delta \abs{\xi - \eta}^2 |m|^{q-2 + \delta}  \log^{s}(1+|m|)
		\end{aligned}
	\end{equation}
	for any $1< \delta  <2$ where
	\begin{align*}
		C_\delta =
		\begin{cases}
			\min \{ 2^{-\delta}, 2^{-1} \} & \text{if } \delta \geq 0,   \\
			\delta + 1  & \text{if } -1 < \delta < 0.
		\end{cases}
	\end{align*}
\end{lemma}

\begin{proof}
	The proof follows from Arora--Crespo-Blanco--Winkert \cite[Lemma 4.3]{Arora-Crespo-Blanco-Winkert-2023} by taking $h(t) = \log^s(1+ t)$ if $s \geq 0$. If $s<0$, consider $h(t) = t^{q-2 + \delta} \log^s(1+ t)$ for $\delta \in (1,2)$ in Arora--Crespo-Blanco--Winkert \cite[Lemma 4.3]{Arora-Crespo-Blanco-Winkert-2023}.
\end{proof}

\begin{lemma}\label{Le:YoungIneqLog}
	Let $w,t \geq 0$, $q >1$ and $s \in \mathbb{R}$ such that $q +s >1$. Then
	\begin{align*}
		&w t^{q-1} \log^{s-1}(1+t) \left[\log(1 + t) + \frac{s t}{q(1 + t)} \right] \\
		& \leq \frac{w^q}{q} \log^{s}(1 + w) + t^q \log^{s-1}(1+t) \left[ \frac{q- 1}{q} \log (1 + t ) + \frac{s t }{q (q + t)} \right].
	\end{align*}
\end{lemma}

\begin{proof}
	The proof follows by repeating the same arguments as in Arora--Crespo-Blanco--Winkert \cite[Lemma 4.4]{Arora-Crespo-Blanco-Winkert-2023} by taking the function $h\colon  \mathbb{R}^+ \to \mathbb{R}$ defined as
	\begin{align*}
		h(t) = t^{q-1} \log^{s-1}(1+t) \left[\log(1+t) + \frac{st}{q(1+t)}\right].
	\end{align*}
	Note that the condition $q + s >1$ implies that the above function $h$ is positive, continuous, strictly increasing, and vanishes at zero. In particular, the arguments in \textbf{Case 1} and \textbf{2} in the proof of Lemma \ref{Le:Prop-S} imply that the function $h$ is strictly increasing.
\end{proof}

Now we can state the main properties of the operator $\mathcal{J}_1$.

\begin{theorem}\label{Th:PropertiesOperator}
	Let \eqref{main:assump} and \eqref{main:assump-1} be satisfied and $\{u_n\}_{n \in \mathbb{N}} \subseteq W_0^{1, \mathcal{S}}(\Omega)$ be a sequence such that
	\begin{align}\label{Eq:SequenceSplus}
		u_n \rightharpoonup u \quad\text{in }  W_0^{1, \mathcal{S}}(\Omega)
		\quad\text{and}\quad
		\limsup_{n \to \infty}{\left\langle \mathcal{J}_1(u_n), u_n - u \right\rangle } \leq 0.
	\end{align}
	Then, the following hold:
	\begin{enumerate}
		\item[\textnormal{(i)}]
			$\nabla u_n \to \nabla u$ a.e.\,in $\Omega$.
		\item[\textnormal{(ii)}]
			$u_n \to u$ in $W_0^{1, \mathcal{S}}(\Omega)$.
	\end{enumerate}
	In particular, the operator $\mathcal{J}_1$  is of type \textnormal{(S$_+$)}.
\end{theorem}

\begin{proof}
	By the strict monotonicity of $\mathcal{J}_1$ and the weak convergence of $u_n$, we obtain
	\begin{align*}
		0
		\leq \liminf_{n \to \infty} \left\langle \mathcal{J}_1(u_n) - \mathcal{J}_1(u) , u_n - u \right\rangle
		& \leq \limsup_{n \to \infty}\left\langle \mathcal{J}_1(u_n) - \mathcal{J}_1(u) , u_n - u \right\rangle \\
		& = \limsup_{n \to \infty}\left\langle \mathcal{J}_1(u_n) , u_n - u \right\rangle
		\leq 0,
	\end{align*}
	which means
	\begin{equation}\label{combined:conv:est}
		\lim_{n \to \infty} \left\langle \mathcal{J}_1(u_n) - \mathcal{J}_1(u) , u_n - u \right\rangle
		= 0.
	\end{equation}
	{\bf Claim: } $\nabla u_n \to \nabla u$ in measure.

	In particular, as the previous expression can be decomposed in the sum of nonnegative terms, it follows
	\begin{equation}\label{Eq:CasePgeq2}
		\lim_{n \to \infty} \int_{ \{p \geq 2 \}}
		a(x)\left(|\nabla u_n|^{p(x)-2} \nabla u_n - |\nabla u|^{p(x)-2} \nabla u \right) \cdot \left( \nabla u_n - \nabla u \right) \,\mathrm{d}x = 0,
	\end{equation}
	\begin{equation}\label{Eq:CaseP<2}
		\lim_{n \to \infty} \int_{ \{p < 2 \}} a(x) \left(|\nabla u_n|^{p(x)-2} \nabla u_n - |\nabla u|^{p(x)-2} \nabla u \right) \cdot \left( \nabla u_n - \nabla u \right) \,\mathrm{d}x = 0.
	\end{equation}
	\begin{equation}\label{Eq:Casebneq0}
		\begin{aligned}
			 & \lim_{n \to \infty} \int_{\Omega} b(x) \Big(|\nabla u_n|^{q(x)-2} \nabla u_n \log^{s(x)}(1+ |\nabla u_n|)  \\
			 & \qquad\qquad\qquad\qquad - |\nabla u|^{q(x)-2} \nabla u \log^{s(x)}(1+ |\nabla u|) \Big) \cdot \left( \nabla u_n - \nabla u \right) \,\mathrm{d}x = 0.
		\end{aligned}
	\end{equation}
	By repeating the same arguments as in the proof of Arora--Crespo-Blanco--Winkert \cite[Theorem 4.4]{Arora-Crespo-Blanco-Winkert-2023} for the integrals \eqref{Eq:CasePgeq2} and \eqref{Eq:CaseP<2}, we obtain
	\begin{align*}
		\nabla u_n 1_{ \{a >0\}} \to \nabla u 1_{ \{a>0\}} \quad\text{in measure}.
	\end{align*}
Now, in order to prove our claim in $1_{ \{b >0\}}$,  we partition the domain of the integral in \eqref{Eq:Casebneq0} as
	\begin{align*}
		\int_{\Omega}= \int_{E^{(1)}_n} \cdots + \int_{E^{(2)}_n} \cdots + \int_{E^{(2)}_n} \cdots + \int_{E^{(4)}_n} \cdots = J_1 + J_2 + J_3 + J_4
	\end{align*}
	where
	\begin{align*}
		E^{(1)}_n = \{ \nabla u_n \neq 0,  \nabla u \neq 0\}, & \quad E^{(2)}_n = \{ \nabla u_n \neq 0,  \nabla u = 0\}, \\
		E^{(3)}_n = \{ \nabla u_n = 0,  \nabla u \neq 0\},    & \quad E^{(4)}_n = \{ \nabla u_n = 0,  \nabla u = 0\}.
	\end{align*}
	It is easy to see that $J_i \to 0$ for $i=2,3,4$, which also implies
	\begin{align*}
		1_{\cup_{i=2}^4 E^{(i)}_n} \nabla u_n \to 0  \quad\text{a.e.\,in }  \Omega.
	\end{align*}
	On the other hand, we partition the domain of the integral $J_1$ as
	\begin{align*}
		\int_{E_n^{(1)}} \cdots = \int_{\{s \geq 0, q\geq 2, b>0\} \cap E_n^{(1)}} \cdots + \int_{\{s \geq 0, q<2, b>0 \} \cap E_n^{(1)}} \cdots + \int_{\{s<0, b>0\} \cap E_n^{(1)}} \cdots = I_1 + I_2 + I_3
	\end{align*}
	and using \eqref{est:s-posi-qgeq2}-\eqref{est:s-negi} and \eqref{combined:conv:est}, we get
	\begin{equation}\label{conv:I's}
		I_i \to 0  \quad\text{as }  n \to \infty  \text{ for }  i=1,2,3.
	\end{equation}
	\newline
	\textbf{Step 1: Convergence in $\{s \geq 0, q\geq 2, b>0\} \cap E_n^{(1)}$:}

	From \eqref{est:s-posi-qgeq2}, for any $\varepsilon > 0$ we know that
	\begin{align*}
		 & \left\lbrace 1_{ \{q \geq 2, b>0, s \geq 0 \}} (q^- - 1) |\nabla u_n - \nabla u|^q 1_{E_n^{(1)}} \log^{s(x)}(1+ t(x)) \geq \varepsilon\right\rbrace  \\
		 & \subseteq \bigg\lbrace 1_{ \{q >2, b>0, s\geq 0 \}} \left( |\nabla u_n|^{q(x)-2} \nabla u_n \log^{s(x)}(1+ |\nabla u_n|)- \abs{\nabla u}^{q(x)-2} \nabla u \log^{s(x)}(1+ |\nabla u|) \right) \\
		 & \qquad \qquad \qquad \qquad \qquad \qquad \cdot \left( \nabla u_n - \nabla u \right) \geq \varepsilon\bigg\rbrace,
	\end{align*}
	where $t(x)= \min\{|\nabla u_{n}(x)|, |\nabla u(x)|\}$. From \eqref{conv:I's} and the previous expression we obtain that
	\begin{align*}
		1_{ \{q \geq 2, b>0, s \geq 0 \}} (q^- - 1) |\nabla u_n - \nabla u|^q 1_{E_n^{(1)}} \log^{s(x)}(1+ t(x)) \rightarrow 0 \quad \text{in measure.}
	\end{align*}
	This implies (up to a subsequence)
	\begin{align*}
		\text{either}\quad  1_{ \{q \geq 2, b>0, s \geq 0 \} \cap E_n^{(1)}} \nabla u_n \to 1_{ \{q \geq 2, b>0, s \geq 0 \}} \nabla u \quad \text{or} \quad 1_{ \{q > 2, b>0, s \geq 0\} \cap E_n^{(1)}} \nabla u_n \to 0  \quad\text{a.e.\,in }  \Omega.
	\end{align*}
	We claim the latter part does not hold. Suppose it holds. Then,
	\begin{align*}
		1_{ \{q \geq 2, b>0, s \geq 0 \}} & \left( |\nabla u_n|^{q(x)-2} \nabla u_n \log^{s(x)}(1+ |\nabla u_n|)- \abs{\nabla u}^{q(x)-2} \nabla u \log^{s(x)}(1+ |\nabla u|) \right) \cdot \left( \nabla u_n - \nabla u \right) \\
		& \to 1_{ \{q >2, b>0, s\geq 0 \}} \abs{\nabla u}^{q(x)} \log^{s(x)}(1+ |\nabla u|) \quad \text{a.e.\,in }  \Omega.
	\end{align*}
	Since $u_n$ is a bounded sequence in $W_0^{1, \mathcal{S}}(\Omega)$, by Vitali's convergence theorem
	\begin{align*}
		I_1 \to \int_{1_{ \{q \geq 2, b>0, s \geq 0 \}}} \abs{\nabla u}^{q(x)} \log^{s(x)}(1+ |\nabla u|) \,\mathrm{d}x,
	\end{align*}
	which is a contradiction to \eqref{conv:I's}. Hence, by the subsequence principle, the following is true for the whole sequence $u_n$,
	\begin{equation}\label{est:aeconv-1}
		1_{ \{q \geq 2, b>0, s \geq 0 \} \cap E_n^{(1)}} \nabla u_n \to 1_{ \{q \geq 2, b>0, s \geq 0 \}} \nabla u \quad \text{a.e.\,in }  \Omega.
	\end{equation}
	\textbf{Step 2: Convergence in $\{s \geq 0, q < 2, b>0\} \cap E_n^{(1)}$:}

	From \eqref{est:s-posi-qleq2}, for any $\varepsilon > 0$ we know that
	\begin{align*}
		& \left\lbrace 1_{ \{q < 2, b>0, s \geq 0 \}} (q^- - 1) |\nabla u_n - \nabla u|^2 (|\nabla u_n| + |\nabla u|)^{q(x)-2} 1_{E_n^{(1)}} \log^{s(x)}(1+ t(x)) \geq \varepsilon\right\rbrace  \\
		& \subseteq \bigg\lbrace 1_{ \{q <2, b>0, s\geq 0 \}} \left( |\nabla u_n|^{q(x)-2} \nabla u_n \log^{s(x)}(1+ |\nabla u_n|)- \abs{\nabla u}^{q(x)-2} \nabla u \log^{s(x)}(1+ |\nabla u|) \right) \\
		& \qquad \qquad \qquad \qquad \qquad \qquad \cdot \left( \nabla u_n - \nabla u \right) \geq \varepsilon\bigg\rbrace.
	\end{align*}
	From \eqref{conv:I's} and the previous expression we obtain that
	\begin{align*}
		1_{ \{q < 2, b>0, s \geq 0 \} \cap E_n^{(1)}} (q^- - 1) |\nabla u_n - \nabla u|^2 (|\nabla u_n| + |\nabla u|)^{q(x)-2} \log^{s(x)}(1+ t(x)) \rightarrow 0 \quad \text{in measure.}
	\end{align*}
	This implies (up to a subsequence), either $1_{ \{q < 2, b>0, s \geq 0\} \cap E_n^{(1)}} \nabla u_n \not \to 0$ a.e.\,in $\Omega$ and
	\begin{align*}
		1_{ \{q < 2, b>0, s \geq 0 \} \cap E_n^{(1)}} (q^- - 1) |\nabla u_n - \nabla u|^2 (|\nabla u_n| + |\nabla u|)^{q(x)-2} \log^{s(x)}(1+ t(x)) \to 0   \quad \text{a.e.\,in }  \Omega
	\end{align*}
	or
	\begin{align*}
		1_{ \{q < 2, b>0, s \geq 0\}} \nabla u_n \to 0  \quad\text{a.e.\,in }  \Omega.
	\end{align*}
	By repeating the same argument as above, we can show that the latter part gives a contradiction to \eqref{conv:I's}. Hence, $1_{ \{q < 2, b>0, s \geq 0\} \cap E_n^{(1)}} \nabla u_n \not \to 0$ a.e.\,in $\Omega$ and
	\begin{align*}
		1_{ \{q < 2, b>0, s \geq 0 \} \cap E_n^{(1)}} (q^- - 1) |\nabla u_n - \nabla u|^2 (|\nabla u_n| + |\nabla u|)^{q(x)-2} \log^{s(x)}(1+ t(x)) \to 0   \quad\text{a.e.\,in }  \Omega.
	\end{align*}
	Then for a.e.\,$x \in \Omega$ there exists $M(x) > 0$ such that for all $k \in \N$
	\begin{align*}
		M(x) & \geq 1_{ \{q < 2, b>0, s \geq 0 \} \cap E_n^{(1)}} \abs{\nabla u_{n_k} - \nabla u}^2 1_{E_{n_k}} \left( \abs{\nabla u_{n_k}} + \abs{\nabla u} \right)^{q(x) - 2} \log^{s(x)}(1+ t(x)) \\
		& \geq 1_{ \{q < 2, b>0, s \geq 0 \} \cap E_n^{(1)}} \abs{ \abs{\nabla u_{n_k}} - \abs{\nabla u} }^2 1_{E_{n_k}} \left( \abs{\nabla u_{n_k}} + \abs{\nabla u} \right)^{q(x) - 2} \log^{s(x)}(1+ t(x)).
	\end{align*}
	Note that, given any $c > 0$ and $0 < Q < 1$ and $S \in \mathbb{S}$, the function $h(t) = \abs{t - c}^2 (t + c)^{Q-2} \log^{S}(1 + \min\{t,c\})$ satisfies $\lim_{t \to +\infty} h(t) = +\infty$. Therefore, there exists $m(x) > 0$ such that $\abs{ \nabla u_{n_k} } \leq m(x)$ for a.a.\,$x \in \Omega$ and for all $k \in \N$. As a consequence
	\begin{align*}
		& 1_{ \{q < 2, b>0, s \geq 0 \} \cap E_n^{(1)}} \abs{\nabla u_{n_k} - \nabla u}^2 1_{E_{n_k}} \left( \abs{\nabla u_{n_k}} + \abs{\nabla u} \right)^{q(x) - 2} \log^{s(x)}(1 + t(x)) \\
		& \geq 1_{ \{q < 2, b>0, s \geq 0 \} \cap E_n^{(1)}} \abs{\nabla u_{n_k} - \nabla u}^2 1_{E_{n_k}} \left( m(x) + \abs{\nabla u} \right)^{q(x) - 2} \log^{s(x)}(1 + t(x))
	\end{align*}
	and the convergence a.e.\,to zero of the left-hand side and by the subsequence principle, this yields
	\begin{equation}\label{est:aeconv-2}
		1_{ \{q < 2, b>0, s \geq 0 \} \cap E_n^{(1)}} \nabla u_{n} \to 1_{ \{q < 2, b>0, s \geq 0 \}}  \nabla u  \quad\text{a.e.\,in }  \Omega
	\end{equation}
	since $1_{ \{q < 2, b>0, s \geq 0\} \cap E_n^{(1)}} \nabla u_n \not \to 0$ a.e.\,in $\Omega$.\\
	\textbf{Step 3: Convergence in $\{s < 0, b>0\} \cap E_n^{(1)}$:}

	Now by replacing $2-q$ by $\delta$ in \textbf{Step 2} and using \eqref{est:s-negi} in place of \eqref{est:s-posi-qleq2} with $2-r< \delta <2 $, we obtain
	\begin{equation}\label{est:aeconv-3}
		1_{ \{b>0, s <0 \} \cap E_n^{(1)}} \nabla u_{n} \to 1_{ \{b>0, s < 0 \}}  \nabla u \quad \text{a.e.\,in }  \Omega
	\end{equation}
	Finally, combining \eqref{est:aeconv-1}, \eqref{est:aeconv-2} and \eqref{est:aeconv-3}, we obtain the required claim \textnormal{(i)}. From Young's inequality and Lemma \ref{Le:YoungIneqLog} it follows that
	\begin{align*}
		& a(x) |\nabla u_n|^{p(x)-2}\nabla u_n \cdot \nabla (u_n - u ) \,\mathrm{d}x \\
		& \quad+ b(x) \left[ \log (1 + |\nabla u_n| ) + \frac{s(x) |\nabla u_n|}{q(x) (1 + |\nabla u_n|)} \right]
		|\nabla u_n|^{q(x)-2}\nabla u_n  \log^{s(x)-1} (1 + |\nabla u_n|) \cdot\nabla (u_n - u ) \,\mathrm{d}x \\
		& = a(x) |\nabla u_n|^{p(x)} \,\mathrm{d}x-  a(x) |\nabla u_n|^{p(x)-2}\nabla u_n \cdot \nabla u \,\mathrm{d}x \\
		& \quad + b(x) \left[ \log (1 + |\nabla u_n| ) + \frac{s(x)|\nabla u_n|}{q(x) (1 + |\nabla u_n|)} \right] |\nabla u_n|^{q(x)} \log^{s(x)-1} (1 + |\nabla u_n|) \,\mathrm{d}x   \\
		& \quad - b(x) \left[ \log (1 + |\nabla u_n| ) + \frac{s(x) |\nabla u_n|}{q(x) (1 + |\nabla u_n|)} \right] |\nabla u_n|^{q(x)-2}\nabla u_n \log^{s(x)-1} (1 + |\nabla u_n|) \cdot \nabla u \,\mathrm{d}x \\
		 & \geq a(x) |\nabla u_n|^{p(x)} \,\mathrm{d}x- a(x) |\nabla u_n|^{p(x)-1} \abs{\nabla u} \,\mathrm{d}x  \\
		& \quad + b(x) \left[ \log (1 + |\nabla u_n| ) + \frac{s(x) |\nabla u_n|}{q(x) (1 + |\nabla u_n|)} \right] |\nabla u_n|^{q(x)} \log^{s(x)-1} (1 + |\nabla u_n|) \,\mathrm{d}x  \\
		& \quad - b(x) \left[ \log (1 + |\nabla u_n| ) + \frac{s(x)|\nabla u_n|}{q(x) (1 + |\nabla u_n|)} \right] |\nabla u_n|^{q(x)- 1} \log^{s(x)-1} (1 + |\nabla u_n|) \abs{\nabla u} \,\mathrm{d}x   \\
		& \geq a(x) |\nabla u_n|^{p(x)} \,\mathrm{d}x - a(x) \left( \frac{p(x)-1}{p(x)}|\nabla u_n|^{p(x)}+\frac{1}{p(x)} \abs{\nabla u}^{p(x)} \right) \,\mathrm{d}x  \\
		& \quad + b(x) \left[ \log (1 + |\nabla u_n| ) + \frac{s(x)|\nabla u_n|}{q(x) (1 + |\nabla u_n|)} \right] |\nabla u_n|^{q(x)} \log^{s(x)-1} (1 + |\nabla u_n|) \,\mathrm{d}x   \\
		& \quad - b(x)\left( \left[ \frac{q(x)-1}{q(x)} \log (1 + |\nabla u_n| ) + \frac{s(x)|\nabla u_n|}{q(x) (1 + |\nabla u_n|)} \right] |\nabla u_n|^{q(x)} \log^{s(x)-1} (1 + |\nabla u_n|)\right.  \\
		& \quad \left. + \frac{1}{q(x)}\abs{\nabla u}^{q(x)} \log^{s(x)} (1 + \abs{ \nabla u }) \right) \,\mathrm{d}x \\
		& = \frac{a(x)}{p(x)}|\nabla u_n|^{p(x)} \,\mathrm{d}x- \frac{a(x)}{p(x)} \abs{\nabla u}^{p(x)} \,\mathrm{d}x\\
		& \quad + \frac{b(x)}{q(x)}|\nabla u_n|^{q(x)} \log^{s(x)} (1 + \abs{ \nabla u_n }) \,\mathrm{d}x- \frac{b(x)}{q(x)}\abs{\nabla u}^{q(x)} \log^{s(x)} (1 + \abs{ \nabla u }) \,\mathrm{d}x .
	\end{align*}
	As a consequence, by \eqref{Eq:SequenceSplus} and Fatou's Lemma, we obtain
	\begin{align*}
		& \lim_{n \to \infty} \int_\Omega \left(\frac{a(x)}{p(x)} |\nabla u_n|^{p(x)} + \frac{b(x)}{q(x)} |\nabla u_n|^{q(x)} \log^{s(x)} (1 + \abs{ \nabla u_n }) \right) \,\mathrm{d}x \\
		& = \int_\Omega \left(\frac{a(x)}{p(x)} \abs{\nabla u}^{p(x)} + \frac{b(x)}{q(x)} \abs{\nabla u}^{q(x)} \log^{s(x)} (1 + \abs{ \nabla u }) \right) \,\mathrm{d}x.
	\end{align*}
	By the previous Claim, passing to a.e.\,convergence along a subsequence and using the subsequence principle, we can prove that the integrand of the left-hand side converges in measure to the integrand of the right-hand side. Finally by applying the Br\'{e}zis-Lieb Lemma (Lemma \ref{BL-lemma}) and by Proposition \ref{pro:norm-mod:relation} \textnormal{(iv)}, we obtain $u_n \to u$ in $W_0^{1, \mathcal{S}}(\Omega)$.
\end{proof}

\section{Application of concentration compactness principle}\label{Section-6}

In this section, we show multiplicity results for the $(\Lambda, \lambda)$-parametrized problem \eqref{main:prob} with critical and sublinear/superlinear growth.

\subsection{Superlinear growth}

In this subsection, we assume that $\lambda =1$ and $\Lambda >0$. For the sake of simplicity, we write $\mathcal{E}_{\Lambda, \lambda}$ as $ \mathcal{E}_\Lambda$. We suppose the following assumptions:
\begin{enumerate}[label=\textnormal{(H$_\star^{\text{sup}}$)},ref=\textnormal{H$_\star^{\text{sup}}$}]
	\item\label{main:assump-2}
		$p_\star, q_\star \in C(\overline{\Omega})$, $s_\star \in L^\infty(\Omega)$, $1 < p_\star(x), q_\star(x)< N$, $p(x) \leq p_{\star}(x)$ and $q(x) \leq q_{\star}(x)$  for all $x \in \overline{\Omega}$, and  $q_\star(x)+ s_\star(x) > r> 1$, $ s_\star(x) \leq s(x)$ for a.a.\,$x \in \Omega$.
\end{enumerate}

\begin{lemma}\label{PS:bounded}
	Let \eqref{main:assump}, \eqref{main:assump-1}, \eqref{main:assump-1-2} and \eqref{main:assump-2} be satisfied and suppose that
	\begin{equation}\label{sandwich:cond-1}
		\max\{p^+, ( q + \lceil s \rceil)^+\} < \min\left\{p_\star^-, \left(q_\star + \lfloor s_\star \rfloor \frac{q_\star}{q}\right)^- \right\}.
	\end{equation}
	Then, every Palais-Smale sequence $\{u_n\}_{n \in \mathbb{N}} \subset W_0^{1, \mathcal{S}}(\Omega)$ is bounded.
\end{lemma}

\begin{proof}
	By the definition of the Palais-Smale sequence $\{u_n\}_{n \in \mathbb{N}}$, we have
	\begin{equation}\label{PS-cond}
		\mathcal{E}_\Lambda(u_n) \to c \quad \text{ and } \quad \langle \mathcal{E}_\Lambda'(u_n), \phi \rangle \to 0 \quad \text{for every }  \phi \in W_0^{1, \mathcal{S}}(\Omega),  \text{for some }  c \in \mathbb{R}.
	\end{equation}
	We set
	\begin{align*}
		\sigma:= \frac{\max\{p^+, ( q + \lceil s \rceil)^+\} + \min\left\{p_\star^-, \left(q_\star + \lfloor s_\star \rfloor \frac{q_\star}{q}\right)^- \right\}}{2}.
	\end{align*}
	Now, by taking $\frac{u_n}{\sigma}$ as a test function, and by using the Palais-Smale condition \eqref{PS-cond} for $n \geq n_0$, we obtain,
	\begin{equation}\label{PS-bdd-1}
		c+1 \geq \mathcal{E}_\Lambda(u_n) - \langle \mathcal{J}(u_n), \frac{u_n}{\sigma} \rangle
		= \sum_{i=1}^2 \mathcal{E}_i(u_n) - \langle \mathcal{J}_i(u_n), \frac{u_n}{\sigma} \rangle
		= \sum_{i=1}^2 \mathcal{I}_i(u_n),
	\end{equation}
	where
	\begin{align*}
		\mathcal{I}_1(u_n)
		&:= \int_{\Omega} \left(\frac{1}{p(x)} -\frac{1}{\sigma}\right) a(x) |\nabla u_n|^{p(x)} \,\mathrm{d}x \\
		& \qquad + \Lambda \int_{\Omega} \left( \frac{1}{\sigma} - \frac{1}{p^\ast(x)}\right) \left(a(x)^\frac{1}{p(x)} |u_n|\right)^{p^\ast(x)} \,\mathrm{d}x\\
		&\qquad+ \int_{\Omega} \left( \frac{1}{\sigma} - \frac{1}{p_\star(x)}\right) \left(a(x)^\frac{1}{p(x)} |u_n|\right)^{p_\star(x)} \,\mathrm{d}x  ,
	\end{align*}
	and
	\begin{align*}
		\mathcal{I}_2(u_n)
		& :=\int_{\Omega} \left(\frac{1}{q(x)}- \frac{1}{\sigma} \right) b(x) |\nabla u_n|^{q(x)} \log^{s(x)}(1+ |\nabla u_n|) \,\mathrm{d}x   \\
		& \qquad- \int_{\Omega} \frac{b(x) s(x)}{q(x) \sigma} |\nabla u_n|^{q(x)+1} \frac{\log^{s(x)-1 }(1+ |\nabla u_n|)}{1+ |\nabla u_n|}  \,\mathrm{d}x \\
		& \qquad+ \Lambda \int_{\Omega} \left(\frac{1}{\sigma} - \frac{1}{q^\ast(x)} +  \frac{|u_n|}{\log(1+ |u_n|)(1+ |u_n|)} \frac{\lfloor s(x) \rfloor q^\ast(x)}{q(x) \sigma}\right)\\
		&\qquad\qquad\times\left(b(x) \log^{s(x)}(1+ |u_n|) \right)^\frac{q^\ast(x)}{q(x)} |u_n|^{q^\ast(x)} \,\mathrm{d}x  \\
		& \qquad+ \int_{\Omega} \left( \frac{1}{\sigma} - \frac{1}{q_\star(x)}  + \frac{|u_n|}{\log(1+ |u_n|)(1+ |u_n|)} \frac{\lfloor s_\star(x) \rfloor q_\star(x)}{q(x) \sigma}\right)\\
		&\qquad\qquad\times \left(b(x) \log^{s_\star(x)}(1+ |u_n|) \right)^\frac{q_\star(x)}{q(x)} |u_n|^{q_\star(x)} \,\mathrm{d}x.
	\end{align*}
	\textbf{Estimate for $\mathcal{I}_1(u_n):$}
	By the choice of $\sigma$, we can choose $\delta>0$ small enough such that
	\begin{align*}
		\delta \leq \frac{1}{2} \max\limits_{x \in \overline{\Omega}}  \max\left\{\frac{1}{p(x)} -\frac{1}{\sigma},  \frac{1}{\sigma} - \frac{1}{p^\ast(x)} \right\}
	\end{align*}
	and
	\begin{equation}\label{lower:est-5}
		\begin{aligned}
			\mathcal{I}_1(u_n) & \geq \delta \int_{\Omega} a(x) |\nabla u_n|^{p(x)} \,\mathrm{d}x.
		\end{aligned}
	\end{equation}
	\textbf{Estimate for $\mathcal{I}_2(u_n):$} For every $\delta >0$ there exists a constant $M(\delta, s^+) >0$ such that $\frac{t}{\log(1+t)(1+t)} < \frac{\delta}{s^+ +1}$ for $t \geq M$, therefore we have the following estimates: By splitting the domain depending upon the size of $|\nabla u_n|$, we get
	\begin{equation}\label{lower:est-7}
		\begin{aligned}
			&\bigg|\int_{\Omega} \frac{b(x) s(x)}{q(x) \sigma} |\nabla u_n|^{q(x)+1} \frac{\log^{s(x)-1 }(1+ |\nabla u_n|)}{1+ |\nabla u_n|}  \,\mathrm{d}x \bigg|  \\
			& \leq C(\nu) + \nu \int_{\Omega} b(x) |\nabla u_n|^{q(x)} \log^{s(x)}(1+ |\nabla u_n|) \,\mathrm{d}x.
		\end{aligned}
	\end{equation}
	Now, again by choice of $\sigma$, choosing $\delta>0$ small enough such that
	\begin{align*}
		\delta \leq \max\limits_{x \in \overline{\Omega}} \frac{1}{4} \max\left\{\frac{1}{q(x)} -\frac{1}{\sigma}-\nu,  \frac{1}{\sigma} - \frac{1}{q^\ast(x)} +  \frac{\lfloor s_\star(x) \rfloor q_\star(x)}{q(x) \sigma} \right\}.
	\end{align*}
	By using estimates in \eqref{lower:est-7} in light of \eqref{sandwich:cond-1} and \eqref{main:assump-2}, we obtain
	\begin{equation}\label{lower:est-10}
		\begin{aligned}
			\mathcal{I}_2(u_n)
			 & \geq \delta \int_{\Omega} b(x) |\nabla u_n|^{q(x)} \log^{s(x)}(1+ |\nabla u_n|) \,\mathrm{d}x - C(\delta).
		\end{aligned}
	\end{equation}
	Finally, inserting the estimates \eqref{lower:est-5} and \eqref{lower:est-10} in \eqref{PS-bdd-1} and using Proposition \ref{pro:norm-mod:relation}, we deduce
	\begin{equation}\label{lower:est-11}
		\begin{aligned}
			c+ c_\ast(\delta) & \geq  \delta \int_{\Omega} a(x) |\nabla u_n|^{p(x)} \,\mathrm{d} \\
			& + \delta \int_{\Omega} b(x) |\nabla u_n|^{q(x)} \log^{s(x)}(1+ |\nabla u_n|) \,\mathrm{d}x \geq \delta \| \nabla u_n\|_{\mathcal{S}}^{\min\{p^-, (q+\lfloor s \rfloor)^-\}},
		\end{aligned}
	\end{equation}
	which implies that the sequence $\{u_n\}_{n \in \mathbb{N}}$ is bounded in $W_0^{1, \mathcal{S}}(\Omega)$.
\end{proof}

We set
\begin{align*}
	c^\ast(\Lambda) := \min_{i=1,2;j=1,2} \left(\delta (C_\sharp(\Lambda))^{c_i} ({\bf C}^\ast )^{c_j}\right) - c_\ast, \quad C_\sharp(\Lambda):= \Lambda^{-1} \min\{1, r (q^+)^{-1}\} \left\|1+ |s| q^\ast q^{-1}\right\|_{\infty}^{-1},
\end{align*}
where ${\bf C}^\ast $ and $c_\ast$ are the constants obtained in Theorem \ref{concentration:compactness} and \eqref{lower:est-11}, respectively and the constants $\delta$ and $c_i, c_j$ depend on the given data.

\begin{lemma}\label{PS-cond-super}
	Let \eqref{main:assump}, \eqref{main:assump-1}, \eqref{main:assump-1-2}, \eqref{main:assump-2} and \eqref{sandwich:cond-1} be satisfied. Then the energy functional $\mathcal{E}_\Lambda$ satisfy the \textnormal{(PS)$_c$} condition for $c < c^\ast(\Lambda)$.
\end{lemma}

\begin{proof}
	Let $\{u_n\}_{n \in \mathbb{N}}$ be a Palais-Smale sequence satisfying \eqref{PS-cond}. Then, by Lemma \ref{PS:bounded} and Theorem \ref{concentration:compactness}, there exists a weakly convergent subsequence, a countable index set $I$, positive numbers $\Theta_i, \theta_i$ for each $i\in I$ and $C^\ast >0$ such that
	\begin{equation}\label{weak:conv}
		\begin{aligned}
			u_n & \to u  \quad\text{a.e.\,in }  \Omega, \\
			u_n & \rightharpoonup u  \quad\text{in }  W_0^{1, \mathcal{S}}(\Omega)   \\
			\mathcal{S}(\cdot, |\nabla u_n|)                                                                    & \rightharpoonup \theta \geq \mathcal{S}(\cdot, |\nabla u|) + \sum_{i \in I} \theta_i \delta_{x_i}  \quad\text{weakly-}\ast  \text{ in the sense of measures} \\
			\mathcal{S}^\ast(\cdot, u_n)                                                                        & \rightharpoonup \mathcal{S}^\ast(\cdot,u) + \sum_{i \in I} \Theta_i \delta_{x_i}  \quad\text{weakly-}\ast  \text{ in the sense of measures}   \\
			\min\{\Theta_j^\frac{1}{\mathtt{n}_\varepsilon^\ast(x_j)}, \Theta_j^\frac{1}{\mathtt{m}_0^\ast(x_j)}\} & \leq {\bf C}^\ast \max\{\theta_j^\frac{1}{\mathtt{m}_-(x_j)}, \theta_j^\frac{1}{\mathtt{n}_\varepsilon(x_j)}\} \quad \text{for }  j \in I.
		\end{aligned}
	\end{equation}
	We claim that $I = \emptyset$. Suppose that there exists $j \in I$. Let $\gamma >0$ and define $\phi_{\gamma, j}$ as in the proof of Theorem \ref{concentration:compactness}. Taking $\phi_{\gamma, j} u_n \in W_0^{1, \mathcal{S}}(\Omega)$ as a test function in \eqref{PS-cond}, we get
	\begin{equation}\label{compactness:est-1}
		\begin{aligned}
			& \min\left\{1, \frac{r}{q^+}\right\} \int_{\Omega} \mathcal{S}(x, |\nabla u_n|) \phi_{\gamma, j} \,\mathrm{d}x \leq  \langle \mathcal{J}_1(u_n), \phi_{\gamma, j} u \rangle_{\mathcal{S}}  \\
			& = \langle \mathcal{J}(u_n), \phi_{\gamma, j} u \rangle_{\mathcal{S}} + \Lambda \langle \mathcal{J}_2(u_n), \phi_{\gamma, j} u \rangle_{\mathcal{S}} +  \langle \mathcal{J}_3(u_n), \phi_{\gamma, j} u\rangle_{\mathcal{S}}\\ &\qquad + \int_{\Omega} a(x) |\nabla u_n|^{p(x)-2} \nabla u_n \cdot \nabla \phi_{\gamma, j}  u_n\,\mathrm{d}x \\
			& \qquad+ \int_{\Omega} b(x) |\nabla u_n|^{q(x)-2} \log^{s(x)-1 }(1+ |\nabla u_n|)\\
			&\qquad\qquad\times \left(\log(1+ |\nabla u_n|) + \frac{s(x)}{q(x)} \frac{|\nabla u_n|}{1+ |\nabla u_n|} \right) \nabla u_n \cdot \nabla \phi_{\gamma, j}  u_n \,\mathrm{d}x.
		\end{aligned}
	\end{equation}
	By applying Young's inequality and Lemma \ref{Le:YoungIneqLog} for $\delta>0$, we obtain
	\begin{align*}
		&\bigg|\int_{\Omega}  a(x) |\nabla u_n|^{p(x)-2} \nabla u_n \cdot \nabla \phi_{\gamma, j}  u_n\,\mathrm{d}x  \\
		& \quad+ \int_{\Omega} b(x) |\nabla u_n|^{q(x)-2} \log^{s(x)-1 }(1+ |\nabla u_n|) \left(\log(1+ |\nabla u_n|) + \frac{s(x)}{q(x)} \frac{|\nabla u_n|}{1+ |\nabla u_n|} \right) \nabla u_n \cdot \nabla \phi_{\gamma, j}  u_n \,\mathrm{d}x\bigg| \\
		& \leq \delta \int_{\Omega} \mathcal{S}(x, |\nabla u_n|) \,\mathrm{d}x + C(\delta) \int_{\Omega}  \mathcal{S}(x, |\nabla \phi_{\gamma, j} u_n|) \,\mathrm{d}x.
	\end{align*}
	Since $\{\phi_{\gamma, j} u_n\}_{n \in \mathbb{N}}$ is a bounded sequence in $W_0^{1, \mathcal{S}}(\Omega)$ and \eqref{PS-cond} holds,
	\begin{equation}\label{compactness:est-1-2}
		\lim_{n \to \infty} \langle \mathcal{E}_\Lambda'(u_n), \phi_{\gamma, j} u_n \rangle  =0.
	\end{equation}
	By Proposition \ref{imp:embedding}, it follows from \eqref{weak:conv} that $u_n \to u$ in $L^{\mathcal{S}}(\Omega)$. From this, we obtain
	\begin{equation}\label{compactness:est-1-3}
		\lim_{n \to \infty} \int_{\Omega}  \mathcal{S}(x, |\nabla \phi_{\gamma, j} u_n|) \,\mathrm{d}x = \int_{\Omega}  \mathcal{S}(x, |\nabla \phi_{\gamma, j} u|) \,\mathrm{d}x.
	\end{equation}
	The fact that $u_n$ is bounded in $W_0^{1, \mathcal{S}}(\Omega)$ and $W_0^{1, \mathcal{S}}(\Omega) \hookrightarrow L^{\mathcal{S}_\star}(\Omega)$ implies
	\begin{align*}
		& \left|\langle \mathcal{J}_3(u_n), \phi_{\gamma, j} u_n\rangle_{\mathcal{S}} \right|\leq \|\phi_{\gamma,j}\|_{\infty} \bigg(\int_{\Omega} \left(a(x)\right)^\frac{p_\star(x)}{p(x)} |u_n|^{p_\star(x)}  \,\mathrm{d}x \\
		& \quad + \int_{\Omega} \left(b(x) \log^{s_\star(x)}(1+ |u_n|) \right)^\frac{q_\star(x)}{q(x)} |u_n|^{q_\star(x)} \left(1+  \frac{|s_\star(x)| q_\star(x)}{q(x)} \frac{|u_n|}{\log(1+ |u_n|)(1+ |u_n|)}\right) \,\mathrm{d}x \bigg) \\
		& \leq \|\phi_{\gamma,j}\|_{\infty} \left\|1+  \frac{|s_\star(x)| q_\star(x)}{q(x)} \right\|_{\infty} \int_{\Omega} \mathcal{S}_\star(x, |u_n|)  \,\mathrm{d}x < C_1,
	\end{align*}
	where $C_1$ is independent of $n$ and by Vitali's convergence theorem, we get
	\begin{align*}
		\lim_{\gamma \to 0^+} \lim_{n \to \infty} \langle \mathcal{J}_3(u_n), \phi_{\gamma,j} u_n \rangle_{\mathcal{S}} = \lim_{\gamma \to 0^+} \langle \mathcal{J}_3(u), \phi_{\gamma,j} u \rangle_{\mathcal{S}} =0.
	\end{align*}
	Following the same arguments as above, we obtain
	\begin{align*}
		& \left|\langle \mathcal{J}_2(u_n), \phi_{\gamma, j} u_n\rangle_{\mathcal{S}} \right|\leq \int_{\Omega} \left(a(x)\right)^\frac{p^\ast(x)}{p(x)} |u_n|^{p^\ast(x)} \phi_{\gamma, j} \,\mathrm{d}x  \\
		& \quad + \int_{\Omega} \left(b(x) \log^{s(x)}(1+ |u_n|) \right)^\frac{q^\ast(x)}{q(x)} |u_n|^{q^\ast(x)} \left(1+  \frac{|s(x)| q^\ast(x)}{q(x)} \frac{|u_n|}{\log(1+ |u_n|)( 1+ |u_n|)}\right) \phi_{\gamma, j}\,\mathrm{d}x \\
		& \leq C_2 \int_{\Omega} \mathcal{S}^\ast(x, |u_n|) \phi_{\gamma,j}  \,\mathrm{d}x < C  \quad \text{where} \  C_2:= \left\|1+  \frac{|s(\cdot)| q^\ast(\cdot)}{q(\cdot)} \right\|_{\infty}
	\end{align*}
	and
	\begin{equation}\label{compactness:est-1-4}
		\lim_{n \to \infty} \left|\langle \mathcal{J}_2(u_n), \phi_{\gamma,j} u_n \rangle_{\mathcal{S}}\right| \leq  C_2 \int_{\Omega} \left(\mathcal{S}^\ast(x,u) + \sum_{i \in I} \Theta_i \delta_{x_i}\right) \phi_{\gamma,j} \,\mathrm{d}x.
	\end{equation}
	Again by using Proposition \ref{Prop:Poincare} and passing as $\gamma \to 0^+$, it follows
	\begin{equation}\label{compactness:est-2}
		\lim_{\gamma \to 0^+} \int_{\Omega} \mathcal{S}(x, |\nabla \phi_{\gamma,j} u|) \,\mathrm{d}x = \lim_{\gamma \to 0^+} \langle \mathcal{J}_3(u), \phi_{\gamma,j} u \rangle_{\mathcal{S}} = 0.
	\end{equation}
	Passing to the limits $n \to \infty$ and $\gamma \to 0^+$ in \eqref{compactness:est-1} and applying \eqref{compactness:est-1}, \eqref{compactness:est-1-2}, \eqref{compactness:est-1-3}, \eqref{compactness:est-1-4}, \eqref{compactness:est-2}, \eqref{PS-cond}, and \eqref{weak:conv}, we get
	\begin{align*}
		\min\{1, r (q^+)^{-1}\}  \theta_j \leq C_2 \Lambda \Theta_j + \delta C_0,
	\end{align*}
	where
	\begin{align*}
		C_0= \sup_{n \in \mathbb{N}} \int_{\Omega} \mathcal{S}(x, |\nabla u_n|) \,\mathrm{d}x.
	\end{align*}
	Since $\delta>0$ is chosen arbitrarily, we obtain
	\begin{align*}
		C_\sharp \theta_j \leq \Theta_j, \quad C_\sharp := \frac{\min\{1, r (q^+)^{-1}\}}{C_2 \Lambda}.
	\end{align*}
	From this and \eqref{weak:conv}, we deduce
	\begin{equation}\label{compactness:est-3}
		\min\{(C_\sharp \theta_j)^\frac{1}{\mathtt{n}_\varepsilon^\ast(x_j)}, (C_\sharp \theta_j)^\frac{1}{\mathtt{m}_0^\ast(x_j)}\}  \leq {\bf C}^\ast \max\{\theta_j^\frac{1}{\mathtt{m}_-(x_j)}, \theta_j^\frac{1}{\mathtt{n}_\varepsilon(x_j)}\} \quad \text{for }  j \in I.
	\end{equation}
	The condition \eqref{main:assump-1} implies that there exists $\varepsilon>0$ small enough such that
	\begin{align*}
		\zeta := \min_{x \in \overline{\Omega}}\left[\min\{\mathtt{n}_\varepsilon^\ast(x), \mathtt{m}_0^\ast(x) \} - \max\{\mathtt{m}_-(x), \mathtt{n}_\varepsilon(x)\} \right] >0.
	\end{align*}
	Combining this with \eqref{compactness:est-3}, we can find $c_i$ depending on $\zeta, p,q$ and independent of $j$ such that
	\begin{equation}\label{compactness:est-new}
		C_\sharp^{\ell_1} (C^{\ast})^{\ell_2} \leq \theta_j  \quad \text{with }  \ell_1 \in \{c_1, c_2\}  \text{ and }  \ell_2 \in \{c_3, c_4\}.
	\end{equation}
	Using \eqref{lower:est-11} and \eqref{compactness:est-new}, we obtain
	\begin{align*}
		c \geq \delta \theta_j - c_\ast \geq \delta C_\sharp^{\ell_1} (C^{\ast})^{\ell_2} - c_\ast :=c^\ast,
	\end{align*}
	which is a contradiction. Hence $I = \emptyset$ and by virtue of \eqref{weak:conv}, Lemma \ref{BL-lemma} and Proposition \ref{pro:norm-mod:relation}, we obtain
	\begin{align*}
		u_n \to u  \quad\text{in }  L^{\mathcal{S}^\ast}(\Omega).
	\end{align*}
	Again, by taking $u_n-u$ as a test function in \eqref{PS-cond}, we obtain
	\begin{equation}\label{compactness:est-7}
		\langle \mathcal{J}_1(u_n), u_n - u \rangle_{\mathcal{S}} = \langle \mathcal{J}(u_n), u_n - u \rangle_{\mathcal{S}} + \Lambda  \langle \mathcal{J}_2(u_n), u_n- u \rangle_{\mathcal{S}} + \langle \mathcal{J}_3(u_n), u_n- u\rangle_{\mathcal{S}}.
	\end{equation}
	The condition in \eqref{main:assump-1} and \eqref{main:assump-2} imply
	\begin{align*}
		q_\star(x) + s_\star(x) \frac{q_\star(x)}{q(x)} >1 \quad \text{and} \quad q_\ast(x) + s(x) \frac{q^\ast(x)}{q(x)} > 1.
	\end{align*}
	In light of the above implication along with H\"older's inequality as in Proposition \ref{holder:ineq}, we obtain
	\begin{align*}
		& \left| \langle \mathcal{J}_2(u_n), u_n- u \rangle_{\mathcal{S}} \right| \leq \int_{\Omega} \left(a(x)\right)^\frac{p^\ast(x)}{p(x)} |u_n|^{p^\ast(x)-1} |u_n-u| \,\mathrm{d}x \\
		& \quad + \int_{\Omega} \left(b(x) \log^{s(x)}(1+ |u_n|) \right)^\frac{q^\ast(x)}{q(x)} |u_n|^{q^\ast(x)-1} \left(1+ \frac{|s(x)| q^\ast(x)}{q(x)} \frac{|u_n|}{\log(1+ |u_n|)(1+ |u_n|)}\right) |u_n-u| \,\mathrm{d}x \\
		& \quad \leq \int_{\Omega} \left(a(x)\right)^\frac{p^\ast(x)}{p(x)} |u_n|^{p^\ast(x)-1} |u_n-u| \,\mathrm{d}x \\
		& \qquad + \left\|1+ \frac{|s(\cdot)| q^\ast(\cdot)}{q(\cdot)}\right\|_{\infty, \Omega} \int_{\Omega} \left(b(x) \log^{s(x)}(1+ |u_n|) \right)^\frac{q^\ast(x)}{q(x)} |u_n|^{q^\ast(x)-1} |u_n-u| \,\mathrm{d}x  \\
		& \quad \leq C \int_{\Omega} \frac{\mathcal{S}^\ast(x, |u_n|)}{|u_n|} |u_n-u| \,\mathrm{d}x \leq C \left\|\frac{\mathcal{S}^\ast(x, |u_n|)}{|u_n|}\right\|_{(S^\ast)^\sharp} \|u_n-u\|_{\mathcal{S}^\ast},
	\end{align*}
	where $(S^\ast)^\sharp$ denotes the Sobolev conjugate function of $\mathcal{S}^\ast$ defined in Definition \ref{def:conjugate}. Now, by using the conjugate modular relation in Proposition \ref{modular-conjugate:relation} and Proposition \ref{Prop:AbstractEmbedding}, we get
	\begin{equation}\label{compactness:est-4}
		\begin{aligned}
			\left| \langle \mathcal{J}_2(u_n), u_n- u \rangle_{\mathcal{S}} \right| & \leq C \left\|\frac{\mathcal{S}^\ast(x, |u_n|)}{|u_n|}\right\|_{(S^\ast)^\sharp} \|u_n-u\|_{\mathcal{S}^\ast} \\
			& \leq C_1 \|u_n\|_{S^\ast} \|u_n-u\|_{S^\ast} \to 0  \quad\text{as }  n \to \infty.
		\end{aligned}
	\end{equation}
	Following the same arguments as above, we deduce
	\begin{equation}\label{compactness:est-5}
		\begin{aligned}
			\left| \langle \mathcal{J}_3(u_n), u_n- u \rangle_{\mathcal{S}} \right|  \leq C_1 \|u_n\|_{S^\ast_\star} \|u_n-u\|_{S^\ast_\star} \to 0\quad  \text{as }  n \to \infty.
		\end{aligned}
	\end{equation}
	Moreover, the boundedness of $\{u_n\}_{n \in \mathbb{N}}$ in $W_0^{1, \mathcal{S}}(\Omega)$ and \eqref{PS-cond}, gives
	\begin{equation}\label{compactness:est-6}
		\lim_{n \to \infty} \left| \langle \mathcal{J}(u_n), u_n - u \rangle_{\mathcal{S}} \right| =0.
	\end{equation}
	Collecting \eqref{compactness:est-4}, \eqref{compactness:est-5} and \eqref{compactness:est-6} in \eqref{compactness:est-7}, we obtain
	\begin{align*}
		\langle \mathcal{J}_1(u_n), u_n - u \rangle_{\mathcal{S}}  \to 0 \quad \text{as }  n \to \infty.
	\end{align*}
	Finally, Theorem \ref{Th:PropertiesOperator} implies that $u_n \to u$ in $W_0^{1, \mathcal{S}}(\Omega)$ by the \textnormal{(S$_+$)} property of the operator $\mathcal{J}_1$.
\end{proof}

\begin{lemma}
	There exists a sequence $\{R_k\}_{k \in \mathbb{N}}$ independent of $\Lambda$ such that $1 < R_k < R_{k+1}$ for all $k \in \mathbb{N}$ and for each $k \in \mathbb{N}$
	\begin{align*}
		\mathcal{E}_\Lambda(u) <0  \quad \text{for all }  u \in X_k  \text{ with }  \|u\|_{1, \mathcal{S}} > R_k,
	\end{align*}
	where $X_k$ are $k$-dimensional subspaces of $W_0^{1, \mathcal{S}}(\Omega)$.
\end{lemma}

\begin{proof}
	Let $k \in \mathbb{N}$. By the equivalence of norms in $X_k$, we find $\eta_k >1$ such that
	\begin{equation}\label{est:lower-1}
		\eta_k^{-1} \|u\|_{1, \mathcal{S}} \leq \|u\|_{\mathcal{S}^\ast} \leq \eta_k \|u\|_{1, \mathcal{S}}  \quad \text{for all }  u \in X_k.
	\end{equation}
	For any $u \in X_k$ with $\|u\|_{1, \mathcal{S}} \geq R_k > \eta_k >1$, using Proposition \ref{pro:norm-mod:relation} and \eqref{est:lower-1}, we have
	\begin{align*}
		\mathcal{E}_\Lambda(u)
		&\leq \frac{\|\nabla u\|_{\mathcal{S}}^{\max\{p^+, ( q + \lceil s \rceil)^+\}}}{\min\{p^-, q^-\}} - \frac{\|u\|_{\mathcal{S}_\star}^{\min\{p_\star^-, \left(q_\star + \lfloor s_\star \rfloor \frac{q_\star}{q}\right)^- \}}}{\max\{p_\star^+, q_\star^+\}}   \\
		& \leq \frac{\|u\|_{1, \mathcal{S}}^{\max\{p^+, ( q + \lceil s \rceil)^+\}}}{\min\{p^-, q^-\}}  - \frac{ \left(\eta_k^{-1}\|u\|_{1, \mathcal{S}}\right)^{{\min\{p_\star^-, \left(q_\star + \lfloor s_\star \rfloor \frac{q_\star}{q}\right)^- \}}}}{\max\{p_\star^+, q_\star^+\}}  \\
		& \leq \|\nabla u\|_{\mathcal{S}}^{{\min\{p_\star^-, \left(q_\star + \lfloor s_\star \rfloor \frac{q_\star}{q}\right)^- \}}} \left( \frac{\|\nabla u\|_{\mathcal{S}}^{\max\{p^+, ( q + \lceil s \rceil)^+\}- {\min\{p_\star^-, \left(q_\star + \lfloor s_\star \rfloor \frac{q_\star}{q}\right)^- \}}}}{\min\{p^-, q^-\}} \right. \\
		& \left.\qquad\qquad\qquad\qquad\qquad\qquad\qquad\qquad -\frac{\eta_k^{-{\min\{p_\star^-, \left(q_\star + \lfloor s_\star \rfloor \frac{q_\star}{q}\right)^- \}}}}{\max\{p_\star^+, q_\star^+\}}\right) \\
		& \leq \|\nabla u\|_{\mathcal{S}}^{{\min\{p_\star^-, \left(q_\star + \lfloor s_\star \rfloor \frac{q_\star}{q}\right)^- \}}} \left(\frac{R_k^{\max\{p^+, ( q + \lceil s \rceil)^+\}- {\min\{p_\star^-, \left(q_\star + \lfloor s_\star \rfloor \frac{q_\star}{q}\right)^- \}}}}{\min\{p^-, q^-\}}\right.   \\
		& \left.\qquad\qquad\qquad\qquad\qquad\qquad\qquad\qquad
		- \frac{\eta_k^{-{\min\{p_\star^-, \left(q_\star + \lfloor s_\star \rfloor \frac{q_\star}{q}\right)^- \}}}}{\max\{p_\star^+, q_\star^+\}}\right).
	\end{align*}
	In view of \eqref{sandwich:cond-1}, the exponent of $R_k$ is negative. Therefore, by choosing $R_k \gg 1$ independent of $\Lambda$, we have the required claim.
\end{proof}

For each $k \in \mathbb{N}$, we define
\begin{align*}
	Y_k & := \{u \in X_k\colon  \|u\|_{1, \mathcal{S}} \leq R_k\}, \\
	Z_k & := \{g \in C(X_k, W_0^{1, \mathcal{S}}(\Omega))\colon  g  \text{ is odd and }  g(u) = u  \text{ on }  \partial X_k \}
\end{align*}
and
\begin{equation}\label{critical:genus}
	d_k := \inf_{g \in Z_k} \max_{u \in Y_k} \mathcal{E}_\Lambda(g(u)).
\end{equation}

\begin{lemma}\label{lem:MPG}
	Let \eqref{main:assump}, \eqref{main:assump-1}, \eqref{main:assump-1-2}, \eqref{main:assump-2} and \eqref{sandwich:cond-1} be satisfied. Then, the following hold:
	\begin{enumerate}
		\item[\textnormal{(i)}]
			There exist numbers $R > 0$ and $r >0$ such that $\mathcal{E}(u) \geq r$ for every
			$u \in S_R := \{v \in W_0^{1, \mathcal{S}}(\Omega)\colon \|\nabla v\|_{\mathcal{S}} = R\}$.
		\item[\textnormal{(ii)}]
			There exists $w\in W_0^{1, \mathcal{S}}(\Omega)$ with $\|w\|>R$ such that $\mathcal{E}(w)<0$.
	\end{enumerate}
\end{lemma}

\begin{proof}
	Let  $u\in W_0^{1, \mathcal{S}}(\Omega)$ with $\|\nabla u\|_{\mathcal{S}}=R<1$. Using Propositions \ref{imp:embedding} and \ref{pro:norm-mod:relation}, we can choose $R \in (0,1)$ such that $\|u\|_{\mathcal{S}^\ast} + \|u\|_{\mathcal{S}_\star} \leq 1 $ and
	\begin{align*}
		\mathcal{E}(u) & = \mathcal{E}_1(u)- \Lambda  \mathcal{E}_2(u) - \mathcal{E}_3(u)  \geq \frac{\varrho_{\mathcal{S}} (|\nabla u|)}{\max\{p^+, q^+\}} - \Lambda \frac{\varrho_{\mathcal{S}^\ast} (| u|)}{\min\{(p^{\ast})^-, (q^{\ast})^-\}} -  \frac{\varrho_{\mathcal{S}_\star} (| u|)}{\min\{p_{\star}^-, q_{\star}^-\}}   \\
		& \geq \frac{\|u\|^{\max\{p^+, (q + \lceil s \rceil)^+\}}_{1,\mathcal{S}}}{\max\{p^+, q^+\}} - \Lambda \frac{\|u\|_{\mathcal{S}^\ast}^{\min\{(p^\ast)^-, \left( q^\ast + \lfloor s \rfloor \frac{q^\ast}{q} \right)^-\}}}{\min\{(p^{\ast})^-, (q^{\ast})^-\}} - \frac{\|u\|_{\mathcal{S}_\star}^{\min\{(p_\star)^-, \left( q_\star + \lfloor s_\star \rfloor \frac{q_\star}{q} \right)^-\}}}{\min\{p^{-}_{\star}, q^{-}_{\star}\}} \\
		& \geq C_1 \|u\|^{\max\{p^+, (q + \lceil s \rceil)^+\}}_{1,\mathcal{S}} - C_2 \Lambda \|u\|_{1, \mathcal{S}}^{\min\{(p^\ast)^-, \left( q^\ast + \lfloor s \rfloor \frac{q^\ast}{q} \right)^-\}} - C_3 \|u\|_{1, \mathcal{S}}^{\min\{(p_\star)^-, \left( q_\star + \lfloor s_\star \rfloor \frac{q_\star}{q} \right)^-\}}.
	\end{align*}
	Now, in view of \eqref{main:assump-1-2}, \eqref{main:assump-2} and \eqref{sandwich:cond-1}, there exist numbers $R \in (0,1)$ and $r>0$ such that $\mathcal{E}(u) >r$. To prove $\textnormal{(ii)}$, set $w = \theta u$ for some $t >1$ and $v \in W_0^{1, \mathcal{S}}(\Omega)$. Note that if $\theta >1$ we have the following inequalities
	\begin{equation*}
		\begin{aligned}
			\log^{\frac{s_\star(x) q_\star(x)}{q(x)}}(1+ \theta |u|) & \leq
			\begin{cases}
				\log^{\frac{s_\star(x) q_\star(x)}{q(x)}}\left((1+|u|)^\theta\right) \leq \theta^{s(x)} \log^{s(x)}(1+  |u|), & \text{if } s_\star(x) >0,\\
				\log^{\frac{s_\star(x) q_\star(x)}{q(x)}}\left(1+|u|\right) , & \text{if } s_\star(x) \leq 0, \\
			\end{cases}\\
			& \leq  \theta^{\frac{\lceil s_\star\rceil(x) q_\star(x)}{q(x)}}\log^{\frac{s_\star(x) q_\star(x)}{q(x)}}(1+  |u|)
		\end{aligned}
	\end{equation*}
	and
	\begin{equation*}
		\begin{aligned}
			\log^{s(x)}(1+ \theta |u|) & \geq
			\begin{cases}
				\log^{s(x)}\left(1+|u|\right) , & \text{if }s(x) \geq 0, \\
				\log^{s(x)}\left((1+|u|)^\theta\right) \geq \theta^{s(x)} \log^{s(x)}(1+  |u|), & \text{if }s(x) < 0.
			\end{cases}
			\\
			& \geq \theta^{\lfloor s\rfloor (x)} \log^{s(x)}(1+  |u|).
		\end{aligned}
	\end{equation*}
	This further gives
	\begin{equation}\label{estimates:3}
		\begin{aligned}
			\theta^{\min\{p^-, (q + \lfloor s \rfloor)^-\}} \ \varrho_{\mathcal{S}}(u)  \leq  \varrho_{\mathcal{S}}(\theta u) \quad \text{and} \quad  \varrho_{\mathcal{S}_\star}(\theta u) \leq   \theta^{\max\{p^+_\star, (q_\star + \lceil s_\star  \rceil \frac{q_\star}{q})^+\}} \varrho_{\mathcal{S}}(u).
		\end{aligned}
	\end{equation}
	Let $\|u\|_{1,\mathcal{S}} =\eta_1$ and $\|u\|_{\mathcal{S}_\star} =\eta_2$   with $0<\eta_1, \eta_2 <1$. Then from (i), we have $\varrho_{1,\mathcal{S}}\left(\frac{u}{\eta_1}\right)=1$ and $\varrho_{\mathcal{S}_\star}\left(\frac{u}{\eta_2}\right)=1$. Now, by taking $\theta=\frac{1}{\eta_1} >1$ and $\theta=\frac{1}{\eta_2} >1$ in \eqref{estimates:3}, we obtain
	\begin{align*}
		\frac{\varrho_{1,\mathcal{S}}(u)}{\eta_1^{\min\{p^-, (q+\lfloor s\rfloor)^-\}}}  \leq  \varrho_{1,\mathcal{S}}\left(\frac{u}{\eta_1}\right)=1 \quad \text{and} \quad \varrho_{\mathcal{S}_\star}\left(\frac{u}{\eta_2}\right)=1 \leq \frac{\varrho_{\mathcal{S}_\star}(u)}{\eta_2^{\max\{p^+_\star, (q_\star + \lceil s_\star  \rceil \frac{q_\star}{q})^+\}}}.
	\end{align*}
    Using the above estimate as well as \eqref{sandwich:cond-1}, we get
	\begin{align*}
		\mathcal{E}(w) & = \mathcal{E}_1(\theta u)- \Lambda \mathcal{E}_2(\theta u) -  \mathcal{E}_3(\theta u) \\
		& \leq \theta^{\max\{p^+, (q + \lceil s \rceil)^+\}} \mathcal{E}_1(u) -  \theta^{\min\left\{p_\star^-, \left(q_\star + \lfloor s_\star \rfloor \frac{q_\star}{q}\right)^- \right\}} \mathcal{E}_3(u) \to -\infty  \quad\text{as }  t \to \infty.
	\end{align*}
	Hence, by taking $\theta$ large enough, we have $\|\nabla w\|_{\mathcal{S}} > R$ and $\mathcal{E}(w) <0$.
\end{proof}

Using Lemma \ref{lem:MPG}, the deformation lemma from Ambrosetti--Rabinowitz \cite[Lemma 1.3]{Ambrosetti-Rabinowitz-1973} and following the same arguments as in Theorem 2.1 by Ambrosetti--Rabinowitz \cite{Ambrosetti-Rabinowitz-1973}, we obtain the following result.

\begin{lemma}\label{critical:values}
	For each $k \in \mathbb{N}$, $d_k$ is the critical value of $\mathcal{E}_\Lambda$ provided $\mathcal{E}_\Lambda$ satisfies the $\textnormal{(PS)}_{d_k}$ condition.
\end{lemma}

Let $\{T_k\}_{k \in \mathbb{N}}$ be a sequence of closed linear subspaces of $W_0^{1, \mathcal{S}}(\Omega)$ with finite codimensions and $\{e_k\}_{k \in \mathbb{N}}$ be a Schauder basis of $W_0^{1, \mathcal{S}}(\Omega)$. For each $n \in \mathbb{N}$, let $f_n \in (W_0^{1, \mathcal{S}}(\Omega))^*$ be defined as
\begin{align*}
	f_n(u) = \alpha_n \quad \text{for }  u = \sum_{j=1}^\infty \alpha_j e_j \in W_0^{1, \mathcal{S}}(\Omega).
\end{align*}
For each $k \in \mathbb{N}$, define
\begin{align*}
	U_k := \{ u \in W_0^{1, \mathcal{S}}(\Omega)\colon  f_n(u) = 0  \text{ for all }  n \geq k\}
\end{align*}
and
\begin{align*}
	V_k:= \{ u \in W_0^{1, \mathcal{S}}(\Omega)\colon  f_n(u) = 0  \text{ for all }  n \leq k-1\}.
\end{align*}
Then, $W_0^{1, \mathcal{S}}(\Omega) = U_k \oplus V_k$ and $V_k$ has codimension $k-1$. Define
\begin{align*}
	e_k:= \sup_{v \in V_k, \|v\|_{W_0^{1, \mathcal{S}}(\Omega)} \leq 1} \|v\|_{\mathcal{S}_\star}.
\end{align*}

\begin{lemma}\label{conv-prop}
	The sequence $\{e_k\}_{k \in \mathbb{N}}$ defined above satisfies $0< e_{k+1} \leq e_{k}$ for $k \in \mathbb{N}$ and $\lim_{k \to \infty} e_k =0$.
\end{lemma}

\begin{proof}
	Because $V_{k+1} \subset V_k$, we have $e_k \geq e_{k+1} \geq 0$ for all $k \in \mathbb{N}$. Hence $e_k \to e \geq 0$. By the definition of $e_k$, for each $k \in \mathbb{N}$, we can choose $u_k \in V_k$ with  $ \|\nabla u_k\|_{\mathcal{S}}  \leq 1$, such that
	\begin{align*}
		0 \leq e_k - \|\nabla u_k\|_{\mathcal{S}} < \frac{1}{k}.
	\end{align*}
	Using Proposition \ref{pro:spaces-propert}, $W_0^{1, \mathcal{S}}(\Omega)$ is reflexive and $\{u_k\}_{k\in\mathbb{N}}$ is bounded in $W_0^{1, \mathcal{S}}(\Omega)$. So, up to a subsequence, we have $u_k \rightharpoonup u$ in $W_0^{1, \mathcal{S}}(\Omega)$. Then, by Proposition \ref{imp:embedding}, $u_k \to u$ in $L^{\mathcal{S}_\star}(\Omega)$. Now, by the definition of $V_k$ and the fact that $u_k \in V_k$, for each $n \in \mathbb{N}$ and $n<k$, we have $f_n(u_k) =0$ and by passing to the limit as $k \to \infty$ leads to $f_n(u) =0$ for all $n \in \mathbb{N}$. Hence $u \equiv 0$ and $\lim_{k \to \infty} e_k = e =0$.
\end{proof}

\begin{lemma}\label{lem:upper_lower-est}
	There exists $C_1, C_2 >0$ such that
	\begin{align*}
		L(R_k) \leq d_k \leq U(R_k)  \quad\text{for every }  k \in \mathbb{N},
	\end{align*}
	where
	\begin{align*}
		L(R_k):= \frac{R_k^{\max\{p^+, ( q + \lceil s \rceil)^+\}}}{\min\{p^-, q^-\}}  \quad\text{and}\quad  U(R_k) = \frac{C_1}{2 C_2} R_k^{\alpha_1^{(2)} - \alpha_2^{(2)}}.
	\end{align*}
\end{lemma}

\begin{proof}
	Let $\Lambda >0$, $k \in \mathbb{N}$ and $d_k$ is given in \eqref{critical:genus}. First, we will find the upper and lower bound of $d_k$. Using Proposition \ref{pro:norm-mod:relation} for all $u \in Y_k$, it is easy to see that
	\begin{equation}\label{est:upper_cric_1}
		\begin{aligned}
			&\mathcal{E}_\Lambda(u)\\
			& = \mathcal{E}_1(u) - \Lambda \mathcal{E}_2(u)- \mathcal{E}_3(u) \leq \frac{1}{\min\{p^-, q^-\}} \int_{\Omega} \mathcal{S}(x, \nabla u) \,\mathrm{d}x \\
			& \leq \frac{1}{\min\{p^-, q^-\}} \max\{\|\nabla u\|_{\mathcal{S}}^{\min\{p^-, (q+\lfloor s \rfloor)^-\}}, \|\nabla u\|_{\mathcal{S}}^{\max\{p^+, ( q + \lceil s \rceil)^+\}}\} \leq \frac{R_k^{\max\{p^+, ( q + \lceil s \rceil)^+\}}}{\min\{p^-, q^-\}}.
		\end{aligned}
	\end{equation}
	On the other hand, since $Id \in Z_k,$ the definition of $d_k$ gives
	\begin{equation}\label{est:upper_cric_2}
		d_k \leq \max_{u \in Y_k} \mathcal{E}_\Lambda(u).
	\end{equation}
	Combining \eqref{est:upper_cric_1} and \eqref{est:upper_cric_2}, we get
	\begin{align*}
		d_k \leq \frac{R_k^{\max\{p^+, ( q + \lceil s \rceil)^+\}}}{\min\{p^-, q^-\}} := L(R_k).
	\end{align*}
	To get the lower bound of $d_k$, we make use of the following fact, which is a consequence of Lemma 3.9 by Komiya--Kajikiya \cite{Komiya-Kajikiya-2016}
	\begin{align*}
		g(Y_k) \cap \partial B_\tau \cap V_k \neq \emptyset  \quad \text{for all }  g \in Z_k  \text{ and }  \tau \in (0, R_k).
	\end{align*}
	Thus, for any $\tau \in (0, R_k)$, we have
	\begin{align*}
		\max_{u \in Y_k} \mathcal{E}_\Lambda(g(u)) \geq \inf_{u \in \partial B_\tau \cap V_k} \mathcal{E}_\Lambda(u)  \quad \text{for all }  g \in Z_k.
	\end{align*}
	Let $\tau \in (1, R_k)$ be arbitrary and fixed. The above inequality in view of definition of $d_k$ gives
	\begin{align*}
		d_k \geq \inf_{u \in \partial B_\tau \cap V_k} \mathcal{E}_\Lambda(u).
	\end{align*}
	By Lemma \ref{conv-prop} we can find $k_0 \in \mathbb{N}$ such that for $k \geq k_0$, we have
	\begin{align*}
		\|u\|_{\mathcal{S}_\star} \leq e_k \|\nabla u\|_{ \mathcal{S}}  \quad \text{for all }  u \in V_k.
	\end{align*}
	Combining this with Proposition \ref{pro:norm-mod:relation} and Proposition \ref{imp:embedding} for any $u \in \partial B_\tau \cap V_k$ for $k \geq k_0$ and $\|u\|_{\mathcal{S}} = \tau > 1$, we get
	\begin{equation}\label{lower-est-1}
		\begin{aligned}
			\mathcal{E}_\Lambda(u) & \geq \frac{\|\nabla u\|_{\mathcal{S}}^{\alpha_1^{(2)}}}{\max\{p^+, q^+\}} - \frac{\Lambda \max\{\|u\|_{\mathcal{S}^\ast}^{\alpha_2^{(1)}}, \|u\|_{\mathcal{S}^\ast}^{\alpha_2^{(2)}}\}}{\min\{(p^\ast)^-, (q^\ast)^-\}} - \frac{\max\{\|u\|_{\mathcal{S}_\star}^{\alpha_3^{(1)}}, \|u\|_{\mathcal{S}_\star}^{\alpha_3^{(2)}}\}}{\min\{p_\star^-, q_\star^-\}} \\
			& \geq C_1 \|\nabla u\|_{\mathcal{S}}^{\alpha_1^{(2)}} - \Lambda C_2 \|\nabla u\|_{\mathcal{S}}^{\alpha_2^{(2)}} - e_k^{\alpha_3^{(2)}} \|\nabla u\|_{\mathcal{S}}^{\alpha_3^{(2)}}  \\
			& = C_1 \tau^{\alpha_1^{(2)}} - \Lambda C_2 \tau^{\alpha_2^{(2)}} - e_k^{\alpha_3^{(2)}} \tau^{\alpha_3^{(2)}},
		\end{aligned}
	\end{equation}
	where $\alpha_1^{(2)}=\min\{p^-, (q+\lfloor s \rfloor)^-\},$
	\begin{align}
		\alpha_2^{(1)} &= \min\left\{(p^\ast)^-, \left( q^\ast + \lfloor s \rfloor \frac{q^\ast}{q} \right)^-\right\} \quad  \alpha_2^{(2)} = \max\left\{(p^\ast)^+, \left( q^\ast + \lceil s \rceil \frac{q^\ast}{q} \right)^+\right\},\label{notation-1}\\
		\alpha_3^{(1)} &= \min\left\{(p_\star)^-, \left( q_\star + \lfloor s_\star \rfloor \frac{q_\star}{q} \right)^-\right\} \quad \alpha_3^{(2)} =\max\left\{(p_\star)^+, \left( q_\star + \lceil s_\star \rceil \frac{q_\star}{q} \right)^+\right\}\label{notation-2}
	\end{align}
	and $C_i$ for $i=1,2$ depend on the given data and the constants in Proposition \ref{imp:embedding}. Note that for $\Lambda>0$ satisfying
	\begin{align*}
		0< \Lambda < \frac{C_1}{2 C_2} R_k^{\alpha_1^{(2)} - \alpha_2^{(2)}} =: U(R_k)
	\end{align*}
	we get
	\begin{align*}
		C_2 \Lambda \tau^{\alpha_2^{(2)}} \leq \frac{C_1}{2} \tau^{\alpha_1^{(2)}} \quad \text{for all }  \tau \in (1, R_k).
	\end{align*}
	Using this with \eqref{sandwich:cond-1}, \eqref{lower-est-1} and \eqref{critical:genus} for any $k \geq k_0$, we obtain
	\begin{equation}\label{critical-relation}
		d_k \geq \frac{C_1}{2} \tau^{\alpha_1^{(2)}} - e_k^{\alpha_3^{(2)}} \tau^{\alpha_3^{(2)}} := \ell(\tau) \geq \ell(\tau_0)\geq \beta e_k^{- \frac{\alpha_3^{(2)} \alpha_1^{(2)}}{\alpha_3^{(2)}-\alpha_1^{(2)}}},
	\end{equation}
	where
	\begin{align*}
		\tau_0 := \left(\frac{C_1 \alpha_1^{(2)}}{2 \alpha_3^{(2)} e_k^{\alpha_3^{(2)}}}\right)^\frac{1}{\alpha_3^{(2)}-\alpha_1^{(2)}}  \quad\text{and}\quad \beta:= \left(\frac{C_1}{2}\right)^\frac{\alpha_3^{(2)} }{\alpha_3^{(2)}-\alpha_1^{(2)}} \left(\frac{\alpha_1^{(2)}}{\alpha_3^{(2)}} \right)^ \frac{\alpha_1^{(2)} \alpha_1^{(2)}}{\alpha_3^{(2)}-\alpha_1^{(2)}} \left(\frac{\alpha_3^{(2)}- \alpha_1^{(2)}}{\alpha_3^{(2)}}\right) >0.
	\end{align*}
\end{proof}

Now we can prove our multiplicity results in the superlinear case.

\begin{theorem}\label{theo-1}
	Let \eqref{main:assump}, \eqref{main:assump-1}, \eqref{main:assump-1-2}, \eqref{main:assump-2} and \eqref{sandwich:cond-1} be satisfied. Then, for each $n \in \mathbb{N}$, there exists $\Lambda_n>0$ such that for any $\Lambda \in (0, \Lambda_n)$ and $\lambda =1$, the problem \eqref{main:prob} admits at least $n$ pairs of non-trivial weak solutions.
\end{theorem}

\begin{proof}
	We choose the sequence $\{\Lambda_n\}_{n \in \mathbb{N}}$ as follows. By Lemma \ref{conv-prop}, we find for $k \geq k_0$ that $e_{k} >0$. Then, take $\Lambda_1$ satisfying
	\begin{align*}
		\Lambda_1 \in (0, U(R_{k_1}))  \quad \text{and}\quad d_{k_1} < L(R_{k_1}) \leq c^{\ast}(\Lambda_1)  \quad\text{for }  k_1 >k_0,
	\end{align*}
	where the function $U$ and $V$ are defined in Theorem \ref{lem:upper_lower-est}. The above choice of $\Lambda_1$ implies that $\mathcal{E}_{\Lambda}$ satisfies the $\text{(PS)}_{d_{k_1}}$ condition thanks to Lemma \ref{PS-cond-super}. Inductively, we define $\{\Lambda_k\}_{k \in \mathbb{N}}$ satisfying
	\begin{align*}
		\Lambda_n \in (0, U(R_{k_n}))  \quad \text{and}\quad d_{k_n} < L(R_{k_n}) \leq c^{\ast}(\Lambda_n)  \quad\text{for }  k_n >k_0.
	\end{align*}
	Now, let $\Lambda \in (0, \Lambda_n)$ for some $n \in \mathbb{N}$. Then, by using the definition of $\Lambda_n$, Lemma \ref{conv-prop} and \eqref{critical-relation}, we have
	\begin{align*}
		0 < d_{k_1} < d_{k_2} \cdots < d_{k_n} < c^\ast(\Lambda).
	\end{align*}
	Thus, in the view of Lemma \ref{PS-cond-super} and \ref{critical:values}, $d_{k_1}, d_{k_2}, \cdots, d_{k_n}$ are distinct critical values of $\mathcal{E}_\Lambda$. Therefore, $\mathcal{E}_\Lambda$ has at least $n$ distinct pairs of critical points.
\end{proof}

\subsection{Sublinear growth}

In this subsection, we assume that $\Lambda =1$ and $\lambda >0$. For the sake of simplicity, we write $\mathcal{E}_{\Lambda, \lambda}$ as $ \mathcal{E}_\lambda$. We suppose the following conditions:
\begin{enumerate}[label=\textnormal{(H$_\star^{\text{sub}}$)},ref=\textnormal{H$_\star^{\text{sub}}$}]
	\item\label{main:assump-4}
		$p_\star, q_\star \in C(\overline{\Omega})$, $s_\star \in L^\infty(\Omega)$, $1 < p_\star(x), q_\star(x)< N$, $p_{\star}(x) < p(x)$ and $q_{\star}(x) < q(x)$ for all $x \in \overline{\Omega}$, and $q_\star(x)+ s_\star(x) \geq r> 1$, $ s_{\star}(x) < s(x)$ for a.a.\,$x \in \Omega$.
\end{enumerate}

\begin{lemma}\label{PS:bounded-sub}
	Let \eqref{main:assump}, \eqref{main:assump-1}, \eqref{main:assump-1-2} and \eqref{main:assump-4} be satisfied and
	\begin{equation}\label{sandwich:cond-2-sub}
		\max\left\{p_\star^+, \left(q_\star + \lceil s_\star \rceil \frac{q_\star}{q}\right)^+ \right\} < \min\{p^-, (q+ \lfloor s \rfloor)^-\}.
	\end{equation}
	Then, there exists $\lambda_0 >0$ such that for $\lambda \in (0, \lambda_0)$, every Palais-Smale sequence $\{u_n\}_{n \in \mathbb{N}} \subset W_0^{1, \mathcal{S}}(\Omega)$ is bounded.
\end{lemma}

\begin{proof}
	By the definition of the Palais-Smale sequence $\{u_n\}_{n \in \mathbb{N}}$, we have
	\begin{equation}\label{PS-cond-sub}
		\mathcal{E}_\lambda(u_n) \to c \quad \text{ and } \quad \langle \mathcal{E}_\lambda'(u_n), \phi \rangle \to 0 \quad \text{for every }  \phi \in W_0^{1, \mathcal{S}}(\Omega),  \text{ for some }  c \in \mathbb{R}.
	\end{equation}
	Setting
	\begin{align*}
		\rho:= \frac{ \max\left\{p_\star^+, \left(q_\star + \lceil s_\star \rceil \frac{q_\star}{q}\right)^+ \right\} + \min\{p^-, (q+ \lfloor s \rfloor)^-\}}{2}.
	\end{align*}
    Choosing $\frac{u_n}{\rho}$ as a test function in \eqref{multi-phase}, and  using \eqref{PS-cond-sub} for $n \geq n_0$, we obtain,
	\begin{equation}\label{PS-bdd-1-sub}
		\begin{aligned}
			c+1 &\geq \mathcal{E}_\lambda(u_n) - \langle \mathcal{J}(u_n), \frac{u_n}{\rho} \rangle = \sum_{i=1}^2 \mathcal{E}_i(u_n) - \langle \mathcal{J}_i(u_n), \frac{u_n}{\rho} \rangle \geq  \sum_{i=1}^2 \mathcal{I}_i(u_n),
		\end{aligned}
	\end{equation}
	where
	\begin{align*}
		\mathcal{I}_1(u_n)
		&:= \int_{\Omega} \left(\frac{1}{p(x)} -\frac{1}{\rho}\right) a(x) |\nabla u_n|^{p(x)} \,\mathrm{d}x  + \int_{\Omega} \left( \frac{1}{\rho} - \frac{1}{p^\ast(x)}\right) \left(a(x)^\frac{1}{p(x)} |u_n|\right)^{p^\ast(x)} \,\mathrm{d}x\\
		&\qquad + \lambda  \int_{\Omega} \left( \frac{1}{\rho} - \frac{1}{p_\star(x)} \right) \left(a(x)^\frac{1}{p(x)} |u_n|\right)^{p_\star(x)} \,\mathrm{d}x,
	\end{align*}
	and
	\begin{align*}
		\mathcal{I}_2(u_n)
		&:= \int_{\Omega} \left(\frac{1}{q(x)}- \frac{1}{\rho} \right) b(x) |\nabla u_n|^{q(x)} \log^{s(x)}(1+ |\nabla u_n|) \,\mathrm{d}x \\
		& \quad - \int_{\Omega} \frac{b(x) s(x)}{q(x) \rho} |\nabla u_n|^{q(x)+1} \frac{\log^{s(x)-1 }(1+ |\nabla u_n|)}{1+ |\nabla u_n|}  \,\mathrm{d}x  \\
		& \quad+ \int_{\Omega} \left(\frac{1}{\rho} - \frac{1}{q^\ast(x)} +  \frac{\lfloor s(x)\rfloor q^\ast(x)}{q(x) \rho}\right)\left(b(x) \log^{s(x)}(1+ |u_n|) \right)^\frac{q^\ast(x)}{q(x)} |u_n|^{q^\ast(x)} \,\mathrm{d}x \\
		& \quad+ \lambda  \int_{\Omega} \left( \frac{1}{\rho} - \frac{1}{q_\star(x)} + \frac{\lfloor s_\star(x) \rfloor q_\star(x)}{q(x) \rho}\right) \left(b(x) \log^{s_\star(x)}(1+ |u_n|) \right)^\frac{q_\star(x)}{q(x)} |u_n|^{q_\star(x)} \,\mathrm{d}x.
	\end{align*}
	\textbf{Estimate for $\mathcal{I}_1(u_n):$} By condition \eqref{sandwich:cond-2-sub}, we can choose $\delta>0$ small enough such that
	\begin{align*}
		\delta \leq \max\limits_{x \in \overline{\Omega}} \frac{1}{2} \max\left\{\frac{1}{p(x)} -\frac{1}{\rho},  \frac{1}{\rho} - \frac{1}{p^\ast(x)} \right\}.
	\end{align*}
	This further leads to
	\begin{equation}\label{lower:est-5-sub}
		\begin{aligned}
			\mathcal{I}_1(u_n) & \geq \delta \int_{\Omega} a(x) |\nabla u_n|^{p(x)} \,\mathrm{d}x + \delta \int_{\Omega} \left(a(x)^\frac{1}{p(x))} |u_n|\right)^{p^\ast(x)} \,\mathrm{d}x  \\
			& \qquad - \lambda C_0(N,p) \int_{\Omega} \left(a(x)^\frac{1}{p(x)} |u_n|\right)^{p_\star(x)} \,\mathrm{d}x.
		\end{aligned}
	\end{equation}
	\textbf{Estimate for $\mathcal{I}_2(u_n):$} For every $\delta >0$ we can find a constant $M(\delta, s^+) >0$ such that $\frac{t}{\log(1+t)(1+t)} < \frac{\delta}{s^+ +1}$ for $t \geq M$. Therefore we have the following estimates: By splitting the domain depending upon the size of $|\nabla u_n|$, we get
	\begin{equation}\label{lower:est-7-sub}
		\begin{aligned}
			&\bigg|\int_{\Omega} \frac{b(x) s(x)}{q(x) \rho} |\nabla u_n|^{q(x)+1} \frac{\log^{s(x)-1 }(1+ |\nabla u_n|)}{1+ |\nabla u_n|}  \,\mathrm{d}x \bigg|   \\
			& \leq C(\nu) + \nu \int_{\Omega} b(x) |\nabla u_n|^{q(x)} \log^{s(x)}(1+ |\nabla u_n|) \,\mathrm{d}x.
		\end{aligned}
	\end{equation}
	We choose $\delta, \nu >0$ small enough such that
	\begin{align*}
		\delta \leq \frac{1}{4} \max\limits_{x \in \operatorname{supp}(b)} \max\left\{\frac{1}{q(x)} -\frac{1}{\rho} -\nu,  \frac{1}{\rho} - \frac{1}{q^\ast(x)} \right\}.
	\end{align*}
	By using the estimates in \eqref{lower:est-7-sub}, Propositions \ref{pro:norm-mod:relation} and \ref{imp:embedding}, we obtain
	\begin{equation}\label{lower:est-10-sub}
		\begin{aligned}
			\mathcal{I}_2(u_n) & \geq \delta \int_{\Omega} b(x) |\nabla u_n|^{q(x)} \log^{s(x)}(1+ |\nabla u_n|) \,\mathrm{d}x\\
			&\quad + \delta \int_{\Omega} \left(b(x) \log^{s(x)}(1+ |u_n|) \right)^\frac{q^\ast(x)}{q(x)} |u_n|^{q^\ast(x)} \,\mathrm{d}x \\
			& \quad - \lambda C_0 \int_{\Omega} \left(b(x) \log^{s_\star(x)}(1+ |u_n|) \right)^\frac{q_\star(x)}{q(x)} |u_n|^{q_\star(x)} \,\mathrm{d}x - C(M, \delta).
		\end{aligned}
	\end{equation}
	Inserting the estimates \eqref{lower:est-5-sub} and \eqref{lower:est-10-sub} in \eqref{PS-bdd-1-sub} and using \eqref{cond:AbstractEmbedding} with $\lambda < \lambda_0:=\frac{\delta}{CC_0} $, we deduce
	\begin{equation}\label{lower:est-11-sub}
		\begin{aligned}
			c+ c_\ast(\delta) & \geq  \delta \int_{\Omega} \mathcal{S}(x, |\nabla u_n|) \,\mathrm{d}x + \delta \int_{\Omega} \mathcal{S}^\ast(x, |u_n|) \,\mathrm{d}x - \lambda C_0 \int_{\Omega} \mathcal{S}_\star(x, |u_n|) \,\mathrm{d}x \\
			& \geq \delta \int_{\Omega} \mathcal{S}(x, |\nabla u_n|) \,\mathrm{d}x + (\delta - \lambda C C_0) \int_{\Omega} \mathcal{S}^\ast(x, |u_n|) \,\mathrm{d}x - \lambda C_0 \int_{\Omega} h(x) \,\mathrm{d}x  \\
			& \geq \delta \int_{\Omega} \mathcal{S}(x, |\nabla u_n|) \,\mathrm{d}x - \frac{\delta}{C}  \int_{\Omega} h(x) \,\mathrm{d}x.
		\end{aligned}
	\end{equation}
	This shows the assertion of the lemma.
\end{proof}

\begin{lemma}
	Let \eqref{main:assump}, \eqref{main:assump-1}, \eqref{main:assump-1-2}, \eqref{main:assump-4} and \eqref{sandwich:cond-2-sub} be satisfied. Then the energy functional $\mathcal{E}_\lambda$ satisfy the \textnormal{(PS)$_c$} condition for $c < c^{\ast\ast}$ for all $\lambda \in (0, \lambda_0)$, where $\lambda_0$ is given in Theorem \ref{PS:bounded-sub}.
\end{lemma}

\begin{proof}
	Let $\{u_n\}_{n \in \mathbb{N}}$ be a Palais-Smale sequence, that is, \eqref{PS-cond} holds. Taking Lemma \ref{PS:bounded} and Theorem \ref{concentration:compactness} into account, we can find a weakly convergent subsequence satisfying \eqref{weak:conv}.
	We claim that $I = \emptyset$. Assume we can find $j \in I$. By following the same arguments as in Lemma \ref{PS-cond-super}, we obtain
	\begin{equation}\label{compactness:est-1-sub}
		C_\sharp^{\ell_1} (C^{\ast})^{\ell_2} \leq \theta_j  \quad \text{with }  \ell_1 \in \{c_1, c_2\}  \text{ and }  \ell_2 \in \{c_3, c_4\}.
	\end{equation}
	Now, from \eqref{lower:est-11-sub}  and \eqref{compactness:est-1-sub}, it follows that
	\begin{align*}
		c \geq \delta \theta_j - \lambda_0 C_0 \|h\|_{L^1(\Omega)}- c_\ast \geq \delta C_\sharp^{\ell_1} (C^{\ast})^{\ell_2} - \lambda_0 C_0 \|h\|_{L^1(\Omega)} - c_\ast :=c^{\ast \ast},
	\end{align*}
	which is a contradiction. Therefore $I = \emptyset$ and so, from \eqref{weak:conv}, Lemma \ref{BL-lemma} and Proposition \ref{pro:norm-mod:relation}, we conclude that
	\begin{align*}
		u_n \to u  \quad\text{in }  L^{\mathcal{S}^\ast}(\Omega).
	\end{align*}
	Using $u_n-u$ as a test function in \eqref{PS-cond} gives
	\begin{equation}\label{compactness:est-7_new}
		\langle \mathcal{J}_1(u_n), u_n - u \rangle_{\mathcal{S}} = \langle \mathcal{J}(u_n), u_n - u \rangle_{\mathcal{S}} + \langle \mathcal{J}_2(u_n), u_n- u \rangle_{\mathcal{S}} + \lambda \langle \mathcal{J}_3(u_n), u_n- u\rangle_{\mathcal{S}}.
	\end{equation}
	From \eqref{main:assump-1} and \eqref{main:assump-4} we conclude that
	\begin{align*}
		q_\star(x) + s_\star(x) \frac{q_\star(x)}{q(x)} >1 \quad \text{and} \quad q_\ast(x) + s(x) \frac{q^\ast(x)}{q(x)} > 1.
	\end{align*}
	From this along with H\"older's inequality as in Proposition \ref{holder:ineq}, one has
	\begin{align*}
		& \left| \langle \mathcal{J}_2(u_n), u_n- u \rangle_{\mathcal{S}} \right| \leq \int_{\Omega} \left(a(x)\right)^\frac{p^\ast(x)}{p(x)} |u_n|^{p^\ast(x)-1} |u_n-u| \,\mathrm{d}x   \\
		& \quad + \int_{\Omega} \left(b(x) \log^{s(x)}(1+ |u_n|) \right)^\frac{q^\ast(x)}{q(x)} |u_n|^{q^\ast(x)-1} \left(1+ \frac{|s(x)| q^\ast(x)}{q(x)} \frac{|u_n|}{\log(1+ |u_n|)(1+ |u_n|)}\right) |u_n-u| \,\mathrm{d}x \\
		& \leq \int_{\Omega} \left(a(x)\right)^\frac{p^\ast(x)}{p(x)} |u_n|^{p^\ast(x)-1} |u_n-u| \,\mathrm{d}x \\
		& \quad + \left\|1+ \frac{|s(\cdot)| q^\ast(\cdot)}{q(\cdot)}\right\|_{\infty, \Omega} \int_{\Omega} \left(b(x) \log^{s(x)}(1+ |u_n|) \right)^\frac{q^\ast(x)}{q(x)} |u_n|^{q^\ast(x)-1} |u_n-u| \,\mathrm{d}x  \\
		& \leq C \int_{\Omega} \frac{\mathcal{S}^\ast(x, |u_n|)}{|u_n|} |u_n-u| \,\mathrm{d}x \leq C \left\|\frac{\mathcal{S}^\ast(x, |u_n|)}{|u_n|}\right\|_{(S^\ast)^\sharp} \|u_n-u\|_{\mathcal{S}^\ast}
	\end{align*}
	with $(S^\ast)^\sharp$ being the Sobolev conjugate function of $\mathcal{S}^\ast$ as defined in Definition \ref{def:conjugate}. Now, applying the conjugate modular relation in Proposition \ref{modular-conjugate:relation} and Proposition \ref{Prop:AbstractEmbedding} yields
	\begin{equation}\label{compactness:est-4_new}
		\begin{aligned}
			\left| \langle \mathcal{J}_2(u_n), u_n- u \rangle_{\mathcal{S}} \right| & \leq C \left\|\frac{\mathcal{S}^\ast(x, |u_n|)}{|u_n|}\right\|_{(S^\ast)^\sharp} \|u_n-u\|_{\mathcal{S}^\ast} \\
			& \leq C_1 \|u_n\|_{S^\ast} \|u_n-u\|_{S^\ast} \to 0  \text{as}  n \to \infty.
		\end{aligned}
	\end{equation}
	With the same arguments as above, it is concluded that
	\begin{equation}\label{compactness:est-5_new}
		\begin{aligned}
			\left| \langle \mathcal{J}_3(u_n), u_n- u \rangle_{\mathcal{S}} \right|  \leq C_1 \|u_n\|_{S^\ast_\star} \|u_n-u\|_{S^\ast_\star} \to 0  \quad\text{as }  n \to \infty.
		\end{aligned}
	\end{equation}
	Furthermore, by the boundedness of $\{u_n\}_{n \in \mathbb{N}}$ in $W_0^{1, \mathcal{S}}(\Omega)$ along with \eqref{PS-cond}, we obtain
	\begin{equation}\label{compactness:est-6_new}
		\lim_{n \to \infty} \left| \langle \mathcal{J}(u_n), u_n - u \rangle_{\mathcal{S}} \right| =0.
	\end{equation}
	Collecting the estimates \eqref{compactness:est-4_new}, \eqref{compactness:est-5_new} and \eqref{compactness:est-6_new} in \eqref{compactness:est-7_new}, we have
	\begin{align*}
		\langle \mathcal{J}_1(u_n), u_n - u \rangle_{\mathcal{S}}  \to 0  \quad\text{as }  n \to \infty.
	\end{align*}
	Then, using Theorem \ref{Th:PropertiesOperator} gives $u_n \to u$ in $W_0^{1, \mathcal{S}}(\Omega)$ due to the \textnormal{(S$_+$)} property of the operator $\mathcal{J}_1$.
\end{proof}

Let $u \in W_0^{1, \mathcal{S}}(\Omega)$ be such that $\|\nabla u\|_{\mathcal{S}} \leq 1$. By Proposition \ref{pro:norm-mod:relation}, we have
\begin{equation}\label{trunc:est-1}
	\begin{aligned}
		\mathcal{E}_\lambda(u) & \geq \frac{1}{\max\{p^+, q^+\}} \|\nabla u\|^{\alpha_1}_{\mathcal{S}} - \frac{1}{\max\{(p^\ast)^-, (q^\ast)^-\}} \min\{ \|u\|^{\alpha_2^{(1)}}_{\mathcal{S}^\ast}, \|u\|^{\alpha_2^{(2)}}_{\mathcal{S}^\ast}  \} \\
		& \quad - \frac{\lambda}{\max\{(p^\ast_\star)^-, (q^\ast_\star)^-\}} \min\{\|u\|^{\alpha_3^{(1)}}_{\mathcal{S}_\star}, \|u\|^{\alpha_3^{(2)}}_{\mathcal{S}_\star} \},
	\end{aligned}
\end{equation}
where $\alpha_1 = \max\{p^+, (q + \lceil s \rceil)^+\},$ and $\alpha_2^{(i)}$, $\alpha_3^{(i)}$ for $i=1,2$ are defined in \eqref{notation-1} and \eqref{notation-2}. Now, by using the Sobolev embedding in Proposition \ref{imp:embedding}, there exists a constant $C>0$ such that
\begin{align*}
	\max\{\|u\|_{\mathcal{S}_\star}, \|u\|_{\mathcal{S}^\ast} \} \leq C \|\nabla u\|_{\mathcal{S}} \quad \text{for all }  u \in W_0^{1, \mathcal{S}}(\Omega).
\end{align*}
Using the above inequality in \eqref{trunc:est-1} for $\|u\| \leq 1$, we obtain
\begin{equation}\label{trunc:est-0}
	\begin{aligned}
		\mathcal{E}_\lambda(u)
		&\geq \frac{1}{\max\{p^+, q^+\}} \|\nabla u\|^{\alpha_1}_{\mathcal{S}} - \frac{C^{\alpha_2}}{\max\{(p^\ast)^-, (p^\ast)^-\}} \|\nabla u\|^{\alpha_2}_{\mathcal{S}} - \frac{\lambda C^{\alpha_3}}{\max\{(p_\star)^-, (q_\star)^-\}} \|\nabla u\|^{\alpha_3}_{\mathcal{S}} \\
		& = C_1 \|\nabla u\|^{\alpha_3}_{\mathcal{S}} \left(C_2 \|\nabla u\|^{\alpha_1-\alpha_3}_{\mathcal{S}} - C_3 \|\nabla u\|^{\alpha_2-\alpha_3}_{\mathcal{S}} - \lambda \right) :=  h_\lambda(\|\nabla u\|_{\mathcal{S}}),
	\end{aligned}
\end{equation}
where $\alpha_2 \in \{\alpha_2^{(1)}, \alpha_2^{(2)}\}$, $\alpha_3 \in \{\alpha_3^{(1)}, \alpha_3^{(2)}\},$
\begin{align*}
	C_1&:= \frac{\max\{(p_\star)^-, (q_\star)^-\}}{C^{\alpha_3}}, \quad C_2:= \frac{C^{\alpha_3}}{\max\{(p_\star)^-, (q_\star)^-\} \max\{p^+, q^+\}},\\
	C_3&:= \frac{C^{\alpha_2+\alpha_3}}{\max\{(p_\star)^-, (q_\star)^-\} \max\{(p^\ast)^-, (p^\ast)^-\}}.
\end{align*}
In order to study the behavior of $h_\lambda$, we define
\begin{align*}
	g(t) = C_2 t^{\alpha_1 - \alpha_3} - C_3 t^{\alpha_2 - \alpha_3}.
\end{align*}
It is easy to see that the function $g$ has a maximum at $t_\ast = \left[\frac{(\alpha_1-\alpha_3) C_2}{(\alpha_2 - \alpha_3)C_3}\right]^{\frac{1}{(\alpha_2 - \alpha_3)}},$ and the maximum value is
\begin{align*}
	\lambda_1 := g(t_\ast) = C_2^\frac{\alpha_2-\alpha_3}{\alpha_2-\alpha_1} C_3^\frac{\alpha_3-\alpha_2}{\alpha_2-\alpha_1} \left(\frac{\alpha_1-\alpha_3}{\alpha_2-\alpha_3}\right)^\frac{\alpha_1-\alpha_3}{\alpha_2-\alpha_1} \left(\frac{\alpha_2-\alpha_1}{\alpha_2-\alpha_3}\right) >0 \quad \text{if }  \alpha_3< \alpha_1 <\alpha_2.
\end{align*}
Then, for any $\lambda \in (0, \lambda_1),$ the function $h_\lambda$ has clearly two roots $R_1(\lambda)$ and $R_2(\lambda)$ with $0< R_1(\lambda) < t_\ast  < R_2(\lambda)$,
\begin{equation}\label{trunc:est-2}
	R_1(\lambda) \to 0  \quad\text{as }  \lambda \to 0 \quad \text{ and } \quad  h_\lambda (t)
	\begin{cases}
		>0  &\text{if }  t \in (R_1(\lambda), R_2(\lambda)), \\
		<0  &\text{if }  t \in  (0, R_1(\lambda)) \cup (R_2(\lambda), \infty).
	\end{cases}
\end{equation}
By \eqref{trunc:est-2}, we find $\lambda_2 >0$ such that for $\lambda \in (0, \lambda_2)$
\begin{align*}
	R_1(\lambda)^{\min\{p^-, (q+ \lfloor s \rfloor)^-\}} < \min\left\{\frac{1}{2 \max\{p^+, q^+\}\}}, \frac{t_\ast^{\max\{p^+, (q+ \lfloor s \rfloor)^+\}}}{2 \max\{p^+, q^+\}}\right\}.
\end{align*}
Set
\begin{equation}\label{opt-lamb}
	\lambda_\ast := \min\{\lambda_0, \lambda_1, \lambda_2\}.
\end{equation}
For each $\lambda \in (0, \lambda_\ast)$, we define the truncated functional $\tilde{\mathcal{E}}_\lambda\colon W_0^{1, \mathcal{S}}(\Omega) \to \mathbb{R}$ such that
\begin{align*}
	\tilde{\mathcal{E}}_\lambda(u) = \mathcal{E}_1(u) - \tau\left(\mathcal{E}_1(u)\right) \left[\mathcal{E}_2(u) + \lambda \mathcal{E}_3(u) \right], \quad u \in W_0^{1, \mathcal{S}}(\Omega),
\end{align*}
where $\tau \in C_c^\infty(\mathbb{R})$ such that $0 \leq \tau \leq 1$,
\begin{align*}
	\tau(t) =
	\begin{cases}
		1   & \text{if }  |t| \leq {R_1(\lambda)^{\min\{p^-, (q+ \lfloor s \rfloor)^-\}}}, \\
		0   & \text{if }  |t| \geq 2 R_1(\lambda)^{\min\{p^-, (q+ \lfloor s \rfloor)^-\}}.
	\end{cases}
\end{align*}
It is easy to see that
$\tilde{\mathcal{E}}_\lambda \in C^1(W_0^{1,\mathcal{S}}(\Omega), \mathbb{R})$ such that $\mathcal{E}_\lambda(u) \leq \tilde{\mathcal{E}}_\lambda (u)$ for $u \in W_0^{1,\mathcal{S}}(\Omega),$
\begin{equation}\label{trunc:est-3}
	\tilde{\mathcal{E}}_\lambda (u) = \mathcal{E}_1(u) \quad \text{for }  u \in W_0^{1,\mathcal{S}}(\Omega)  \text{ with }  \mathcal{E}_1(u) \geq 2 { R_1(\lambda)^{\min\{p^-, (q+ \lfloor s \rfloor)^-\}}},
\end{equation}
and
\begin{align*}
	\tilde{\mathcal{E}}_\lambda (u) = \mathcal{E}_\lambda(u) \quad \text{for }  u \in W_0^{1,\mathcal{S}}(\Omega)  \text{ with }  \mathcal{E}_1(u) \leq { R_1(\lambda)^{\min\{p^-, (q+ \lfloor s \rfloor)^-\}}}.
\end{align*}

\begin{lemma}\label{lem:trunc_reduced}
	Let $\lambda \in (0, \lambda_\ast)$.
	Then, for $u \in W_0^{1, \mathcal{S}}(\Omega)$ with $\tilde{\mathcal{E}}_\lambda(u) <0$,
	\begin{align*}
		\tilde{\mathcal{E}}_\lambda(u) = \mathcal{E}_\lambda(u) \quad \text{ and } \quad \tilde{\mathcal{E}}_\lambda'(u) = \mathcal{E}_\lambda'(u).
	\end{align*}
\end{lemma}

\begin{proof}
	Let $\lambda \in (0, \lambda_\ast)$ and $u \in W_0^{1, \mathcal{S}}(\Omega)$ with $\tilde{\mathcal{E}}_\lambda(u) <0$. First we claim that $\|\nabla u\|_{\mathcal{S}} < 1$. Suppose that $\|\nabla u\| \geq 1$, then by Proposition \ref{pro:norm-mod:relation} we have $\int_{\Omega} \mathcal{S}(x, \nabla u) \,\mathrm{d}x \geq 1$. Thus,
	\begin{align*}
		\mathcal{E}_1(u) \geq \frac{1}{\max\{p^+, q^+\}} \int_{\Omega} \mathcal{S}(x, \nabla u) \,\mathrm{d}x \geq \frac{1}{\max\{p^+, q^+\}} \geq 2 { R_1(\lambda)^{\min\{p^-, (q+ \lfloor s \rfloor)^-\}}}.
	\end{align*}
	The above inequality gives a contradiction as a virtue of \eqref{trunc:est-3} and $\tilde{\mathcal{E}}_\lambda(u) <0$. Furthermore, \eqref{trunc:est-0} implies $h_\lambda(\|\nabla u\|_{\mathcal{S}}) <0$, i.e., either $\|\nabla u\|_{\mathcal{S}} < R_1(\lambda)$ or $\|\nabla u\|_{\mathcal{S}} > R_2(\lambda) > t_\ast$. The latter inequality with Proposition \ref{pro:norm-mod:relation} and the upper bound of $R_1(\lambda)$ gives
	\begin{align*}
		\mathcal{E}_1(u)
		&\geq \frac{1}{\max\{p^+, q^+\}} \int_{\Omega} \mathcal{S}(x, \nabla u) \,\mathrm{d}x  \geq \frac{\|\nabla u\|_{\mathcal{S}}^{\max\{p^+, (q+ \lfloor s \rfloor)^+\}}}{\max\{p^+, q^+\}} \\
		& \geq {\frac{t_\ast^{\max\{p^+, (q+ \lfloor s \rfloor)^+\}}}{\max\{p^+, q^+\}}} \geq 2 R_1(\lambda)^{\min\{p^-, (q+ \lfloor s \rfloor)^-\}}.
	\end{align*}
	This is a contradiction to $\tilde{\mathcal{E}}_\lambda(u) <0$ because of \eqref{trunc:est-3}.  Thus, it must hold $\|\nabla u\|_{\mathcal{S}} < R_1(\lambda)$ and therefore by Proposition \ref{pro:norm-mod:relation}
	\begin{align*}
		\mathcal{E}_1(u) \leq \frac{1}{\min\{p^-, q^-\}} \int_{\Omega} \mathcal{S}(x, \nabla u) \,\mathrm{d}x \leq \frac{\|\nabla u\|_{\mathcal{S}}^{\min\{p^-, (q+ \lfloor s \rfloor)^-\}}}{\min\{p^-, q^-\}} \leq R_1(\lambda)^{\min\{p^-, (q+ \lfloor s \rfloor)^-\}}.
	\end{align*}
\end{proof}

In order to prove our multiplicity results, we will use some topological results introduced by Krasnosel'skii \cite{Krasnoselskii-1964}. To this end, let $X$ be a Banach space and let $\Sigma$ be the class of all closed subsets $A \subset X \setminus \{0\}$ that are symmetric with respect to the origin, that is, $u \in A$ implies $-u \in A$.
\begin{definition}
	Let $A \in \Sigma$. The Krasnosel'skii genus $\gamma(A)$ of $A$ is defined as being the least positive integer $k$ such that there is an odd mapping $\phi \in C(A, \mathbb{R}^k)$ such that $\phi(x) \neq 0$ for any $x \in A$. If $k$ does not exist, we have $\gamma(A) = \infty$. Furthermore, we set $\gamma(\emptyset) =0$.
\end{definition}

The following proposition states the main properties of the Krasnosel'skii genus, see Rabinowitz \cite{Rabinowitz-1986}.
\begin{proposition}\label{pro:Krasnoselskii}
	Let $A, B \in \Sigma$. Then, the following hold:
	\begin{enumerate}
		\item[\textnormal{(i)}]
			If there exists an odd continuous mapping from $A$ to $B$, then $\gamma(A) \leq \gamma(B)$.
		\item[\textnormal{(ii)}]
			If there is an odd homomorphism from $A$ to $B$, then $\gamma(A) = \gamma(B)$.
		\item[\textnormal{(iii)}]
			If $\gamma(B) < \infty,$ then $\gamma(\overline{A \setminus B}) \geq \gamma(A) - \gamma(B)$.
		\item[\textnormal{(iv)}]
			The $k$-dimensional sphere $S^k$ has a genus of $k+1$ by the Borsuk-Ulam Theorem.
		\item[\textnormal{(v)}]
			If $A$ is compact, then $\gamma(A) < \infty$ and there exists $\delta >0$ such that $N_\delta(A) \subset \Sigma$ and $\gamma(N_\delta(A)) = \gamma(A),$ where $N_\delta(A) = \{x \in X\colon  \operatorname{dist}(x, A) \leq \delta\}$.
	\end{enumerate}
\end{proposition}

Now, we are going to construct an appropriate mini-max sequence of negative critical values for the functional $\tilde{\mathcal{E}}_\lambda$.

\begin{lemma}\label{lem:genus_lowerbd}
	Let $\lambda \in (0, \lambda_\ast).$ Then, for each $k \in \mathbb{N}$ there exists $\varepsilon >0$ such that
	\begin{align*}
		\gamma(\tilde{\mathcal{E}}_\lambda^{-\varepsilon}) \geq k.
	\end{align*}
	where $\tilde{\mathcal{E}}_\lambda^{-\varepsilon} = \{ u \in W_0^{1, \mathcal{S}}(\Omega) \colon  \tilde{\mathcal{E}}_\lambda(u) \leq - \varepsilon\}$.
\end{lemma}

\begin{proof}
	Let $\lambda \in (0, \lambda_\ast)$, $k \in \mathbb{N}$ be fixed and $X_k$ an $k$-dimensional subspace of $W_0^{1, \mathcal{S}}(\Omega)$. Since all norms are equivalent on $X_k$, we find $\delta_k > (R_1(\lambda))^{-\min\{p^-, (q+ \lfloor s \rfloor)^-\}} >1$ such that
	\begin{equation}\label{trunc:est-4}
		\delta_k^{-1 }\|u\|_{\mathcal{S}_\star} \leq \|\nabla u\|_{\mathcal{S}} \leq \delta_k \|u\|_{\mathcal{S}_\star} \quad \text{for all }  u \in X_k.
	\end{equation}
	For any $u \in X_k$ with $\|\nabla u\|_{\mathcal{S}} \leq \delta_k^{-1} <1$, using Proposition \ref{pro:norm-mod:relation} and \eqref{trunc:est-4}, we have
	\begin{align*}
		\tilde{\mathcal{E}}_\lambda(u)
		&\leq \frac{\|\nabla u\|_{\mathcal{S}}^{\min\{p^-, (q+ \lfloor s \rfloor)^-\}}}{\min\{p^-, q^-\}} - \lambda \frac{\|u\|_{\mathcal{S}_\star}^{\max\{p_\star^+, \left(q_\star + \lceil s_\star \rceil \frac{q_\star}{q}\right)^+ \}}}{\max\{p_\star^+, q_\star^+\}}   \\
		& \leq \frac{\|\nabla u\|_{\mathcal{S}}^{\min\{p^-, (q+ \lfloor s \rfloor)^-\}}}{\min\{p^-, q^-\}} - \lambda \frac{ \left(\delta_n^{-1}\|\nabla u\|_{\mathcal{S}}\right)^{\max\{p_\star^+, \left(q_\star + \lceil s_\star \rceil \frac{q_\star}{q}\right)^+ \}}}{\max\{p_\star^+, q_\star^+\}}  \\
		 & < \|\nabla u\|_{\mathcal{S}}^{\max\{p_\star^+, \left(q_\star + \lceil s_\star \rceil \frac{q_\star}{q}\right)^+ \}} \left( \frac{\|\nabla u\|_{\mathcal{S}}^{\min\{p^-, (q+ \lfloor s \rfloor)^-\} - \max\{p_\star^+, \left(q_\star + \lceil s_\star \rceil \frac{q_\star}{q}\right)^+ \}}}{\min\{p^-, q^-\}}\right. \\
		 & \left. \qquad\qquad\qquad\qquad\qquad\qquad\qquad\qquad- \frac{\lambda \delta_n^{-\max\{p_\star^+, \left(q_\star + \lceil s_\star \rceil \frac{q_\star}{q}\right)^+ \}}}{\max\{p_\star^+, q_\star^+\}}\right).
	\end{align*}
	Now we choose $t$ such that
	\begin{align*}
		0<t < \min\left\{\delta_k^{-1}, \left(\frac{\lambda \delta_k^{-\max\{p_\star^+, \left(q_\star + \lceil s_\star \rceil \frac{q_\star}{q}\right)^+ \}}}{\max\{p_\star^+, q_\star^+\}}\right)^{\frac{1}{\min\{p^-, (q+ \lfloor s \rfloor)^-\} - \max\{p_\star^+, \left(q_\star + \lceil s_\star \rceil \frac{q_\star}{q}\right)^+ \}}} \right\}
	\end{align*}
	and let
	\begin{align*}
		\mathbb{S}_k:= \{u \in X_k\colon  \|\nabla u\|_{\mathcal{S}} = t \}.
	\end{align*}
	Clearly, $\mathbb{S}_k$ is homeomorphic to the $(k-1)$-dimensional sphere $S^{k-1}$. Hence, by Proposition \ref{pro:Krasnoselskii} \textnormal{(iv)}, we know that $\gamma(\mathbb{S}_k) =k$. With the above choice of $t$, we obtain
	\begin{equation}\label{trunc:est-5}
		\tilde{\mathcal{E}}_\lambda (u) \leq - \varepsilon <0 \quad \text{for all }  u \in \mathbb{S}_k
	\end{equation}
	where
	\begin{align*}
		\varepsilon = - t^{\max\{p_\star^+, \left(q_\star + \lceil s_\star \rceil \frac{q_\star}{q}\right)^+ \}} \left( \frac{t^{\min\{p^-, (q+ \lfloor s \rfloor)^-\} - \max\{p_\star^+, \left(q_\star + \lceil s_\star \rceil \frac{q_\star}{q}\right)^+ \}}}{\min\{p^-, q^-\}} - \frac{\lambda \delta_k^{-\max\{p_\star^+, \left(q_\star + \lceil s_\star \rceil \frac{q_\star}{q}\right)^+ \}}}{\max\{p_\star^+, q_\star^+\}}\right).
	\end{align*}
	Finally, by using \eqref{trunc:est-5} and Proposition \ref{pro:Krasnoselskii} \textnormal{(i)}, we get
	\begin{align*}
		\mathbb{S}_k \subset \tilde{\mathcal{E}}_\lambda^{-\varepsilon} \quad \text{and} \quad \gamma(\tilde{\mathcal{E}}_\lambda^{-\varepsilon}) \geq \gamma(\mathbb{S}_k) =k.
	\end{align*}
\end{proof}

Now, we define the following sets, for any $k \in \mathbb{N},$
\begin{align*}
	\Sigma_k &= \{ A \subset W_0^{1, \mathcal{S}}(\Omega) \setminus \{0\}\colon  A  \text{ is closed, } A = -A  \text{ and }  \gamma(A) \geq k\},\\
	K_c&= \{ u \in W_0^{1, \mathcal{S}}(\Omega) \setminus \{0\}\colon  \tilde{\mathcal{E}}_\lambda'(u) =0  \text{ and }  \tilde{\mathcal{E}}_\lambda(u) =c\}
\end{align*}
and the number
\begin{align*}
	c_k = \inf_{A \in \Sigma_k} \sup_{u \in A} \tilde{\mathcal{E}}_\lambda(u).
\end{align*}
It is easy to see that $c_k \leq c_{k+1}$ for any $n \in \mathbb{N}$.

\begin{lemma}\label{lem:neg_crit}
	For each $k \in \mathbb{N}$, it holds that
	\begin{align*}
		- \infty < c_k < 0.
	\end{align*}
\end{lemma}

\begin{proof}
	Let $\lambda \in (0,\lambda^\ast)$ and $k$ be fixed. From Lemma \ref{lem:genus_lowerbd}, we know that there exists $\varepsilon >0$ such that $\gamma(\tilde{\mathcal{E}}_\lambda^{-\varepsilon}) \geq k$. Furthermore, because $\tilde{\mathcal{E}}_\lambda$ is even and continuous, we know that $\tilde{\mathcal{E}}_\lambda^{-\varepsilon} \in \Sigma_k$. From $\tilde{\mathcal{E}}_\lambda(0) =0$, we have $0 \in \tilde{\mathcal{E}}_\lambda^{-\varepsilon}$. Since $\sup_{u \in \tilde{\mathcal{E}}_\lambda^{-\varepsilon}} \tilde{\mathcal{E}}_\lambda(u) \leq -\varepsilon$ and $\tilde{\mathcal{E}}_\lambda$ is bounded from below, we obtain the assertion.
\end{proof}

\begin{lemma}\label{lem:genus_counting}
	Let $\lambda \in (0, \lambda^\ast)$ and $k \in \mathbb{N}$. If $c=c_k = c_{k+1} = \cdots = c_{k+m}$ for some $m \in \mathbb{N}$, then $K_c \in \Pi$ and
	\begin{align*}
		\gamma(K_c) \geq m+1.
	\end{align*}
\end{lemma}

\begin{proof}
	The proof follows by repeating the same arguments as in Lemma 3.6 by Farkas--Fiscella--Winkert \cite{Farkas-Fiscella-Winkert-2022}.
\end{proof}

\begin{theorem}\label{theo-2}
	Let \eqref{main:assump}, \eqref{main:assump-1}, \eqref{main:assump-1-2}, \eqref{main:assump-4}, and \eqref{sandwich:cond-2-sub} be satisfied. Then, for $\lambda \in (0, \lambda^\ast)$ and $\Lambda =1$, problem \eqref{main:prob} admits infinitely many weak solutions with negative energy values. Moreover, if $u_\lambda$ is a solution of \eqref{main:prob} corresponding to $\lambda$, then
	\begin{align*}
		\lim_{\lambda \to 0^+} \|u_\lambda \|_{1, \mathcal{S}} =0.
	\end{align*}
\end{theorem}

\begin{proof}
	Let $\lambda \in (0, \lambda^\ast)$, where $\lambda^\ast$ is given in \eqref{opt-lamb}. By Lemma \ref{lem:neg_crit}, $\tilde{\mathcal{E}}_\lambda$ admits a sequence $\{u_k\}_{k \in \mathbb{N}}$ of critical points with $\tilde{\mathcal{E}}_\lambda(u_k) <0$. We consider two situations. If $-\infty < c_1 < c_2 < \dots < c_n < c_{n+1} \dots,$ then $\tilde{\mathcal{E}}_\lambda$ admits infinitely many critical values. If there exists $k, l \in \mathbb{N}$ such that $c_k = c_{k+1} = \dots = c_{k+l},$ then by Lemma \ref{lem:genus_counting} and Rabinowitz \cite[Remark 7.3]{Rabinowitz-1986}, $\gamma(K_c) \geq l+1$ and the set $K_c$ has infinitely many points, which are infinitely many critical values of $\tilde{\mathcal{E}}_\lambda$.

	By Lemma \ref{lem:trunc_reduced}, $\{u_k\}_{k \in \mathbb{N}}$ are the critical points of $\mathcal{E}_\lambda$ and hence weak solutions of \eqref{main:prob}. Now, let $u_\lambda$ be a solution of \eqref{main:prob} corresponding to $\lambda$. Then, again by Lemma \ref{lem:trunc_reduced} and \eqref{trunc:est-2}
	\begin{align*}
		\frac{1}{\max\{p^+, q^+\}} \int_{\Omega} \mathcal{S}(x, \nabla u_\lambda) \,\mathrm{d}x \leq \int_{\Omega} \mathcal{M}(x, \nabla u_\lambda) \,\mathrm{d}x < { R_1(\lambda)^{\min\{p^-, (q+ \lfloor s \rfloor)^-\}}} \to 0  \quad\text{as }  \lambda \to 0^+.
	\end{align*}
	Finally, from Proposition \ref{Prop:oneHlogModularNorm}, we get the required claim and the proof is complete.
\end{proof}

\section*{Acknowledgments}

The first author acknowledges the support of the Start-up Research Grant (SRG) SRG/2023/000308, Anusandhan National Research Foundation (ANRF erstwhile SERB), India, and Seed grant IIT(BHU)/ DMS/2023-24/493. The second author was funded by the Deutsche Forschungsgemeinschaft (DFG, German Research Foundation) under Germany's Excellence Strategy -
The Berlin Mathematics Research Center MATH+ and the Berlin Mathematical School (BMS) (EXC-2046/1, project ID: 390685689).

\end{document}